\documentclass[11pt, reqno]{amsart}

\usepackage[T1]{fontenc}
\usepackage[latin1]{inputenc}

\usepackage[a4paper]{geometry}

\usepackage{amssymb}
\usepackage{amsmath}
\usepackage{amsthm}
\usepackage[all, 2cell, poly]{xypic}
\SelectTips{cm}{10}

\usepackage{version}

\include{macros}

\author{Dimitri Ara}

\address{Dimitri Ara, Institut de Mathématiques de Jussieu, Université Paris
Diderot~-- Paris 7, Case 7012, Bâtiment Chevaleret, 75205 Paris Cedex 13,
France}
\email{ara@math.jussieu.fr}
\urladdr{http://people.math.jussieu.fr/~ara/}

 
\title[On the homotopy theory of Grothendieck $\infty$-groupoids]{On the
homotopy theory of\\ Grothendieck $\infty$-groupoids}

\let\nbd\nobreakdash

\begin{document}

\begin{abstract}
We present a slight variation on a notion of weak \oo-groupoid introduced by
Grothendieck in \emph{Pursuing Stacks} and we study the homotopy theory of
these \oo-groupoids. We prove that the obvious definition for homotopy
groups of Grothendieck \oo-groupoids does not depend on any choice. This
allows us to give equivalent characterizations of weak equivalences of
Grothendieck \oo-groupoids, generalizing a well-known result for strict
\oo-groupoids. On the other hand, given a model category~$\M$ in which every
object is fibrant, we construct, following Grothendieck, a fundamental
\oo-groupoid functor $\Pi_\infty$ from $\M$ to the category of
Grothendieck \oo-groupoids. We show that if $X$ is an object of $\M$, then
the homotopy groups of~$\Pi_\infty(X)$ and of $X$ are canonically
isomorphic. We deduce that the functor $\Pi_\infty$ respects weak
equivalences.
\end{abstract}

\maketitle

\section*{Introduction}

The notion of Grothendieck \oo-groupoid has been introduced in 1983 by
Grothendieck in a famous letter to Quillen that became the starting point of
the equally famous text \emph{Pursuing Stacks} (\cite{GrothPS}). For more
than twenty years, people thought that Grothendieck's definition was only
informal. Maltsiniotis realized in 2006 that this definition is perfectly
precise. This original definition is explained in \cite{MaltsiGr}.
Maltsiniotis also suggested a simplification of the definition in
\cite{MaltsiCat}. All the texts written since then (namely \cite{AraThesis}
and \cite{MaltsiGrCat}) also use this simplification. In this article, we
use a slight variation on this simplification.

The main motivation of Grothendieck is to generalize to higher dimensions the
classical fact that the homotopy $1$-type of a topological space is
classified by its fundamental groupoid up to equivalence. Given a space $X$,
Grothendieck considers the \oo-graph whose objects are points of $X$, whose
$1$-arrows are paths, whose $2$-arrows are relative homotopies between paths,
whose $3$-arrows are relative homotopies between relative homotopies between
paths, and so on. This \oo-graph seems to bear an algebraic structure: for
instance, one can compose arrows, though in a non-canonical way; these
compositions are associative up to non-canonical higher arrows; these
higher arrows satisfy higher coherences, again in a non-canonical way; etc.
Grothendieck suggests that this \oo-graph should be equipped with the
structure of an \oo-groupoid (a notion to be defined) and that this
\oo-groupoid, up to equivalence, should classify the homotopy type of $X$.
This is an imprecise statement of Grothendieck's conjecture.

The question now is how to define this structure of \oo-groupoid. There
exists definitions of $n$-groupoids for small $n$'s, obtained by giving
explicit generators for coherences. But even for $n = 3$, the standard
definition (see \cite{GPSTricat}) is almost intractable. One has to find
another kind of definition.

Here is how Grothendieck proceeds. His idea is to define a category $C$
encoding the algebraic structure of an \oo-groupoid. The category of
\oo-groupoids will then be defined as the full subcategory of the category of
presheaves on $C$ satisfying some left exactness condition (i.e., some 
higher Segal condition). This category $C$ will not be unique. This reflects
the fact that there is no canonical choice of generators for higher
coherences. A category $C$ encoding the theory of \oo-groupoids will be
called a coherator. But how to define a coherator?

Grothendieck's main insight is that a very simple principle can be used to
generate inductively higher coherences, hence giving a definition for
coherators. We will refer to this principle as the ``coherences generating
principle''. Here is how it goes. Suppose $G$ is a (weak) \oo-groupoid,
whatever it means. Let $X$ be a globular pasting scheme decorated by arrows
of $G$. For instance, $X$ might be
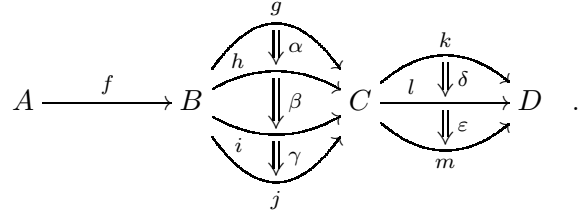
\[
\UseAllTwocells
\xymatrix@C=4pc@H=5pc{
A
\ar[r]^f
&
B
\ar@/^2.5pc/[r]|{\vphantom{X}}="g"^g
\ar@/^1pc/[r]|{\vphantom{X}}="h"^(0.3)h
\ar@/_1pc/[r]|{\vphantom{X}}="i"_(0.3)i
\ar@/_2.5pc/[r]|{\vphantom{X}}="j"_j
\ar@{=>}"g";"h"^{\,\alpha}
\ar@{=>}"h";"i"^{\,\beta}
\ar@{=>}"i";"j"^{\,\gamma}
&
C
\ar@/^1.5pc/[r]|{\vphantom{X}}="k"^k
\ar[r]|{\vphantom{X}}="l"^(0.3)l
\ar@/_1.5pc/[r]|{\vphantom{X}}="m"_m
\ar@{=>}"k";"l"^{\,\delta}
\ar@{=>}"l";"m"^{\,\epsilon}
&
D\pbox{.}
&
}
\]
Suppose that from such an $X$, one can build, using operations of the algebraic
structure of \oo-groupoids, two parallel $n$-arrows $\Lambda$ and
$\Lambda'$. Then the coherences generating principle states that there
should exist an operation, in the algebraic structure of \oo-groupoid,
producing from $X$ an $(n+1)$-arrow going from $\Lambda$ to $\Lambda'$.

For instance, let $X$ be as above and consider
\[
\Lambda =
\big(
(\epsilon \circ \delta)
\comp
((\gamma \circ \beta) \circ \alpha)
\big) \comp \id{f} 
\quad\text{and}\quad
\Lambda' =
\big(
(\epsilon \comp \gamma)
\circ
(\delta \comp (\beta \circ \alpha))
\big) \comp \id{f},
\]
where we have denoted by $\circ$ the vertical composition of $2$-arrows and
by $\comp$ the horizontal composition of $2$-arrows. These two arrows are
parallel: their source is $(kg)f$ and their target is $(mj)f$. Hence, the
coherences generating principle says that there should exist an operation, in
the algebraic structure of \oo-groupoid, producing a $3$-arrow 
\[
\big(
(\epsilon \circ \delta)
\comp
((\gamma \circ \beta) \circ \alpha)
\big) \comp \id{f} 
\Rrightarrow
\big(
(\epsilon \comp \gamma)
\circ
(\delta \comp (\beta \circ \alpha))
\big) \comp \id{f}
\]
from $f$, $\alpha$, $\beta$, $\gamma$, $\delta$, $\epsilon$ fitting in a
diagram as above.

If $C$ is a category encoding an algebraic theory satisfying the coherences
generating principle, then $C$ describes an algebraic structure where all
operations of \oo-groupoids exist, but possibly in a too strict way. A
coherator will be defined as a category encoding an algebraic theory freely
satisfying the coherences generating principle, that is, a theory obtained
by freely adding operations, using the principle, from the operations source
and target.

Now that we have given an idea of Grothendieck's definition, let us come
back to his original motivation: the classification of homotopy types.
Grothendieck shows that if $X$ is a topological space, then the \oo-graph
associated to $X$ can be endowed with the structure of an \oo-groupoid. More
precisely, he constructs a fundamental \oo-groupoid functor~$\Pi_\infty$
from the category of topological spaces to the category $\wgpd$ of
Grothendieck \oo-groupoids. He conjectures that this functor induces an
equivalence of categories between the homotopy category of topological
spaces $\Hot$ and an appropriate localization of $\wgpd$.  This conjecture
is still not proved.  At the end of this article, we give a more precise
statement of the conjecture.

A special feature of Grothendieck \oo-groupoids is that they are not
defined as \oo-cate\-gories satisfying some invertibility conditions.
Nevertheless, Maltsiniotis realized that a variation on Grothendieck's
definition gives rise to a new notion of weak \oo-category. This notion of
\oo-category is closely related to the notion of weak \oo-category
introduced by Batanin in \cite{BataninWCat}. The precise relation between
Grothendieck-Maltsiniotis \oo-categories and Batanin \oo-categories is
studied in the PhD thesis \cite{AraThesis} of the author.

\bigskip

This article is about the homotopy theory of Grothendieck \oo-groupoids. By
homotopy theory of \oo-groupoids, we mean the study of the category of
\oo-groupoids endowed with weak equivalences of \oo-groupoids. Our
contribution to the subject, in addition to foundational aspects, is of
three kinds.

First, we propose a slight modification of the definition of Grothendieck
\oo-groupoids, whose purpose is to make canonical the inclusion functor from
strict \oo-groupoids to Grothendieck \oo-groupoids. This modification
takes the form of an additional condition in the definition of an admissible
pair. The importance of this modification will be made clear in the
forthcoming paper \cite{AraStrWeak}.

Second, we prove several foundational results on homotopy groups and weak
equivalences of Grothendieck \oo-groupoids. If $G$ is an \oo-groupoid and
$x$ is an object of $G$, the $n$-th homotopy group $\pi_n(G, x)$ is defined
as the group of $n$-arrows, up to $(n+1)$-arrows, whose source and target
are the iterated unit of $x$ in dimension $n-1$. This definition depends a
priori on several choices. We show that $\pi_n(G, x)$ does not depend on
these choices.  The heart of the proof is the so-called division lemma. We
also deduce from this lemma that a $1$-arrow $u : x \to y$ of $G$ induces an
isomorphism from $\pi_n(G, x)$ to $\pi_n(G, y)$. Finally, we give four
equivalent characterizations of weak equivalences of Grothendieck
\oo-groupoids, generalizing a theorem of Simpson on strict $n$-categories
with weak inverses (see Theorem 2.1.III of \cite{Simpson} or Definition
2.2.3 of \cite{SimpsonHTHC}).

Third, given a model category $\M$ in which every object is fibrant, we
construct, following Grothendieck, a fundamental \oo-groupoid functor
$\Pi_\infty : \M \to \wgpd$, which depends on liftings in $\M$. If $X$ is an
object of $\M$, we define a model categorical notion of homotopy groups of
$X$. We show that the homotopy groups of~$\Pi_\infty(X)$ are canonically
isomorphic to the homotopy groups of $X$. We deduce that the homotopy groups
of $\Pi_\infty(X)$ depend only on $X$ and that the functor $\Pi_\infty$
respects weak equivalences. In particular, applying this result to the
category of topological spaces, we give a precise statement of
Grothendieck's conjecture.

\bigskip

Our paper is organized as follows. In the first section, we introduce the
globular language, in which the notion of Grothendieck \oo-groupoid will be
phrased, and in particular the notion of globular extension. Section 2 is
dedicated to the definition of Grothendieck \oo-groupoids. We define in
particular contractible globular extensions and coherators. In Section 3, we
explain how to construct structural maps out of a Grothendieck \oo-groupoid.
In particular, we construct enough structural maps to show that a
Grothendieck \oo-groupoid can be truncated to a bigroupoid. In Section 4, we
study the notions of homotopy groups and weak equivalences of \oo-groupoids.
We show that the homotopy groups are well-defined and we give four
equivalent characterizations of weak equivalences. In Section 5, we explain
Grothendieck's construction of the fundamental \oo-groupoid functor
$\Pi_\infty : \M \to \wgpd$. We interpret this construction in terms of
Reedy model structures. In Section 6, we recall Quillen's $\pi_1$ theory
introduced in \cite{Quillen}. We give an alternative formulation in terms of
slice categories. In Section 7, we work in a model category $\M$ in which
every object is fibrant. We define a notion of based objects of $\M$ and,
using a loop space construction, we define a theory of homotopy groups for
these based objects. One difference with Quillen's original theory is that
our category $\M$ has no zero object. We show that with our definitions,
$\pi_1(X, x)$ is canonically isomorphic to $\pi_0(\Omega_x X)$. Finally, in
Section 8, we compare the homotopy groups of $\Pi_\infty(X)$ and the
homotopy groups of $X$, where $X$ is an object of a model category in which
every object is fibrant. We finish by a precise statement of Grothendieck's
conjecture.

\bigskip

If $C$ is a category, we will denote by $C^\op$ the opposite category
and by $\pref{C}$ the category of presheaves on $C$. If
\[
\xymatrix@C=1pc@R=1pc{
X_1 \ar[dr]_{f_1} & & X_2 \ar[dl]^{g_1} \ar[dr]_{f_2} & &  \cdots & & X_n
\ar[dl]^{g_{n-1}} \\
& Y_1 & & Y_2 & \cdots & Y_{n-1}
}
\]
is a diagram in $C$, we will denote by 
\[ (X_1, f_1) \times_{Y_1} (g_1, X_2, f_2) \times_{Y_2} \dots
\times_{Y_{n-1}} (g_{n-1}, X_n) \]
its limit. Dually, we will denote by
\[ (X_1, f_1) \amalg_{Y_1} (g_1, X_2, f_2) \amalg_{Y_2} \dots
\amalg_{Y_{n-1}} (g_{n-1}, X_n) \]
the colimit of the corresponding diagram in $C^\op$.

\section{The globular language}

\begin{tparagr}{The globe category}
We will denote by $\G$ the \ndef{globe category}, that is, the category
generated by the graph
\[
\xymatrix{
\Dn{0} \ar@<.6ex>[r]^-{\Ths{1}} \ar@<-.6ex>[r]_-{\Tht{1}} &
\Dn{1} \ar@<.6ex>[r]^-{\Ths{2}} \ar@<-.6ex>[r]_-{\Tht{2}} &
\cdots \ar@<.6ex>[r]^-{\Ths{i-1}} \ar@<-.6ex>[r]_-{\Tht{i-1}} &
\Dn{i-1} \ar@<.6ex>[r]^-{\Ths{i}} \ar@<-.6ex>[r]_-{\Tht{i}} &
\Dn{i} \ar@<.6ex>[r]^-{\Ths{i+1}} \ar@<-.6ex>[r]_-{\Tht{i+1}} &
\dots
}
\]
under the coglobular relations
\[\Ths{i+1}\Ths{i} = \Tht{i+1}\Ths{i}\quad\text{and}\quad\Ths{i+1}\Tht{i} =
\Tht{i+1}\Tht{i}, \qquad i \ge 1.\]
For $i \ge j \ge 0$, we will denote by $\Ths[j]{i}$ and $\Tht[j]{i}$ the
morphisms from $\Dn{j}$ to $\Dn{i}$ defined by
\[\Ths[j]{i} = \Ths{i}\cdots\Ths{j+2}\Ths{j+1}\quad\text{and}\quad
  \Tht[j]{i} = \Tht{i}\cdots\Tht{j+2}\Tht{j+1}.\]
\end{tparagr}

\begin{tparagr}{Globular sets}
The category of \ndef{globular sets} or \ndef{\oo-graphs} is the category
$\pref{\G}$ of presheaves on $\G$. The data of a globular set $X$ amounts
to the data of a diagram of sets
\[
\xymatrix{
\cdots \ar@<.6ex>[r]^-{\Gls{i+1}} \ar@<-.6ex>[r]_-{\Glt{i+1}} &
X_{i} \ar@<.6ex>[r]^-{\Gls{i}} \ar@<-.6ex>[r]_-{\Glt{i}} &
X_{i-1} \ar@<.6ex>[r]^-{\Gls{i-1}} \ar@<-.6ex>[r]_-{\Glt{i-1}} &
\cdots \ar@<.6ex>[r]^-{\Gls{2}} \ar@<-.6ex>[r]_-{\Glt{2}} &
X_1 \ar@<.6ex>[r]^-{\Gls{1}} \ar@<-.6ex>[r]_-{\Glt{1}} &
X_0
}
\]
satisfying the globular relations
\[\Gls{i}\Gls{i+1} = \Gls{i}\Glt{i+1}\quad\text{and}\quad\Glt{i}\Gls{i+1} =
\Glt{i}\Glt{i+1}, \qquad i \ge 1.\]
For $i \ge j \ge 0$, we will denote by $\Gls[j]{i}$ and $\Glt[j]{i}$ the
maps from $X_i$ to $X_j$ defined by
\[\Gls[j]{i} = \Gls{j+1}\cdots\Gls{i-1}\Gls{i}\quad\text{and}\quad
  \Glt[j]{i} = \Glt{j+1}\cdots\Glt{i-1}\Glt{i}.\]

If $X$ is a globular set, we will call $X_0$ the set of \ndef{objects} of
$X$ and, for $i \ge 0$, $X_i$ the set of \ndef{$i$-arrows}. If $u$ is an
$i$-arrow of $X$ for $i \ge 1$, $\Gls{i}(u)$ (resp.~$\Glt{i}(u)$)
will be called the \ndef{source} (resp.~the \ndef{target}) of $u$. We
will often denote an arrow $u$ of $X$ whose source is $x$ and whose target
is $y$ by $u : x \to y$.
\end{tparagr}

\begin{tparagr}{Globular sums}
Let $n$ be a positive integer. A \ndef{table of dimensions} of width $n$ is
the data of integers $i_1, \dots, i_n, i'_1, \dots, i'_{n-1}$ such that
\[ i_k > i'_k\quad\text{and}\quad i_{k+1} > i'_k, \qquad 1 \le k \le n - 1. \]
We will denote such a table of dimensions by
\[
\tabdim.
\]
The \ndef{dimension} of such a table is the greatest integer appearing in
the table.

Let $(C, F)$ be a category under $\G$, that is, a category $C$ endowed with a functor
$F : \G \to C$. We will denote in the same way the objects and morphisms
of $\G$ and their image by the functor $F$. Let
\[ T = \tabdim \]
be a table of dimensions. The \ndef{globular sum} in $C$ associated to $T$
(if it exists) is the iterated amalgamated sum
\[ (\Dn{i_1}, \Ths[i'_1]{i_1}) \amalgd{i'_1} (\Tht[i'_1]{i_2}, \Dn{i_2},
\Ths[i'_2]{i_2}) \amalgd{i'_2} \dots
\amalgd{i'_{n-1}} (\Tht[i'_{n-1}]{i_n}, \Dn{i_n}) \]
in $C$, that is, the colimit of the diagram
\[
\xymatrix@R=.2pc@C=1pc{
\Dn{i_1} &  & \Dn{i_2} &  & \Dn{i_3} &        & \Dn{i_{n - 1}} & & \Dn{i_{n}} \\
  &  &   &  &   & \cdots &     & & \\
  & \Dn{i'_1}
  \ar[uul]^{\Ths[i'_1]{i_1}}
  \ar[uur]_{\negthickspace \Tht[i'_1]{i_2}}  &   &
  \Dn{i'_2}' 
  \ar[uul]^{\Ths[i'_2]{i_2}\negthickspace}
  \ar[uur]_{\Tht[i'_2]{i_3}}  &  &  & &
\Dn{i'_{n-1}} \ar[uul]^{\Ths[i'_{n-1}]{i_{n-1}}}  \ar[uur]_{\Tht[i'_{n-1}]{i_n}}
& &
}
\]
in $C$.
We will denote it briefly by
\[
\Dn{i_1} \amalgd{i'_1} \Dn{i_2} \amalgd{i'_2} \dots
\amalgd{i'_{n-1}} \Dn{i_n}.
\]
\end{tparagr}

\begin{tparagr}{Globular extensions}
A category $C$ under $\G$ is said to be a \ndef{globular extension} if for
every table of dimensions $T$ (of any width), the globular sum associated to
$T$ exists in $C$. 

Let $C$ and $D$ be two globular extensions. A \ndef{morphism of globular
extensions} from $C$ to $D$ is a functor from $C$ to $D$ under $\G$
(that is, such that the triangle
\[
\xymatrix@C=1pc@R=1pc{
& \G \ar[ld] \ar[rd] \\
C \ar[rr] & & D
}
\]
commutes) which respects globular sums. We will also call such a functor a
\ndef{globular functor}. 
\end{tparagr}

\begin{exs}\label{exs:glext}
\
\begin{myenumerate}
\item 
If $C$ is a cocomplete category and $F : \G \to C$ is any functor, then $(C,
F)$ is a globular extension.

\item 
Let $\Top$ be the category of topological spaces. We define a functor $R :
\G \to \Top$ in the following way. For $i \ge 0$, the object $\Dn{i}$ is
sent by $R$ to the $i$-dimensional ball
\[ \Dtop{i} = \{x \in \mathbb{R}^i ; ||x|| \le 1\}. \]
For $i \ge 1$, the morphisms $\Ths{i}$ and $\Tht{i}$ are sent by $R$
respectively to $\Tops{i}$ and $\Topt{i}$ defined by
\[ \Tops{i}(x) = (x, \sqrt{1 - \Vert x\Vert^2}) \quad\text{and}\quad
  \Topt{i}(x) = (x, -\sqrt{1 - \Vert x\Vert^2}),
  \qquad x \in \Dtop{i-1}.
\]
These morphisms are the inclusions of the two hemispheres of $\Dtop{i}$. One
instantly checks that these maps satisfy the coglobular relations and our
functor $R$ is thus well-defined. Note that this functor is faithful. The
category $\Top$ being cocomplete, $(\Top, R)$ is a globular extension. In
what follows, $\Top$ will always be endowed with this globular extension
structure.

\item 
Let $h : \G \to \pref{\G}$ be the Yoneda functor. Then $(\pref{\G}, h)$ is a
globular extension. The globular sums for this globular extension are the
globular sets $T^\ast$ (where $T$ is a finite planar rooted tree) introduced
by Batanin in \cite{BataninWCat}.

\item
Let $\Thz$ be the full category of $\pref{\G}$ whose objects consist of a
choice of a globular sum (which is only defined up to isomorphism) for every
table of dimensions. The category~$\Thz$ is obviously endowed with the
structure of a globular extension. This category is canonically
isomorphic to the category $\Thz$ defined in terms of finite planar rooted
trees by Berger in \cite{BergerNerve}.

We will see that the globular extension $\Thz$ is the initial globular
extension in some $2$-categorical sense (Proposition \ref{prop:pu_Thz}).
See also Proposition 3.2 of \cite{AraThtld} for a more abstract point of
view on $\Thz$.

\item
Let $\Theta$ be the full subcategory of the category of strict \oo-categories
whose objects are free strict \oo-categories on objects of $\Thz$.
The category $\Theta$ is canonically endowed with the structure of a
globular extension. This category is canonically isomorphic to the cell
category introduced by Joyal in \cite{JoyalTheta}, as was proved
independently by Makkai and Zawadowski in \cite{MakZaw} and by Berger in
\cite{BergerNerve}. Alternative definitions of $\Theta$ are given in
\cite{BergerNerve} and \cite{BergerWreath}. See also Proposition 3.11 of
\cite{AraThtld} for a definition of $\Theta$ by universal property.

\item
Let $\Thtld$ be the full subcategory of the category of strict \oo-groupoids
whose objects are free strict \oo-groupoids on objects of $\Thz$. The
category $\Thtld$ is canonically endowed with the structure of a globular
extension. The category $\Thtld$ can be thought of as a groupoidal analogue
to Joyal's cell category. See Proposition 3.18 of \cite{AraThtld} for a
definition of $\Thtld$ by universal property.
\end{myenumerate}
\end{exs}

\begin{prop}\label{prop:pu_Thz}
The globular extension $\Thz$ has the following universal property: for every
globular extension $(C, F)$, there exists a globular functor $F_0 : \Thz \to
C$, unique up to a unique natural transformation, such that the triangle
\[
\xymatrix@C=2.5pc@R=1.5pc{
\Thz \ar[dr]^{F_0} \\
\G \ar[r]_F \ar[u] & C \\
}
\]
commutes.
\end{prop}

\begin{proof}
See Proposition 3.2 and the next paragraph of \cite{AraThtld}.
\end{proof}

\begin{rem}
If $(C, F)$ is a globular extension, an lifting $F_0 : \Thz \to C$, as in the
above universal property, amounts to the choice of a globular sum in $C$ for
every table of dimensions.
\end{rem}

\begin{tparagr}{Globular presheaves}
Let $C$ be a globular extension. A \ndef{globular presheaf} on $C$ or
\ndef{model} of $C$ is a presheaf~$X : C^\op \to \Set$ on $C$ such that the
functor $X^\op : C \to \Set^\op$ respects globular sums,
i.e., such that for every table of dimensions
\[ T = \tabdim, \]
the canonical map
\[
X(\Dn{i_1} \amalgd{i'_1} \dots \amalgd{i'_{n-1}} \Dn{i_n}) \to
X_{i_1} \times_{X_{i'_1}} \dots \times_{X_{i'_{n-1}}} X_{i_n}
\]
is a bijection. We will denote by $\Mod{C}$ the full subcategory of the
category $\pref{C}$ of presheaves on $C$ whose objects are globular
presheaves.

The canonical functor $\G \to C$ induces a functor $\pref{C} \to \pref{\G}$
which restricts to a functor $\Mod{C} \to \pref{\G}$. If $X$ is a globular
presheaf on $C$, the image of $X$ by this functor will be called the
\ndef{underlying globular set} of $X$. We will often implicitly apply the
underlying globular set functor to transfer notation and terminology from
globular sets to globular presheaves.  For instance, we will denote
$X(\Dn{i})$ by $X_i$ and we will call this set the set of $i$-arrows of $X$.
\end{tparagr}

\begin{exs}
\ 
\nopagebreak
\begin{myenumerate}
\item 
The category of globular presheaves on $\Thz$ is canonically equivalent to
the category of globular sets. More precisely, the composition
\[ \Mod{\Thz} \to \pref{\Thz} \xrightarrow{i^\ast} \pref{\G}, \]
where $i^\ast$ denotes the restriction functor induced by the canonical
functor $i : \G \to \Thz$, is an equivalence of categories. See Lemma 1.6 of
\cite{BergerNerve} or Proposition 3.5 of \cite{AraThtld}.

\item 
The category of globular presheaves on $\Theta$ is canonically equivalent to
the category of strict \oo-categories. See Theorem 1.12 of
\cite{BergerNerve} or Proposition 3.14 of \cite{AraThtld}.

\item 
The category of globular presheaves on $\Thtld$ is canonically equivalent to
the category of strict \oo-groupoids. See Proposition 3.21 of
\cite{AraThtld}.
\end{myenumerate}
\end{exs}

\begin{tparagr}{Globular extensions under $\Thz$}
A \ndef{globular extension under $\Thz$} is a category $C$ endowed with a
functor $\Thz \to C$ such that $(C, \G \to \Thz \to C)$ is a globular
extension.  If~$C$ is a globular extension under $\Thz$, the globular sum
associated to a table of dimensions is uniquely defined. A \ndef{morphism of
globular extensions under $\Thz$} is a functor under $\Thz$ between globular
extensions under~$\Thz$. Such a functor automatically respects globular
sums.
\end{tparagr}

\begin{prop}
Let $C$ be a category under $\Thz$. There exists a globular
extension~$\glcomp{C}$ under $\Thz$, endowed with a functor $C \to
\glcomp{C}$ under $\Thz$ having the following universal property:
for every globular extension $D$ under $\Thz$, endowed with a
functor $C \to D$ under $\Thz$, there exists a unique functor
$\glcomp{C} \to D$ such that the triangle
\[
\xymatrix@R=.2pc{
& \glcomp{C} \ar[dd] \\
C \ar[ur] \ar[dr] \\
& D
}
\]
commutes.
\end{prop}

\begin{proof}
This is a special case of a standard categorical construction (see
Proposition 3 of \cite{Ehresmann}). See also Section 2.6 of \cite{AraThesis}
and paragraph 3.10 of \cite{MaltsiGrCat} for this particular case.
\end{proof}

\begin{tparagr}{Globular completion}
If $C$ is a category under $\Thz$, the globular extension $\glcomp{C}$ of
the previous proposition (which is unique up to a unique isomorphism) will
be called the \ndef{globular completion} of~$C$. Note that the functor $C
\to \glcomp{C}$ is bijective on objects.
\end{tparagr}

\begin{rem}
The theory of globular extensions is in some sense generated by
the category~$\G$ and the diagrams in $\G$ describing globular sums. More
precisely, starting from a category $I$ and a set $D$ of small diagrams in
$I$, there is an obvious generalization of the theory of globular extensions
to a theory of $(I, D)$-extensions. When $I$ is the terminal category and
$D$ is the set of diagrams describing finite sums, we obtain (up to a
variance issue) the theory of Lawvere theories. In this general setting, the
category~$\Theta_0$ (resp.~the globular extensions $C$ under $\Thz$ such
that $\Thz \to C$ is bijective on objects) plays the same role as a skeleton
of the category of finite sets (resp.~as Lawvere theories).
\end{rem}

\section{Grothendieck $\infty$-groupoids}

\begin{tparagr}{Globularly parallel arrows and liftings}
Let $C$ be a globular extension. If $f : \Dn{n} \to X$ is a morphism of $C$
whose source is a $\Dn{n}$, $n \ge 1$, then the \ndef{globular source}
(resp.~the \ndef{globular target}) of $f$ is the morphism $f\Ths{n} :
\Dn{n-1} \to X$ (resp.~$f\Tht{n} : \Dn{n-1} \to X$).

If $f, g : \Dn{n} \to X$ are two morphisms of $C$ whose source is a
$\Dn{n}$, $n \ge 0$, we will say that $f$ and $g$ are \ndef{globularly
parallel} if, either $n = 0$, or $n \ge 1$ and $f, g$ have the same globular
source and globular target.

Let now $(f, g) : \Dn{n} \to X$ be a pair of morphisms of $C$.
A \ndef{lifting} of the pair $(f, g)$ is a morphism $h : \Dn{n+1} \to X$
whose globular source is $f$ and whose globular target is $g$, that is, such
that the inner and the outer triangles of the diagram
\[
\xymatrix@C=3pc{
\Dn{n+1} \ar[dr]^h \\
\Dn{n} \ar@<.5ex>[u]^(.4){\Tht{n+1}} \ar@<-.5ex>[u]_(.4){\Ths{n+1}} \ar@<.5ex>[r]^f \ar@<-.5ex>[r]_g & X \\
}
\]
commute. The existence of such a lifting obviously implies that $f$
and $g$ are globularly parallel.
\end{tparagr}

\begin{ex}\label{ex:Toplift}
Let $C = \Top$. Two maps $f, g : \Dtop{n} \to X$ are globularly parallel if
their restrictions of the boundary $\eDtop{n}$ of $\Dtop{n}$ coincide, i.e.,
if they induce a map
\[ (f, g) : \eDtop{n+1} = \Dtop{n} \amalg_{\eDtop{n}} \Dtop{n} \to X. \]
A lifting of the pair $(f, g)$ corresponds to a lifting of the induced map
$\eDtop{n+1} \to X$ to $\Dtop{n+1}$, that is, to a map $h : \Dtop{n+1} \to X$
such that the triangle
\[
\xymatrix{
\Dtop{n+1} \ar[rd]^h \\
\eDtop{n+1} \ar[r]_-{(f,g)} \ar[u] & X\\
}
\]
commutes. Note that when $X$ is fixed, such an $h$ exists for every $n \ge
0$ and every $(f, g) : \Dtop{n} \to X$ globularly parallel if and only if
$X$ is weakly contractible.
\end{ex}

\begin{tparagr}{Admissible pairs}
Let $C$ be a globular extension.
A pair of morphisms 
\[ (f, g) : \Dn{n} \to S \]
of $C$ whose source is a
$\Dn{n}$, $n \ge 0$, is said to be \ndef{$(\infty, 0)$-admissible},
or briefly \ndef{admissible}, if
\begin{itemize}
\item the morphisms $f$ and $g$ are globularly parallel;
\item the object $S$ is a globular sum;
\item the dimension of $S$ is less than or equal to $n + 1$.
\end{itemize}
An admissible pair is \ndef{strictly admissible} if it does not admit a
lifting.
\end{tparagr}

\begin{tparagr}{Contractible globular extensions}
We will say that a globular extension $C$ is \ndef{$(\infty, 0)$\nbd-contractible},
or briefly \ndef{contractible}, if every admissible pair of $C$ admits a
lifting. Such a globular extension is called a pseudo-coherator in
\cite{MaltsiGrCat} and \cite{AraThesis}.
\end{tparagr}

\begin{tparagr}{$\infty$-groupoids of type~$C$}
Let $C$ be a contractible globular extension. The category of
\ndef{\oo-groupoids of type~$C$} is the category $\Mod{C}$ of globular
presheaves on $C$. We will denote it in a more suggestive way by
$\wgpdC{C}$.
\end{tparagr}

\begin{rem}
It might be more reasonable to define \oo-groupoids of type $C$ only when
objects of~$C$ are in bijection with tables of dimensions. We chose not to
do so for technical reasons (see for instance Proposition \ref{prop:PI_f}).
\end{rem}

\begin{exs}\label{exs:contr}
\ 
\begin{myenumerate}
\item\label{item:Top_contr}
The globular extension $\Top$ is contractible. Indeed, it is obvious
that every globular sum in $\Top$ is contractible in the topological sense
(see Proposition \ref{prop:glob_contr} for a proof in a more general
setting). The contractibility of $\Top$ then follows from the last assertion
of Example \ref{ex:Toplift}.

\item\label{item:Thtld_contr}
We will prove in \cite{AraStrWeak} that the globular extension $\Thtld$ is
contractible. More precisely, we will show that every admissible pair in
$\Thtld$ admits a \emph{unique} lifting. This will allow us to define a
canonical inclusion functor of strict \oo-groupoids into \oo-groupoids of
type $C$, where $C$ is a coherator endowed with a defining tower (see Remark
\ref{rem:incl_strict}).

\item
The globular extension $\Th$ is \emph{not} $(\infty, 0)$-contractible. Indeed,
the admissible pair $(\Tht{n}, \Ths{n})$, $n \ge 1$, does not admit a
lifting. This reflects the fact that globular presheaves on~$\Th$, i.e.,
strict \oo-categories, do not have inverses. Nevertheless, the globular
extension~$\Th$ is $(\infty, \infty)$-contractible in some sense (admissible
for a theory of \oo-categories in the terminology of \cite{MaltsiGrCat} and
\cite{AraThesis}). See Proposition 5.1.5 of \cite{AraThesis}.
\end{myenumerate}
\end{exs}

\begin{rems}
\ 
\begin{myenumerate}
\item 
The definition of admissible pairs differs from the one given in the
previous texts on Grothendieck \oo-groupoids: a dimensional condition has
been added. The purpose of this condition is to make unique the lifting of an
admissible pair in $\Thtld$ and hence to make canonical the inclusion
functor of strict \oo-groupoids into Grothendieck \oo-groupoids.

\item 
The category $\wgpdC{C}$ should \emph{not} be thought of as a category of
\emph{weak} \oo-groupoids unless $C$ satisfies some freeness condition.
Indeed, for $C = \Thtld$, the category $\wgpdC{C}$ is nothing but the
category of strict \oo-groupoids. One way to define weak \oo-groupoids
without defining this freeness condition would be to define the category
of weak \oo-group\-oids as the ``union'' of all the $\wgpdC{C}$'s, where $C$
varies among contractible globular extensions.
\end{myenumerate}
\end{rems}

\begin{tparagr}{Adding liftings to globular extensions}
Let $C$ be a globular extension and let~$A$ be a set of admissible pairs of
$C$. We will denote by $\CHAb$ the category obtained from $C$ by formally
adding a lifting $h_{(f, g)}$ to
every pair $(f, g)$ in $A$. More precisely, the category $\CHAb$ is the
category, endowed with a functor $C \to \CHAb$ such that the image of every
pair of $A$ admits a lifting in $\CHAb$, satisfying the following
universal property: for every category $D$, endowed with a functor $C \to D$
such that the image of every pair of~$A$ admits a lifting in $D$, there
exists a unique functor $\CHAb \to D$ such that the triangle
\[
\xymatrix@R=.2pc{
& \CHAb \ar[dd] \\
C \ar[ur] \ar[dr] \\
& D
}
\]
commutes.

The category $\CHAb$ is naturally a category under $\G$ but it has no
reason to be a globular extension. Let us assume that the globular
extension $C$ is a globular extension under $\Thz$. Then $\CHAb$ is also a
category under $\Thz$ and we can consider its globular completion
$\glcomp{\CHAb}$. We will denote this globular extension under $\Thz$ by
$\CHA$. The category $\CHA$ has the following universal property:
for every globular extension $D$ over~$\Thz$, endowed
with a functor $C \to D$ under $\Thz$ such that the
image of every admissible pair of $A$ admits a lifting in $D$, there
exists a unique functor $\CHA \to D$ such that the triangle
\[
\xymatrix@R=.2pc{
& \CHA \ar[dd] \\
C \ar[ur] \ar[dr] \\
& D
}
\]
commutes. Note that the functor $C \to \CHA$ is bijective on objects.
\end{tparagr}

\begin{tparagr}{Free globular extensions}
A \ndef{cellular tower} of globular extensions is a tower of globular
extensions
\[ C_0 = \Thz \to C_1 \to \dots C_n \to \dots\pbox{,} \]
endowed, for each $n \ge 0$, with a set $A_n$ of admissible pairs of $C_n$
such that
\[  C_{n+1} = \CHAd{C_n}{A_n}. \]
Such a tower is entirely defined by the $A_n$'s. We will say that a cellular
tower $(C_n, A_n)$ \ndef{defines a globular extension} $C$ if $C$ is
isomorphic to $\varinjlim C_n$. In this case, we will also say that
$(C_\ast, A_\ast)$ is a \ndef{defining tower} of $C$.

We will say that a globular extension $C$ is \ndef{free} if $C$ admits a
defining tower. Note that if $C$ is free, a defining tower of $C$ gives a
functor $\Thz \to C$ which is bijective on objects. 
\end{tparagr}

\begin{tparagr}{Coherators}
An \ndef{$(\infty, 0)$-coherator}, or briefly a \ndef{coherator}, is a
globular extension which is free and contractible.
\end{tparagr}

\begin{exs}
Let $(C_\ast, A_\ast)$ be a cellular tower. If the $A_n$'s
are such that every admissible pair of $C = \varinjlim C_n$ comes from an
$A_n$, then $C$ is a coherator. This remark allows us to define three
coherators:
\begin{itemize}
\item The canonical coherator: $A_n$ is the set of all admissible pairs of
$C_n$.
\item The reduced canonical coherator: $A_n$ is the set of all strictly
admissible pairs of~$C_n$.
\item The Batanin-Leinster coherator: $A_n$ is the set of the admissible pairs
that do not come from an $A_m$, $m \le n$, via the functor $C_m \to C_n$.
The name of this coherator comes from the relation it bears with
Batanin-Leinster \oo-categories (see Section~6.7 of \cite{AraThesis}).
\end{itemize}
\end{exs}

\begin{rem}
If $C$ is a coherator, the category of \oo-groupoids of type $C$ can be
thought of as a category of \emph{weak} \oo-groupoids. The fact that this
category depends on a coherator reflects the non-uniqueness of
the choice of generators for higher coherences. Nevertheless, if $C$ and $C'$
are two coherators, the category $\wgpdC{C}$ and $\wgpdC{C'}$ should be
equivalent in some weak sense to be defined. Note that Grothendieck's
conjecture (Conjecture \ref{conj:Groth}) implies that their homotopy
categories are equivalent.
\end{rem}

\begin{prop}\label{prop:coh_weak_init}
Let $C$ be a free globular extension. For any contractible globular
extension $D$, there exists a globular functor $C \to D$. 
\end{prop}

\begin{proof}
Let $(C_\ast, A_\ast)$ be a cellular tower defining $C$. By the universal
property of $\Thz$, the functor $\G \to D$ lifts to a globular functor $F_0 : C_0 =
\Thz \to D$. Suppose now by induction that we have a globular functor $F_n : C_n \to
D$. Every admissible pair of $A_n$ is sent to an admissible pair of $D$.
Since $D$ is contractible, every such pair admits a lifting, and, by the
universal property of $C_{n+1} = \CHAd{C_n}{A_n}$, we can lift $F_n$ to a
globular functor $F_{n+1} : C_{n+1} \to D$. We hence get a functor $F =
\varinjlim F_n : C \to D$ which is obviously globular.
\end{proof}

\begin{rems}\label{rem:incl_strict}
\ 
\begin{myenumerate}
\item 
The globular functor $C \to D$ of the previous proposition is not unique: it
depends on a cellular tower defining $C$ and on a choice of liftings. 
Nevertheless, if such a choice is made, the globular functor $C \to D$
becomes unique (up to a unique isomorphism).

\item
In particular, if $C$ is a coherator endowed with a defining tower and $D =
\Thtld$, by Example \ref{exs:contr}.\ref{item:Thtld_contr}, there exists a
unique functor $F : C \to \Thtld$ under $\Thz$. This functor induces a
functor from strict \oo-groupoids to \oo-groupoids of type $C$.  This is the
canonical inclusion functor of strict \oo-groupoids into \oo-groupoids of
type $C$.
\end{myenumerate}
\end{rems}

\section{Some structural maps of Grothendieck $\infty$-groupoids}
\label{sec:struct_maps}

\begin{paragr}
In this section, we fix a contractible globular extension $C$ and an
\oo-groupoid $G$ of type $C$. The purpose of the section is to convince the
reader that $G$ deserves to be called an \oo-groupoid. For this purpose, we
will explain how to construct structural maps (i.e., operations and
coherences) for $G$ out of $C$. More precisely, we will show that $G$ can be
endowed with compositions, units and inverses, and that these operations
satisfy the axioms of strict \oo-groupoids up to coherences. We will also
give examples of higher coherences between those coherences.
\end{paragr}

\begin{tparagr}{First example: codimension $1$
compositions}\label{paragr:first_ex}
Let $i \ge 1$. We will explain how to endow $G$ with a
composition of $i$-arrows in codimension $1$, i.e., with a map
\[ \comp_{i-1}^i : G_i \times^{}_{G_{i-1}} G_i \to G_i \]
sending $i$-arrows 
\[ x \xrightarrow{u} y \xrightarrow{v} z \]
to an $i$-arrow
\[ v \comp^i_{i-1} u : x \to z. \]

Let $p_1, p_2 : G_i \times^{}_{G_{i-1}} G_i \to G_i$ denote the canonical
projections. The conditions on the source and the target of $\comp_{i-1}^i$
can be rewritten in the following way:
\[
\Gls{i}\comp^i_{i-1} = \Gls{i}p_2
\quad\text{and}\quad
\Glt{i}\comp^i_{i-1} = \Glt{i}p_1.
\]
The important fact that will allow us to construct our map $\comp^i_{i+1}$
is that the maps $\Gls{i}p_2$ and $\Glt{i}p_1$ are induced by morphisms of
$C$. Indeed, denote by $\ceps{1}, \ceps{2} : \Dn{i} \to \Dn{i} \amalgd{i-1}
\Dn{i}$ the canonical morphisms. Since $G$ is a globular presheaf, the
canonical morphism 
\[ j : G(\Dn{i} \amalgd{i-1} \Dn{i}) \to G_i \times_{G_{i-1}} G_i \]
is a bijection and we have
\[
 G(\ceps{i}) = p_ij, \quad i = 1,2.
\]
It follows that
\[
G(\ceps{2}\Ths{i})j^{-1} = \Gls{i}p_2
\quad\text{and}\quad
G(\ceps{1}\Tht{i})j^{-1} = \Glt{i}p_1.
\]
Consider now the pair
\[
(\ceps{2}\Ths{i}, \ceps{1}\Tht{i}) : \Dn{i-1} \to \Dn{i} \amalgd{i-1}
\Dn{i}
\]
of morphisms of $C$. We claim that this pair is admissible. For $i = 1$,
there is nothing to check. For $i \ge 2$,
we have
\[
\ceps{2}\Ths{i}\Ths{i-1} = \ceps{2}\Tht{i}\Ths{i-1} =
\ceps{1}\Ths{i}\Ths{i-1} = \ceps{1}\Tht{i}\Ths{i-1}
\]
and
\[
\ceps{2}\Ths{i}\Tht{i-1} = \ceps{2}\Tht{i}\Tht{i-1} =
\ceps{1}\Ths{i}\Tht{i-1} = \ceps{1}\Tht{i}\Tht{i-1}.
\]
Since $C$ is contractible, this pair admits a lifting in $C$.
It follows that there exists a morphism
\[ \Thn[i-1]{i} : \Dn{i} \to \Dn{i} \amalgd{i-1} \Dn{i} \]
in $C$ such that
\[
\Thn[i-1]{i}\Ths{i} = \ceps{2}\Ths{i}
\quad\text{and}\quad
\Thn[i-1]{i}\Tht{i} = \ceps{1}\Tht{i}.
\]
This morphism will also be denoted by $\Thn{i}$. It induces a map
\[ \comp_{i-1}^i : G_i \times^{}_{G_{i-1}} G_i \xrightarrow{j^{-1}}
G(\Dn{i} \amalgd{i-1} \Dn{i}) \xrightarrow{G(\Thn{i})} G_i \]
which has the desired source and target. Indeed, we have
\[
\Gls{i}\comp^i_{i-1} = G(\Ths{i})G(\Thn{i})j^{-1} = G(\Thn{i}\Ths{i})j^{-1}
= G(\ceps{2}\Ths{i})j^{-1} = \Gls{i}p_2
\]
and
\[
\Glt{i}\comp^i_{i-1} = G(\Tht{i})G(\Thn{i})j^{-1} = G(\Thn{i}\Tht{i})j^{-1}
= G(\ceps{1}\Tht{i})j^{-1} = \Glt{i}p_1.
\]
Note that the composition $\comp^i_{i-1}$ depends on the choice of the lifting $\Thn{i}$.
Nevertheless, it is easy to show that this composition is unique up to
$(i+1)$-arrows. See Proposition \ref{prop:PI_indep} for details.
\end{tparagr}

\begin{tparagr}{The general pattern}
Let $i \ge 1$ and let \[ \tabdim  \] be a table of dimensions. Suppose we
want to construct a structural map
\[
m : G_{i_1} \times_{G_{i'_1}} \dots \times_{G_{i'_{n-1}}} G_{i_n} \to G_i
\]
such that 
\[
\Gls{i}m = f 
\quad\text{and}\quad
\Glt{i}m = g,
\]
where
\[  
f, g : G_{i_1} \times_{G_{i'_1}} \dots \times_{G_{i'_{n-1}}} G_{i_n}
\to G_{i-1}
\] 
are two fixed maps. To do so, we first have to find morphisms
\[
\varphi, \gamma :
\Dn{i-1} \to \Dn{i_1} \amalgd{i'_1} \dots \amalgd{i'_{n-1}} \Dn{i_n}
\]
in $C$ such that
\[
G(\varphi)j^{-1} = f
\quad\text{and}\quad 
G(\gamma)j^{-1} = g,
\] 
where
\[
j : G(\Dn{i_1} \amalgd{i'_1} \dots \amalgd{i'_{n-1}} \Dn{i_n}) \to
G_{i_1} \times_{G_{i'_1}} \dots \times_{G_{i'_{n-1}}} G_{i_n}
\]
denote the canonical morphism. We must then check that the pair $(\varphi,
\gamma)$ is admissible (this will be the case if the structural
map $m$ is ``reasonable''). Then, any lifting
\[
\mu : \Dn{i} \to\Dn{i_1} \amalgd{i'_1} \dots \amalgd{i'_{n-1}}
\Dn{i_n}
\]
in $C$ of the pair $(\phi, \gamma)$ will induce a map
\[
m : G_{i_1} \times_{G_{i'_1}} \dots \times_{G_{i'_{n-1}}} G_{i_n}
\xrightarrow{j^{-1}}
G(\Dn{i_1} \amalgd{i'_1} \dots \amalgd{i'_{n-1}} \Dn{i_n})
\xrightarrow{G(\mu)} G_i \]
with the desired source and target.
\end{tparagr}

\begin{paragr}
In the rest of this section, we will assume that our contractible globular
extension~$C$ is the canonical coherator. The canonical cellular tower
defining $C$ will be denoted by $(C_\ast)$. It should be clear to the reader
that all the structural morphisms that we will define in $C$ exist in every
contractible globular extension. Our exposition only uses the canonical
coherator to highlight the natural hierarchy between structural maps.
For instance, since the pair
\[
(\ceps{2}\Ths{i}, \ceps{1}\Tht{i}) : \Dn{i-1} \to \Dn{i} \amalgd{i-1}
\Dn{i}
\]
considered in paragraph \ref{paragr:first_ex} actually comes from $C_0 = \Thz$, it
admits a lifting in $C_1$. The composition $\comp^i_{i-1}$ is hence in some
sense a primary operation.
\end{paragr}

\begin{tparagr}{Examples of structural maps appearing in $C_1$}
\begin{myitemize}
\item \emph{Codimension $1$ compositions}\\
See paragraph \ref{paragr:first_ex}.

\item \emph{Units} \\
Let $i \ge 0$. The pair
\[ (\id{\Dn{i}}, \id{\Dn{i}}) : \Dn{i} \to \Dn{i} \]
of morphisms of $C_0$ is obviously admissible. Hence there exists a lifting
\[ \Thk{i} : \Dn{i+1} \to \Dn{i} \]
in $C_1$ such that
\[ \Thk{i}\Ths{i+1} = \id{\Dn{i}}
\quad\text{and}\quad
\Thk{i}\Tht{i+1} = \id{\Dn{i}}.
\]
This morphism induces a structural map
\[ \Glk{i} : G_i \to G_{i+1} \]
of $G$ for units. This map sends an $i$-arrow $u$ to an $(i+1)$-arrow
\[ \Glk{i}(u) : u \to u. \]
We will see that $\Glk{i}(u)$ is a unit for the
composition $\comp^{i+1}_i$ up to an $(i+2)$-arrow.

\item \emph{Codimension $1$ inverses} \\
Let $i \ge 1$. It follows from the coglobular identities that the pair
\[ (\Tht{i}, \Ths{i}) : \Dn{i-1} \to \Dn{i} \]
of morphisms of $C_0$ is admissible. Hence there exists a lifting
\[ \Thw[i-1]{i} : \Dn{i} \to \Dn{i} \]
in $C_1$ such that
\[
\Thw[i-1]{i}\Ths{i} = \Tht{i}
\quad\text{and}\quad
\Thw[i-1]{i}\Tht{i} = \Ths{i}.
\]
This morphism will also be denoted by $\Thw{i}$. It induces a structural map
\[ \Glw[i-1]{i} : G_i \to G_i \]
of $G$ for codimension $1$ inverses. This maps sends
an $i$-arrow $u : x \to y$
to an $i$-arrow 
\[ \Glw[i-1]{i}(u) : y \to x. \]
We will see that $\Glw[i-1]{i}(u)$ is an inverse of $u$ for the composition
$\comp^{i}_{i-1}$ up to an $(i+1)$-arrow.

\end{myitemize}
\end{tparagr}

\begin{rem}
The list of structural maps appearing in $C_1$ we have given above is not
exhaustive: $n$-ary compositions or operations mixing compositions, units
and inverses are other examples.
\end{rem}

\begin{tparagr}{Examples of structural maps appearing in $C_2$}
\label{paragr:coh_cat}
\begin{myitemize}
\item \emph{Codimension $2$ compositions} \\
Let $i \ge 2$. The pair
\[
((\ceps{1}\Ths{i}, \ceps{2}\Ths{i})\Thn{i-1},
(\ceps{1}\Tht{i}, \ceps{2}\Tht{i})\Thn{i-1}) :
\Dn{i-1} \to \Dn{i} \amalgd{i-2} \Dn{i}
\]
of morphisms of $C_1$,
where $\ceps{1}, \ceps{2} : \Dn{i} \to \Dn{i} \amalgd{i-2} \Dn{i}$
denote the canonical morphisms, is admissible. Indeed, we have
\[
\begin{split}
(\ceps{1}\Ths{i}, \ceps{2}\Ths{i})\Thn{i-1}\Ths{i-1}
& =
(\ceps{1}\Ths{i}, \ceps{2}\Ths{i})\ceps{2}\Ths{i-1} \\
& =
\ceps{2}\Ths{i}\Ths{i-1}
= \ceps{2}\Tht{i}\Ths{i-1} \\
& = 
(\ceps{1}\Tht{i}, \ceps{2}\Tht{i})\ceps{2}\Ths{i-1} \\
& =
(\ceps{1}\Tht{i}, \ceps{2}\Tht{i})\Thn{i-1}\Ths{i-1}.
\end{split}
\]
In the same way, we get
\[
(\ceps{1}\Ths{i}, \ceps{2}\Ths{i})\Thn{i-1}\Tht{i-1} =
\ceps{1}\Tht{i}\Tht{i-1} =
(\ceps{1}\Tht{i}, \ceps{2}\Tht{i})\Thn{i-1}\Tht{i-1}.
\]
Hence there exists a lifting
\[ \Thn[i-2]{i} : \Dn{i} \to \Dn{i} \amalgd{i-2} \Dn{i} \]
in $C_2$ such that
\[
\Thn[i-2]{i}\Ths{i} =
(\ceps{1}\Ths{i}, \ceps{2}\Ths{i})\Thn{i-1}
\quad\text{and}\quad
\Thn[i-2]{i}\Tht{i} = 
(\ceps{1}\Tht{i}, \ceps{2}\Tht{i})\Thn{i-1}.
\]
This morphism induces a structural map
\[ \comp_{i-2}^i : G_i \times^{}_{G_{i-2}} G_i \to G_i \]
of $G$ for codimension $2$ composition. This map sends
$i$-arrows
\[
\UseAllTwocells
\xymatrix@C=3pc{
x \rtwocell^u_v{\,\alpha}
&
y \rtwocell^{u'}_{v'}{\,\,\alpha'}
&
z
}
\]
to an $i$-arrow 
\[
\UseAllTwocells
\alpha' \comp^i_{i-2} \alpha :
\xymatrix@C=3pc{
y \rtwocell^{<1.5>u' \comp^{i-1}_{i-2} u}_{<1.5>v' \comp^{i-1}_{i-2} v}
&
x \pbox{.}
}
\]

\item \emph{Codimension $2$ inverses} \\
Let $i \ge 2$. The pair 
\[ (\Ths{i}\Thw{i-1}, \Tht{i}\Thw{i-1}) : \Dn{i-1} \to \Dn{i} \]
of morphisms of $C_1$ is admissible. Indeed, we have
\[ \Ths{i}\Thw{i-1}\Ths{i-1} = \Ths{i}\Tht{i-1} = \Tht{i}\Tht{i-1} =
\Tht{i}\Thw{i-1}\Ths{i-1}. \]
In the same way, we get
\[ \Ths{i}\Thw{i-1}\Tht{i-1} = \Ths{i}\Ths{i-1} = \Tht{i}\Thw{i-1}\Tht{i-1}. \]
Hence there exists a lifting
\[
\Thw[i-2]{i} : \Dn{i} \to \Dn{i}
\]
in $C_2$ such that
\[
\Thw[i-2]{i}\Ths{i} = \Ths{i}\Thw{i-1}
\quad\text{and}\quad
\Thw[i-2]{i}\Tht{i} = \Tht{i}\Thw{i-1}.
\]
This morphism induces a structural map
\[ \Glw[i-2]{i} : G_i \to G_i \]
of $G$ for codimension $2$ inverses. This map sends an
$i$-arrow
\[
\UseAllTwocells
\xymatrix@C=3pc{
x \rtwocell^u_v{\,\alpha}
&
y
}
\]
to an $i$-arrow
\[
\Glw[i-2]{i}(\alpha) :
\UseAllTwocells
\xymatrix@C=3pc{
y \rtwocell^{<1.5>\Glw[i-2]{i-1}(u)}_{<1.5>\Glw[i-2]{i-1}(v)}
&
x \pbox{.}
}
\]
We will see that $\Glw[i-2]{i}(\alpha)$ is an inverse of $\alpha$ for the
composition $\comp^i_{i-2}$ up to an $(i+1)$-arrow.

\item \emph{Codimension $1$ associativity constraints} \\
Let $i \ge 1$. The pair 
\[ \Big(\big(\Thn{i}\amalgd{i-1} \id{\Dn{i}}\big)\Thn{i}, 
\big(\id{\Dn{i}} \amalgd{i-1} \Thn{i}\big)\Thn{i}\Big) :
\Dn{i} \to \Dn{i}\amalgd{i-1}\Dn{i}\amalgd{i-1}\Dn{i} \]
of morphism of $C_1$ is admissible. Indeed, if
\[
\ceps{1}, \ceps{2}, \ceps{3}, : \Dn{i} \to \Dn{i} \amalgd{i-1} \Dn{i} \amalgd{i-1} \Dn{i}
\quad\text{and}\quad
\cepsp{1}, \cepsp{2} : \Dn{i} \to \Dn{i} \amalgd{i-1} \Dn{i}
\]
denote the canonical morphisms, then we have
\[
\begin{split}
\big(\Thn{i} \amalgd{i-1} \Dn{i}\big)\Thn{i}\Ths{i}
& = \big(\Thn{i} \amalgd{i-1} \Dn{i}\big)\cepsp{2}\Ths{i} 
= \ceps{3}\Ths{i} \\
& = (\ceps{2}, \ceps{3})\cepsp{2}\Ths{i}
= (\ceps{2}, \ceps{3})\Thn{i}\Ths{i} \\
& = \big(\Dn{i} \amalgd{i} \Thn{i}\big)\cepsp{2}\Ths{i} \\
& = \big(\Dn{i} \amalgd{i} \Thn{i}\big)\Thn{i}\Ths{i}.
\end{split}
\]
In the same way, we get
\[
\big(\Thn{i} \amalgd{i-1} \Dn{i}\big)\Thn{i}\Tht{i}
= \ceps{1}\Tht{i} \\
= \big(\Dn{i} \amalgd{i} \Thn{i}\big)\Thn{i}\Tht{i}.
\]
Hence there exists a lifting
\[
 \alpha^{}_i : \Dn{i+1} \to \Dn{i}\amalgd{i-1}\Dn{i}\amalgd{i-1}\Dn{i} \]
in $C_2$ such that
\[
\alpha^{}_i\Ths{i+1} = 
\big(\Thn{i}\amalgd{i-1} \Dn{i}\big)\Thn{i}
\quad\text{and}\quad
\alpha^{}_i\Tht{i+1} = 
\big(\Dn{i} \amalgd{i-1} \Thn{i}\big)\Thn{i}.
\]
This morphism induces a structural map 
\[
a_i : G_i \times_{G_{i-1}} G_i \times_{G_{i-1}} G_i \to G_{i+1}
\]
of $G$ for associativity constraints for the composition $\comp^i_{i-1}$.
This map sends $i$-arrows
\[
\xymatrix{
\ar[r]^u &
\ar[r]^v &
\ar[r]^w &
}
\]
to an $(i+1)$-arrow
\[
a^{}_{w,v,u} : \big(w \comp^i_{i-1} v\big) \comp^i_{i-1} u \to
w \comp^i_{i-1} \big(v \comp^i_{i-1} u\big).
\]
This shows that the composition $\comp^i_{i-1}$ is associative up to
$(i+1)$-arrows.

\item \emph{Codimension $1$ unit constraints} \\
Let $i \ge 1$. Consider the pair 
\[
\big(\big(\id{\Dn{i}}, \Ths{i}\Thk{i-1}\big)\Thn{i}, \id{\Dn{i}}\big) :
\Dn{i} \to \Dn{i}
\]
of morphisms of $C_1$. (Note that the fact that the morphism
\[ \big(\id{\Dn{i}}, \Ths{i}\Thk{i-1}\big) : \Dn{i} \amalgd{i-1} \Dn{i} \to
\Dn{i} \]
is well-defined requires a calculation. We will skip these calculations in
this section.) We claim that this pair is admissible.  Indeed, if $\ceps{1},
\ceps{2} :
\Dn{i} \to \Dn{i} \amalgd{i-1} \Dn{i}$ denote the canonical morphisms, then
we have
\[
\big(\id{\Dn{i}}, \Ths{i}\Thk{i-1}\big)\Thn{i}\Ths{i}
=
\big(\id{\Dn{i}}, \Ths{i}\Thk{i-1}\big)\ceps{2}\Ths{i}
=
\Ths{i}\Thk{i-1}\Ths{i}
= 
\Ths{i}
\]
and
\[
\big(\id{\Dn{i}}, \Ths{i}\Thk{i-1}\big)\Thn{i}\Tht{i}
=
\big(\id{\Dn{i}}, \Ths{i}\Thk{i-1}\big)\ceps{1}\Tht{i}
=
\Tht{i}.
\]
Hence there exists a lifting
\[ \rho^{}_i : \Dn{i+1} \to \Dn{i} \]
in $C_2$ such that
\[
\rho^{}_i\Ths{i+1} = \big(\id{\Dn{i}},\Ths{i}\Thk{i-1}\big)\Thn{i}
\quad\text{and}\quad
\rho^{}_i\Tht{i+1} = \id{\Dn{i}}.
\]
This morphism induces a structural map
\[
r_i : G_i \to G_{i+1}
\]
of $G$ for right unit constraints for $\Glk{i-1}$. This map sends
an $i$-arrow $u : x \to y$ to an $(i+1)$-arrow
 \[
r^{}_u : u \comp^i_{i-1} \Glk{i-1}(x) \to u.
\]
This shows that $\Glk{i-1}(x)$ is a right unit for the composition
$\comp^i_{i-1}$ up to an $(i+1)$-arrow.

We get in a similar way a morphism
\[ \lambda^{}_i : \Dn{i+1} \to \Dn{i} \]
of $C_2$ inducing a codimension $1$ left unit constraint
\[
l_i : G_i \to G_{i+1}.
\]
We thus get an $(i+1)$-arrow
\[
l^{}_u : \Glk{i-1}(y) \comp^i_{i-1} u \to u.
\]
showing that $\Glk{i-1}(y)$ is a left unit for the composition
$\comp^i_{i-1}$ up to an $(i+1)$-arrow.

\item \emph{Codimension $1$ inverse constraints} \\
Let $i \ge 1$. The pair 
\[ \big(\big(\id{\Dn{i}}, \Thw{i}\big)\Thn{i}, \Tht{i}\Thk{i-1}\big) :
\Dn{i} \to \Dn{i} \]
of morphism of $C_1$ is admissible. Indeed,
if $\ceps{1}, \ceps{2} : \Dn{i} \to \Dn{i} \amalgd{i-1} \Dn{i}$
  denote the canonical morphisms, then we have
\[
\big(\id{\Dn{i}}, \Thw{i}\big)\Thn{i}\Ths{i}
=
\big(\id{\Dn{i}}, \Thw{i}\big)\ceps{2}\Ths{i}
=
\Thw{i}\Ths{i}
=
\Tht{i}
=
\Tht{i}\Thk{i-1}\Ths{i}
\]
and
\[
\big(\id{\Dn{i}}, \Thw{i}\big)\Thn{i}\Tht{i}
=
\big(\id{\Dn{i}}, \Thw{i}\big)\ceps{1}\Tht{i}
=
\Tht{i}
=
\Tht{i}\Thk{i-1}\Tht{i}.
\]
Hence there exists a lifting
\[
\delta^{}_i : \Dn{i+1} \to \Dn{i}
\]
in $C_2$ such that
\[
\delta^{}_i\Ths{i+1} = \big(\id{\Dn{i}}, \Thw{i}\big)\Thn{i}
\quad\text{and}\quad
\delta^{}_i\Tht{i+1} = \Tht{i}\Thk{i-1}.
\]
This morphism induces a structural map
\[
d_i : G_i \to G_{i+1}
\]
of $G$ for right inverse constraints for $\Glw[i-1]{i}$. This
map sends an $i$-arrow $u : x \to y$ to an $(i+1)$-arrow
\[
d^{}_u : u \comp^i_{i-1} \Glw[i-1]{i}(u) \to \Glk{i-1}(y).
\]
This shows that $\Glw[i-1]{i}(u)$ is a right inverse of $u$ for the composition
$\comp^i_{i-1}$ up to an \hbox{$(i+1)$}\nbd-arrow.

We get in a similar way a morphism
\[ \gamma^{}_i : \Dn{i+1} \to \Dn{i} \]
of $C_2$ inducing a left inverse constraint. We thus get an $(i+1)$-arrow
\[
g^{}_u : \Glw[i-1]{i}(u) \comp^i_{i-1} u \to \Glk{i-1}(x)
\]
showing that $\Glw[i-1]{i}(u)$ is a left inverse of $u$ for the composition
$\comp^i_{i-1}$ up to an \hbox{$(i+1)$}\nbd-arrow.
\end{myitemize}
\end{tparagr}

\begin{tparagr}{Examples of morphisms appearing in $C_3$}
\label{paragr:coh_bicat}
\begin{myitemize}
\item \emph{Codimension $1$ Mac Lane's pentagon constraints} \\
Let $i \ge 1$. Denote by
\[
 \ceps{1}, \dots, \ceps{4} :
 \Dn{i} \to \Dn{i} \amalgd{i-1} \Dn{i} \amalgd{i-1} \Dn{i} \amalgd{i-1} \Dn{i}
\]
the canonical morphisms. Let
$c_2 : \Dn{i+1} \to \Dn{i} \amalgd{i-1}
\Dn{i} \amalgd{i-1} \Dn{i} \amalgd{i-1} \Dn{i}$ be the morphism
\[
\Big(
\big(\Dn{i} \amalgd{i-1} \Dn{i} \amalgd{i-1} \Thn{i}\big)\alpha^{}_i,
\big(\Thn{i} \amalgd{i-1} \Dn{i} \amalgd{i-1} \Dn{i}\big)\alpha^{}_i
\Big)
\Thn{i+1}
\]
and let
$c_3 : \Dn{i+1} \to \Dn{i} \amalgd{i-1}
\Dn{i} \amalgd{i-1} \Dn{i} \amalgd{i-1} \Dn{i}$
be the morphism
\[
\begin{split}
&
\Big(
\big(\ceps{1}\Thk{i}, (\ceps{2}, \ceps{3},
\ceps{4})\alpha^{}_i\big)\Thn[i-1]{i+1},
\big(\Dn{i} \amalgd{i-1} \Thn{i} \amalgd{i-1} \Dn{i}\big)\alpha^{}_i,
\\
& \quad \big((\ceps{1}, \ceps{2},
\ceps{3})\alpha^{}_i,
\ceps{4}\Thk{i}\big)\Thn[i-1]{i+1}
\Big)
\Big(\Thn{i+1} \amalgd{i} \Dn{i+1}\Big)
\Thn{i+1}.
\end{split}
\]
The pair
\[
(c_3, c_2) : \Dn{i+1} \to \Dn{i} \amalgd{i-1}
\Dn{i} \amalgd{i-1} \Dn{i} \amalgd{i-1} \Dn{i}
\]
of morphisms of $C_2$ is admissible. Indeed,
if $\cepsp{1}, \cepsp{2} : \Dn{i+1} \to \Dn{i+1} \amalgd{i-1} \Dn{i+1}$
denote the canonical morphisms, then we have
\[
\begin{split}
c_3\Ths{i+1} & =
\big((\ceps{1}, \ceps{2}, \ceps{3})\alpha^{}_i,
\ceps{4}\Thk{i}\big)
\Thn[i-1]{i+1}\Ths{i+1} \\
& = \big((\ceps{1}, \ceps{2},
\ceps{3})\alpha^{}_i,
\ceps{4}\Thk{i}\big)
\big(\cepsp{1}\Ths{i+1}, \cepsp{2}\Ths{i+1}\big)\Thn{i} \\
& = \big((\ceps{1}, \ceps{2},
\ceps{3})\alpha^{}_i\Ths{i+1},
\ceps{4}\Thk{i}\Ths{i+1}\big)
\Thn{i} \\
& = \big((\ceps{1}, \ceps{2},
\ceps{3})
(\Thn{i} \amalgd{i-1} \Dn{i}) \Thn{i},
\ceps{4}\big)
\Thn{i} \\
& =
\big(\Thn{i} \amalgd{i-1} \Dn{i} \amalgd{i-1}
\Dn{i}\big)
\big(\Thn{i} \amalgd{i-1} \Dn{i}\big)
\Thn{i}\\
\end{split}
\]
and
\[
\begin{split}
c_2\Ths{i+1}
& =
\big(\Thn{i} \amalgd{i-1} \Dn{i} \amalgd{i-1}
\Dn{i}\big)\alpha^{}_i\Ths{i+1} \\
& =
\big(\Thn{i} \amalgd{i-1} \Dn{i} \amalgd{i-1}
\Dn{i}\big)
\big(\Thn{i} \amalgd{i-1} \Dn{i}\big)
\Thn{i}.\\
\end{split}
\]
A similar calculation shows that
\[ c_3\Tht{i+1} = 
\big(\Dn{i} \amalgd{i-1} \Dn{i} \amalgd{i-1}
\Thn{i}\big)
\big(\Dn{i} \amalgd{i-1} \Thn{i}\big)
\Thn{i} = c_2\Tht{i+1}.
\]
Hence there exists a lifting
\[ \pi_i : \Dn{i+2} \to \Dn{i} \amalgd{i-1} \Dn{i} \amalgd{i-1} \Dn{i}
\amalgd{i-1} \Dn{i} \]
in $C_3$ such that
\[
\pi_i\Ths{i+2} = c_3
\quad\text{and}\quad
\pi_i\Tht{i+2} = c_2.
\]
This morphism induces a structural map of $G$ for Mac Lane's pentagon
constraints for compositions $\comp^i_{i-1}$, $\comp^{i+1}_i$ and
$\comp^{i+1}_{i-1}$. This maps sends $i$-arrows
\[
\xymatrix{
\ar[r]^u &
\ar[r]^v &
\ar[r]^w &
\ar[r]^x &
}
\]
to an $(i+2)$-arrow
\renewcommand{\objectstyle}{\scriptstyle}
\newcommand\mynodes{\ifcase\xypolynode\or
      {\big(x \comp^i_{i-1} w\big) \comp^i_{i-1} \big(v \comp^i_{i-1}
      u\big)}
    \or
      \big(\big(x \comp^i_{i-1} w\big) \comp^i_{i-1}  v\big) \comp^i_{i-1}  u
    \or
      \big(x \comp^i_{i-1} \big(w \comp^i_{i-1}  v\big)\big) \comp^i_{i-1}  u
    \or
      x \comp^i_{i-1} \big(\big(w \comp^i_{i-1}  v\big) \comp^i_{i-1}
      u\big)
    \or
      x \comp^i_{i-1} \big(w \comp^i_{i-1} \big(v \comp^i_{i-1}
      u\big)\big)\pbox{.}
    \fi
  }%
\[
\begin{xy}/r7pc/:
\xypolygon5{~>{}~:{(1.9,0):(0,.5)::}
\txt{\ \ \strut\ensuremath{\mynodes}}}
\ar "2";"3"_{a_{x,w,v}\comp^{i+1}_{i-1}\Glk{i}(u)\;\quad}
\ar "3";"4"_{a_{x,w\comp^i_{i-1}v,u}} 
\ar "4";"5"_{\Glk{i}(x)\comp^{i+1}_{i-1}a_{w,v,u}}
\ar "2";"1"^{a_{x\comp^i_{i-1}w,v,u}}
\ar "1";"5"^{a_{x,w,v\comp^i_{i-1}u}}
\ar@{} "3";"1"|{\displaystyle\Rtarrow\limits^{ML_{x,w,v,u}}}
\end{xy}
\]

\item \emph{Codimension $1$ exchange constraints} \\
Let $i \ge 2$. Consider the pair
\[
\Big(
\big(\Thn[i-2]{i} \amalg^{}_{\Thn{i-1}} \Thn[i-2]{i}\big)\Thn{i},
\big(\Thn{i} \amalgd{i-2} \Thn{i}\big)\Thn[i-2]{i}
\Big) :
\Dn{i} \to \Dn{i} \amalgd{i-1} \Dn{i} \amalgd{i-2} \Dn{i} \amalgd{i-1} \Dn{i}
\]
of morphisms of $C_2$. (Note that the left morphism is not a
globular sum.  Nevertheless, it is not hard to prove that this sum exists
(its existence also follows from the fact that it is a generalized globular
sum in the sense of Section 2.5 of \cite{AraThesis}).
We claim that this pair is admissible. Indeed,
if $\ceps{1}, \dots, \ceps{4} :
\Dn{i} \to \Dn{i} \amalgd{i-1} \Dn{i} \amalgd{i-2} \Dn{i} \amalgd{i-1} \Dn{i}$
and $\cepsp{1}, \cepsp{2} : \Dn{i} \to \Dn{i} \amalgd{i-2} \Dn{i}$
denote the canonical morphisms, then we have
\[
\begin{split}
\big(\Thn[i-2]{i} \amalg^{}_{\Thn{i-1}} \Thn[i-2]{i}\big)\Thn{i}\Ths{i}
& =
\big((\ceps{1}, \ceps{3})\Thn[i-2]{i}, (\ceps{2},
\ceps{4})\Thn[i-2]{i}\big)\Thn{i}\Ths{i} \\
& =
(\ceps{2}, \ceps{4})\Thn[i-2]{i}\Ths{i} \\
& =
(\ceps{2}, \ceps{4})(\cepsp{1}\Ths{i}, \cepsp{2}\Ths{i})\Thn{i-1} \\
& = 
\big(\ceps{2}\Ths{i}, \ceps{4}\Ths{i}\big)\Thn{i-1} \\
\end{split}
\]
and
\[
\begin{split}
\big(\Thn{i} \amalgd{i-2} \Thn{i}\big)\Thn[i-2]{i}\Ths{i} 
& =
\big(\Thn{i} \amalgd{i-2} \Thn{i}\big)\big(\cepsp{1}\Ths{i}, \cepsp{2}\Ths{i}\big)\Thn{i-1}\\
& =
\big( (\ceps{1}, \ceps{2})\Thn{i}\Ths{i}, (\ceps{3}, \ceps{4})\Thn{i}\Ths{i}\big)\Thn{i-1} \\
& = 
\big(\ceps{2}\Ths{i}, \ceps{4}\Ths{i}\big)\Thn{i-1}.
\end{split}
\]
A similar calculation shows that
\[
\big(\Thn[i-2]{i} \amalg^{}_{\Thn{i-1}} \Thn[i-2]{i}\big)\Thn{i}\Tht{i}
=
\big(\ceps{1}\Tht{i}, \ceps{3}\Tht{i}\big)\Thn{i-1}
=
\big(\Thn{i} \amalgd{i-2} \Thn{i}\big)\Thn[i-2]{i}\Tht{i}.
\]
Hence there exists a lifting
\[ \nvepsilon_i : \Dn{i+1} \to 
\Dn{i} \amalgd{i-1} \Dn{i} \amalgd{i-2} \Dn{i} \amalgd{i-1} \Dn{i} \]
in $C_3$ such that
\[
 \nvepsilon_i\Ths{i+1} =
\big(\Thn[i-2]{i} \amalg^{}_{\Thn{i-1}} \Thn[i-2]{i}\big)\Thn{i}
 \quad\text{and}\quad
 \nvepsilon_i\Tht{i+1} =
\big(\Thn{i} \amalgd{i-2} \Thn{i}\big)\Thn[i-2]{i}.
\]
This morphism induces a structural map of $G$ for exchange constraints
for compositions $\comp^{i-1}_{i-2}$, $\comp^i_{i-1}$ and
$\comp^i_{i-2}$. This maps sends $i$-arrows
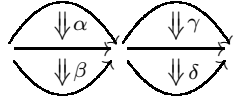
\[
\UseAllTwocells
\xymatrix@C=3pc{
\ruppertwocell{\,\alpha}
\rlowertwocell{\,\beta}
\ar[r]
&
\ruppertwocell{\,\gamma}
\rlowertwocell{\,\delta}
\ar[r]
&
}
\]
to an $(i+1)$-arrow
\[
  e_{\gamma,\delta,\beta.\alpha} :
\big(\delta \comp^i_{i-2} \beta\big)\comp^i_{i-1} \big(\gamma \comp^i_{i-2}
\alpha\big)
\to
\big(\delta \comp^i_{i-1} \gamma\big)\comp^i_{i-2} \big(\beta \comp^i_{i-1}
\alpha\big).
\]
This shows that the pair of compositions $(\comp^i_{i-2}, \comp^i_{i-1})$
satisfies the exchange law up to $(i+1)$-arrows.

\item \emph{Codimension $1$ triangle constraints} \\
Let $i \ge 1$. Denote by $\ceps{1}, \ceps{2} : \Dn{i} \to \Dn{i}
\amalgd{i-1} \Dn{i}$ the canonical morphisms. Let
$d_2 : \Dn{i+1} \to \Dn{i} \amalgd{i-1} \Dn{i}$
be the morphism
\[
\big(
(\ceps{1}\Thk{i}, \ceps{2}\lambda^{}_i)\Thn[i-1]{i+1},
(\ceps{1}, \ceps{1}\Ths{i}\Thk{i-1}, \ceps{2})\alpha^{}_i
\big)\Thn{i+1}
\]
and let
$d_1 : \Dn{i+1} \to \Dn{i} \amalgd{i-1} \Dn{i}$
be the morphism
\[
(\ceps{1}\rho^{}_i, \ceps{2}\Thk{i})\Thn[i-1]{i+1}.
\]
The pair
\[
 (d_2, d_1) : \Dn{i+1} \to \Dn{i} \amalgd{i-1} \Dn{i}
\]
of morphisms of $C_2$ is admissible. Indeed, we have
\[
\begin{split}
d_2\Ths{i+1} & =
(\ceps{1}, \ceps{1}\Ths{i}\Thk{i-1},
\ceps{2})\alpha^{}_i\Ths{i+1} \\
& =
(\ceps{1}, \ceps{1}\Ths{i}\Thk{i-1},
\ceps{2})(\Thn{i} \amalgd{i-1} \Dn{i})\Thn{i} \\
& =
\big((\ceps{1}, \ceps{1}\Ths{i}\Thk{i-1})\Thn{i},
\ceps{2}\big)\Thn{i} \\
& =
\big(\ceps{1}(\id{\Dn{i}}, \Ths{i}\Thk{i-1})\Thn{i}),
\ceps{2}\big)\Thn{i} \\
\end{split}
\]
and
\[
\begin{split}
d_1\Ths{i+1} & =
(\ceps{1}\rho^{}_i\Ths{i+1}, \ceps{2}\Thk{i}\Ths{i+1})\Thn{i} \\
& =
\big(\ceps{1}(\id{\Dn{i}}, \Ths{i}\Thk{i-1})\Thn{i}, \ceps{2}\big)\Thn{i}.
\end{split}
\]
Similarly, we have
\[
\begin{split}
d_2\Tht{i+1} & =
(\ceps{1}\Thk{i}, \ceps{2}\lambda^{}_i)\Thn[i-1]{i+1}\Tht{i+1} \\
& =
(\ceps{1}\Thk{i}\Tht{i+1},
\ceps{2}\lambda^{}_i\Tht{i+1})\Thn{i} \\
& =
(\ceps{1}, \ceps{2})\Thn{i} \\
& = \Thn{i}
\end{split}
\]
and
\[
\begin{split}
d_1\Tht{i+1} & =
(\ceps{1}\rho^{}_i\Tht{i+1}, \ceps{2}\Thk{i}\Tht{i+1})\Thn{i} \\
& =
(\ceps{1}, \ceps{2})\Thn{i} \\
& = \Thn{i}.
\end{split}
\]
Hence there exists a lifting
\[ \nu^{}_i : \Dn{i+2} \to \Dn{i} \amalgd{i-1} \Dn{i} \]
in $C_3$ such that
\[ \nu^{}_i\Ths{i+2} = d_2
\quad\text{and}\quad
\nu^{}_i\Tht{i+2} = d_1.
\]
This morphism induces a structural map of $G$ for triangle constraints for
compositions $\comp^i_{i-1}$, $\comp^{i+1}_i$ and $\comp^{i+1}_{i-1}$.
This map sends $i$-arrows 
\[ x \xrightarrow{u} y \xrightarrow{v} z \]
to an $(i+2)$-arrow
\[
  \xymatrix{
  \big(v \comp^i_{i-1} \Glk{i-1}(y) \big) \comp^i_{i-1} u
  \ar[dr]|{}="a"_{r^{}_v \comp^{i+1}_{i-1} \Glk{i}(u)\quad}
  \ar[rr]^{a_{v,\Glk{i-1}(y), u}} & &
  v \comp^i_{i-1} \big(\Glk{i-1}(y) \comp^i_{i-1} u\big)
   \ar[dl]|{}="b"^{\quad\Glk{i}(v) \comp^{i+1}_{i-1} l^{}_u}
  \ar@{}"a";"b"|\Rtarrow^{T_{v,u}} \\
      & v \comp^i_{i-1} u \pbox{.}
  }
\]
\end{myitemize}

One can show in a similar way that there exists in $C_3$ morphisms
corresponding to the following structural maps:
\begin{itemize}
\item codimension $3$ compositions;
\item codimension $3$ inverses;
\item codimension $2$ associativity constraints;
\item codimension $2$ unit constraints;
\item codimension $2$ inverse constraints.
\end{itemize}
\end{tparagr}

\begin{rem}
The structural maps we have defined can be used to truncate $G$ to a
bicategory in which every arrow is weakly invertible (see Remark
\ref{rem:trunc} for more details).
\end{rem}

\begin{tparagr}{Examples of morphisms appearing in higher $C_n$'s}
One can show that there exists in $C_4$ morphisms corresponding to the
following structural maps:
\begin{itemize}
\item codimension $4$ compositions;
\item codimension $4$ inverses;
\item codimension $3$ associativity constraints;
\item codimension $3$ unit constraints;
\item codimension $3$ inverse constraints;
\item codimension $2$ Mac Lane's pentagon constraints;
\item codimension $2$ exchange constraints;
\item codimension $2$ triangle constraints;
\item codimension $1$ constraints on constraints appearing in $C_3$ (i.e.,
axioms for tricategories).
\end{itemize}

In general, in $C_n$ we have
\begin{itemize}
\item codimension $n$ compositions;
\item codimension $n$ inverses;
\item for every $k$ such that $1 \le k < n$, codimension $n - k$ constraints
on constraints appearing in $C_k$.
\end{itemize}
\end{tparagr}

\begin{tparagr}{Existence of pregroupoidal structures}
\label{paragr:pregroupoid}
Let $C$ be a contractible globular extension. From our previous analysis,
one easily obtains that $C$ can be endowed (in a non-canonical way) with the
structure of a \ndef{pregroupoidal globular extension} in the sense of
\cite{AraThtld}, that is, with morphisms
\[
\begin{split}
\Thn[j]{i} & : \Dn{i} \to \Dn{i} \amalgd{j} \Dn{i},\qquad i > j \ge 0, \\
\Thk{i} & : \Dn{i+1} \to \Dn{i},\qquad i \ge 0,\\
\Thw[j]{i} & : \Dn{i} \to \Dn{i}, \qquad i > j \ge 0,
\end{split}
\]
such that
\begin{enumerate}
\item\label{item:sbg_n} for every $i, j$ such that $i > j \ge 0$, we have
\[ 
\Thn[j]{i}\Ths{i} =
\begin{cases}
    \ceps{2}\Ths{i}, & j = i - 1, \\
    \big(\Ths{i} \amalgd{j} \Ths{i}\big)\Thn[j]{i-1}
    & j < i - 1,
\end{cases}
\]
and
\[
\Thn[j]{i}\Tht{i} =
\begin{cases}
    \ceps{1}\Tht{i}, & j = i - 1, \\
    \big(\Tht{i} \amalgd{j} \Tht{i}\big)\Thn[j]{i-1}
    & j < i - 1,
\end{cases}
\]
where $\ceps{1},\ceps{2} : \Dn{i} \to \Dn{i}\amalgd{i-1} \Dn{i}$
denote the canonical morphisms;
\item\label{item:sbg_k} for every $i \ge 0$, we have
\[ \Thk{i}\Ths{i+1} = \id{\Dn{i}} \quad\text{and}\quad \Thk{i}\Tht{i+1} =
\id{\Dn{i}}\text{;} \]
\item 
for every $i, j$ such that $i > j \ge 0$, we have
\[
\Thw[j]{i}\Ths{i} =
\begin{cases}
\Tht{i} & j = i - 1,\\
\Ths{i}\Thw[j]{i-1} & j < i - 1,
\end{cases}
\]
and
\[
\Thw[j]{i}\Tht{i} =
\begin{cases}
\Ths{i} & j = i - 1,\\
\Tht{i}\Thw[j]{i-1} & j < i - 1.
\end{cases}
\]
\end{enumerate}

Given a pregroupoidal globular extension structure on $C$, any \oo-groupoid
$G$ of type~$C$ is endowed with the structure of an \ndef{\oo-pregroupoids}
in the sense of \cite{AraThtld}, that is, with maps
\[
  \begin{split}
  \comp_j^i & : G_i \times_{G_j} G_i \to 
      G_i,\quad i > j \ge 0,
      \\
  \Glk{i} & : G_i \to G_{i+1}, \quad i \ge 0,\\
  \Glw[j]{i} & : G_i \to G_i, \quad i > j \ge 0,
  \end{split}
\]
such that
\begin{enumerate}
    \item 
      for every 
      $(v, u)$ in $G_i \times_{G_j} G_i$ with
      $i > j \ge 0$, we have
  \[
  \Gls{i}\big(v \comp_j^i u\big) = 
  \begin{cases}
    \Gls{i}(u), & j = i - 1, \\
    \Gls{i}(v) \comp_j^{i-1} \Gls{i}(u), & j < i - 1,
  \end{cases}
  \]
  and
\[
  \Glt{i}(v \comp_j^i u) = 
  \begin{cases}
    \Glt{i}(v), & j = i - 1, \\
    \Glt{i}(v) \comp_j^{i-1} \Glt{i}(u), & j < i - 1;
  \end{cases}
  \]
  \item for every $u$ in $G_i$ with $i \ge 0$, we have
  \[
  \Gls{i+1}\Glk{i}(u) = u = \Glt{i+1}\Glk{i}(u)\text{;}
  \] 
\item
for every $u$ in $G_i$ for $ i \ge 1$ and $j$ such that $i > j \ge
0$, we have
\[
\Gls{i}(\Glw[j]{i}(u)) = 
\begin{cases}
\Glt{i}(u), & j = i - 1,\\
\Glw[j]{i-1}(\Gls{i}(u)), & j < i - 1,
\end{cases}
\]
and
\[
\Glt{i}(\Glw[j]{i}(u)) = 
\begin{cases}
\Gls{i}(u), & j = i - 1,\\
\Glw[j]{i-1}(\Glt{i}(u)), & j < i - 1.
\end{cases}
\]
\end{enumerate}
For $i \ge j \ge 0$, we will denote by $\Glk[i]{j}$ the map from $X_j \to
X_i$ defined by
  \[ \Glk[i]{j} = \Glk{i-1}\cdots\Glk{j+1}\Glk{j}. \]
\end{tparagr}

\section{Weak equivalences of $\infty$-groupoids}

\begin{paragr}
In this section, we fix a contractible globular extension $C$ and an
\oo-groupoid $G$ of type $C$. Moreover, we \emph{choose} once and for all a
pregroupoidal globular extension structure on $C$. The \oo-groupoid $G$ is
thus endowed with the structure of an \oo-pregroupoid. We will use the same
notation for the pregroupoidal structure on $C$ and the structure of
\oo-pregroupoid on $G$ as in paragraph \ref{paragr:pregroupoid}.
\end{paragr}

\begin{tparagr}{Homotopy relation between $n$-arrows}
Let $u$ and $v$ be two $n$-arrows, $n \ge 0$, of $G$. A \ndef{homotopy} from
$u$ to $v$ is an $(n+1)$\nobreakdash-arrow from $u$ to $v$. If such a homotopy
exists, we will say that $u$ is homotopic to $v$ and we will write $u
\tildeh[n] v$.  Note that if $u$ is homotopic to $v$, then $u$ and $v$ are
parallel.
\end{tparagr}

\begin{lemma}
For every $n \ge 0$, the relation $\tildeh[n]$ is an equivalence relation.
Moreover, if $n \ge 1$, this relation is compatible with the composition
$\comp^n_{n-1}$.
\end{lemma}

\begin{proof}
Let $u$ be an $n$-arrow of $G$. The $(n+1)$-arrow $\Glk{n}(u)$ is a homotopy
from $u$ to $u$. The relation $\tildeh[n]$ is hence reflexive.

Let now $v$ be a second $n$-arrow of $G$ and let $h : u \to v$ be a
homotopy. The $(n+1)$-arrow $\Glw{n+1}(h)$ is a homotopy from $v$ to $u$. 
The relation $\tildeh[n]$ is hence symmetric.

Suppose now $w$ is a third $n$-arrow of $G$ and $k : v \to w$ is a second homotopy.
Then the $(n+1)$-arrow $k \comp^{n+1}_n h$ is a homotopy from $u$ to $w$.
The relation $\tildeh[n]$ is hence transitive.

Finally, suppose we have a diagram
\[
\UseAllTwocells
\xymatrix@C=3pc{
\rtwocell^u_{u'}{\,h}
&
\rtwocell^v_{v'}{\,k}
&
}
\]
in $G$, where single arrows are $n$-arrows with $n \ge 1$, and double arrows
are $(n+1)$-arrows. The $(n+1)$-arrow $k \comp^{n+1}_{n-1} h$ is a homotopy
from $v \comp^{n+1}_n u$ to $v' \comp^{n+1}_n u'$. The relation $\tildeh[n]$
is hence compatible with the composition $\comp^n_{n-1}$.
\end{proof}

\begin{tparagr}{The groupoid $\PI_n(G)$}
For $n \ge 0$, we will denote by $\overline{G_n}$ the quotient of $G_n$ by
the equivalence relation $\tildeh[n]$.

Let us now fix $n \ge 1$. The maps
\[
 \Gls{n}, \Glt{n} : G_n \to G_{n-1},
 \quad
 \Glk{n-1} : G_{n-1} \to G_n,
\]
induce maps
\[
 \Gls{n}, \Glt{n} : \overline{G_n} \to G_{n-1},
 \quad
 \Glk{n-1} : G_{n-1} \to \overline{G_n}.
\]
Moreover, by the previous lemma, the map
\[ \comp^n_{n-1} : G_n \times_{G_{n-1}} G_n \to G_n \]
induces a map
\[
 \comp^n_{n-1} : \overline{G_n} \times_{G_{n-1}} \overline{G_n} \to
 \overline{G_n}.
\]
We will denote by $\PI_n(G)$ the graph
\[
\xymatrix{
\overline{G_n} \ar@<.6ex>[r]^-{\Gls{n}} \ar@<-.6ex>[r]_-{\Glt{n}} &
G_{n-1},
}
\]
endowed with the maps
\[
 \comp^n_{n-1} : \overline{G_n} \times_{G_{n-1}} \overline{G_n} \to
\overline{G_n}
 \quad\text{and}\quad
 \Glk{n-1} : G_{n-1} \to \overline{G_n}.
\]
\end{tparagr}

\begin{prop}
For every $n \ge 1$, $\PI_n(G)$ is a groupoid.
\end{prop}

\begin{proof}
Let $u$, $v$ and $w$ be three $(n+1)$-arrows of $G$, composable in
codimension~$1$. Choose, as in paragraph \ref{paragr:coh_cat}, a morphism
\[
 \alpha^{}_n : \Dn{n+1} \to \Dn{n}\amalgd{n-1}\Dn{n}\amalgd{n-1}\Dn{n}
\]
of $C$ such that
\[
\alpha^{}_n\Ths{n+1} = 
\big(\Thn{n}\amalgd{n-1} \Dn{n}\big)\Thn{n}
\quad\text{and}\quad
\alpha^{}_n\Tht{n+1} = 
\big(\Dn{n} \amalgd{n-1} \Thn{n}\big)\Thn{n}.
\]
This morphism induces an $(n+1)$-arrow
\[
a^{}_{w,v,u} : \big(w \comp^n_{n-1} v\big) \comp^n_{n-1} u \to
w \comp^n_{n-1} \big(v \comp^n_{n-1} u\big),
\]
thereby proving the associativity of the composition $\comp^n_{n-1}$ up to
homotopy.

Let now $u : x \to y$ be an $n$-arrow.
Choose, as in paragraph \ref{paragr:coh_cat}, morphisms
\[
\lambda^{}_n : \Dn{n+1} \to \Dn{n}
\quad\text{and}\quad
\rho^{}_n : \Dn{n+1} \to \Dn{n},
\]
of $C$ such that
\[
\begin{split}
\lambda^{}_n\Ths{n+1} = \big(\Tht{n}\Thk{n-1}, \id{\Dn{n}}\big)\Thn{n}
\quad\text{and}\quad &
\lambda^{}_n\Tht{n+1} = \id{\Dn{n}}, \\
\rho^{}_n\Ths{n+1} = \big(\id{\Dn{n}},\Ths{n}\Thk{n-1}\big)\Thn{n}
\quad\text{and}\quad &
\rho^{}_n\Tht{n+1} = \id{\Dn{n}}. \\
\end{split}
\]
These morphisms induce $(n+1)$-arrows
\[
l^{}_u : \Glk{n-1}(y) \comp^n_{n-1} u \to u
\quad\text{and}\quad
r^{}_u : u \comp^n_{n-1} \Glk{n-1}(x) \to u,
\]
thereby proving that $\Glk{n-1}(x)$ is a unit up to homotopy.

Let us now prove that $\Glw{n}(u) : y \to x$ is an inverse of $u$ up to
homotopy. Choose, as in paragraph \ref{paragr:coh_cat}, morphisms
\[
 \gamma^{}_n : \Dn{n+1} \to \Dn{n}
 \quad\text{and}\quad
 \delta^{}_n : \Dn{n+1} \to \Dn{n}
\]
of $C$ such that
\[
\begin{split}
\gamma^{}_n\Ths{n+1} = \big(\Thw{n}, \id{\Dn{n}}\big)\Thn{n}
\quad\text{and}\quad &
\gamma^{}_n\Tht{n+1} = \Ths{n}\Thk{n-1}, \\
\delta^{}_n\Ths{n+1} = \big(\id{\Dn{n}}, \Thw{n}\big)\Thn{n}
\quad\text{and}\quad &
\delta^{}_n\Tht{n+1} = \Tht{n}\Thk{n-1}. \\
\end{split}
\]
These morphisms induce $(n+1)$-arrows
\[
g^{}_u : \Glw{n}(u) \comp^n_{n-1} u \to x
\quad\text{and}\quad
d^{}_u : f \comp^n_{n-1} \Glw{n}(u) \to y,
\]
thus ending the proof.
\end{proof}

\begin{prop}\label{prop:PI_indep}
The groupoid $\PI_n(G)$ does not depend on the choice of a pregroupoidal
globular structure on $C$.
\end{prop}

\begin{proof}
The groupoid $\PI_n(G)$ depends a priori of the choice of
\[ 
\Thn{n} : \Dn{n} \to \Dn{n} \amalgd{n-1} \Dn{n}
\quad\text{and}\quad
\Thk{n-1} : \Dn{n} \to \Dn{n-1}.
\]
Let us show it does not.

Let
\[ 
\Thnp{n} : \Dn{n} \to \Dn{n} \amalgd{n-1} \Dn{n}
\]
be a morphism of $C$ with same globular source and target as $\Thn{n}$. Denote by
\[
\comp^{'n}_{n-1} : G_n \times_{G_{n-1}} G_n \to G_n
\]
the induced composition. Since the pair $(\Thn{n}, \Thnp{n}) : \Dn{n} \to
\Dn{n} \amalgd{n-1} \Dn{n}$ is admissible, there exists a lifting
\[
\mu : \Dn{n+1} \to \Dn{n} \amalgd{n-1} \Dn{n}
\]
in $C$ such that
\[
\mu\Ths{n+1} = \Thn{n}
\quad\text{and}\quad
\mu\Tht{n+1} = \Thnp{n}.
\]
If $u$ and $v$ are two $n$-arrows of $G$, composable in codimension
$1$, then $\mu$ induces an $(n+1)$-arrow
\[ m^{}_{v,u} :  v \comp^n_{n-1} u \to v \comp^{'n}_{n-1} u. \]
Hence the independence from $\Thn{n}$.

In the same way, if
$\Thkp{n-1} : \Dn{n} \to \Dn{n-1}$
is a morphism of $C$ with same globular source and target as $\Thk{n-1}$,
then the admissible pair $(\Thk{n-1}, \Thkp{n-1}) : \Dn{n} \to \Dn{n-1}$
admits a lifting from which we immediately get the independence from
$\Thk{n-1}$.
\end{proof}

\begin{rem}\label{rem:trunc}
For any $n \ge 1$, $\PI_n(G)$ is actually the truncation in dimension $1$ of
a bigroupoid $\PI^2_n(G)$, i.e., of a bicategory in which every $1$-arrow
is invertible up to a $2$\nobreakdash-arrow and every $2$-arrow is
invertible.  Let us briefly explain how to define this bigroupoid. The
underlying $2$-graph of $\PI^2_n(G)$ is
\[
\xymatrix{
\overline{G_{n+1}} \ar@<.6ex>[r]^-{\Gls{n+1}}
\ar@<-.6ex>[r]_-{\Glt{n+1}} & G_n \ar@<.6ex>[r]^-{\Gls{n}}
\ar@<-.6ex>[r]_-{\Glt{n}} & G_{n-1},
}
\]
and its structure maps are induced by a choice of maps $\Glk{n-1}$,
$\Glk{n}$, $\comp^n_{n-1}$, $\comp^{n+1}_{n-1}$, $\comp^{n+1}_n$, $a_n$,
$l_n$ and $r_n$ as in the previous section.
The axioms of bigroupoids follow from the existence in $C$ of
\begin{itemize}
\item codimension $1$ and $2$ associativity constraints;
\item codimension $1$ and $2$ unit constraints;
\item codimension $1$ pentagon constraints;
\item codimension $1$ triangle constraints;
\item codimension $1$ and $2$ inverse constraints.
\end{itemize}
One can easily show that $\PI^2_n(G)$ does not depend on the choice of the
above maps (up to biequivalence).

The bigroupoid $\PI^2_1(G)$ (resp.~the groupoid $\PI_1(G)$) can be thought
of as a truncation of $G$ in dimension $2$ (resp.~in dimension $1$).

The author has no doubt that a reader more patient then him could check that
a similar construction gives rise to a tricategory in the sense of
\cite{GPSTricat} in which every arrow is weakly invertible.
\end{rem}

\begin{tparagr}{The functor $\PI_n$}
Let $f : G \to H$ be a morphism of \oo-groupoids of type $C$. Such a
morphism induces a morphism of globular sets between the underlying globular
sets and in particular respects the notion of homotopy between $n$-arrows.
It follows that for any $n \ge 1$, $f$ induces a morphism of graphs
\[ \PI_n(f) : \PI_n(G) \to \PI_n(H). \]
The naturality of $f$ implies that $\PI_n(f)$ is actually a functor. We thus
get a functor
\[ \PI_n : \wgpdC{C} \to \Gpd,\]
where $\Gpd$ denotes the category of groupoids.
\end{tparagr}

\begin{tparagr}{Inverse image of globular functors}
If $F : D \to D'$ is a morphism of globular extensions, then the
precomposition by~$F$ defines a functor $\Mod{D'} \to \Mod{D}$ that we will
denote by $F^*$. In particular, if $D$ and~$D'$ are contractible, we get a
functor $F^* : \wgpdC{D'} \to \wgpdC{D}$. Since $F$ is a functor under $\G$,
the underlying globular sets of an \oo-groupoid $G$ of type $D'$ and
of~$F^\ast(G)$ coincide.
\end{tparagr}

\begin{prop}
Let $F: D \to D'$ be a morphism of contractible globular extensions. Then for
any $n \ge 1$, the triangle
\[
\xymatrix@C=1pc@R=1pc{
\wgpdC{D'} \ar[rr]^{F^*} \ar[dr]_{\PI_n} & & \wgpdC{D} \ar[dl]^{\PI_n} \\
& \Gpd \\
}
\]
commutes.
\end{prop}

\begin{proof}
Let $G$ be an \oo-groupoid of type $D'$. Since $G$ and $F^*(G)$ have the
same underlying globular set, the underlying graphs of $\PI_n(G)$ and
$\PI_n(F^*(G))$ coincide. The groupoid structure of $\PI_n(F^*(G))$
is induced by a choice of morphisms
\[
\begin{split}
\Thn{n} & : \Dn{n} \to \Dn{n} \amalgd{n-1} \Dn{n}, \\
\Thk{n-1} & : \Dn{n} \to \Dn{n-1} \\
\end{split}
\]
of $D$ such that
\[
\begin{split}
\Thn{n}\Ths{n} = \ceps{2}\Ths{n}
\quad\text{and}\quad &
\Thn{n}\Tht{n} = \ceps{1}\Tht{n}, \\
\Thk{n-1}\Ths{n} = \id{\Dn{n-1}}
\quad\text{and}\quad &
\Thk{n-1}\Tht{n} = \id{\Dn{n-1}},
\end{split}
 \]
where $\ceps{1},\ceps{2} : \Dn{n} \to \Dn{n}\amalgd{n-1} \Dn{n}$
denote the canonical morphisms.
By applying $F$ to these morphisms, we get morphisms 
\[
\begin{split}
F(\Thn{n}) & : \Dn{n} \to \Dn{n} \amalgd{n-1} \Dn{n}, \\
F(\Thk{n-1}) & : \Dn{n} \to \Dn{n-1} \\
\end{split}
\]
of $D'$ such that
\[
\begin{split}
F(\Thn{n})\Ths{n} = \ceps{2}\Ths{n}
\quad\text{and}\quad &
F(\Thn{n})\Tht{n} = \ceps{1}\Tht{n}, \\
F(\Thk{n-1})\Ths{n} = \id{\Dn{n-1}}
\quad\text{and}\quad &
F(\Thk{n-1})\Tht{n} = \id{\Dn{n-1}}.
\end{split}
 \]
These morphisms can be used to define the groupoid structure of $\PI_n(G)$.
(Note that we have neglected some inoffensive canonical isomorphisms when
describing the target of $F(\Thn{n})$ and $F(\Thk{n-1})$.) But it is clear
that these two choices lead to the same groupoid structure.
\end{proof}

\begin{tparagr}{Homotopy groups of $\infty$-groupoids}
We define the set $\pi_0(G)$ of \ndef{connected components} of $G$ by
\[ \pi_0(G) = \pi_0(\PI_1(G)) = \overline{G_0}. \]

Let $n \ge 1$ and let $u, v$ be two parallel $(n-1)$-arrows of $G$. We will denote by
$\Hom_G(u, v)$ the set of $n$-arrows of $G$ from $u$ to $v$. We set
\[
\pi_n(G, u, v) = \Hom_{\PI_n(G)}(u, v)
\quad\text{and}\quad
\pi_n(G, u) = \pi_n(G, u, u).
\]
The set $\pi_n(G, u, v)$ is nothing but the quotient of $\Hom_G(u, v)$ by
the equivalence relation~$\sim_n$. Note that $\pi_n(G, u)$ is canonically
endowed with a group structure.

If $n \ge 1$ and $x$ is an object of $G$, we define the \ndef{$n$-th
homotopy group} of $(G, x)$ as
\[ \pi_n(G, x) = \pi_n(\Glk[n-1]{0}(x)). \]
The Eckmann-Hilton argument shows that for $n \ge 2$, the group
$\pi_n(G, x)$ is abelian.

From the fact that for every $n \ge 1$, $\PI_n$ is a functor from the
category of \oo-groupoids of type $C$ to groupoids, we get that
\begin{itemize}
\item $\pi_0$ is a functor from the category of \oo-groupoids of type $C$ to
the category of sets;
\item for all $n \ge 1$, $\pi_n$ is a functor from the category of
\oo-groupoids of type $C$ endowed with an $(n-1)$-arrow (or with an object)
to the category of groups.
\end{itemize}
\end{tparagr}

\begin{lemma}[Division lemma]\label{lemma:div}
Let $n \ge 2$ and let $i$ be an integer such that $0 \le i < n - 1$.
Let $u, v$ be a pair of parallel $(n-1)$-arrows of $G$ and let $\gamma : u'
\to v'$ be an $n$-arrow such that
\[
 \Gls[i]{n}(\gamma) = \Glt[i]{n-1}(u) = \Glt[i]{n-1}(v).
\]
Then the map
\[
\begin{split}
\Hom_G(u,v) & \to \Hom_G(u' \comp^{n-1}_i u, v' \comp^{n-1}_i v) \\
\alpha & \mapsto \gamma \comp^n_i \alpha
\end{split}
\]
induces a natural bijection
\[
\pi_n(G, u, v) \to \pi_n(G, u' \comp^{n-1}_i u, v' \comp^{n-1}_i v).
\]
\end{lemma}

\begin{proof}
We will denote by $K$ the map
\[
\begin{split}
\Hom_G(u,v) & \to \Hom_G(u' \comp^{n-1}_i u, v' \comp^{n-1}_i v) \\
\alpha & \mapsto \gamma \comp^n_i \alpha \pbox{.}
\end{split}
\]
If $\Lambda : \alpha \to \alpha'$ is a homotopy between
two $n$-arrows $\alpha, \alpha' : u \to v$, then
the $(n+1)$-arrow
\[ \Glk{n}(\gamma) \comp^{n+1}_i \Lambda :
 \gamma \comp^n_i \alpha \to \gamma \comp^n_i \alpha' \]
is a homotopy from $K(\alpha)$ to $K(\alpha')$. Hence the map $K$ induces a
map
\[
\overline{K} : \pi_n(G, u, v) \to \pi_n(G, u' \comp^{n-1}_i u, 
v' \comp^{n-1}_i v).
\]

We will construct a map
\[
L : \Hom_G(G, u' \comp^{n-1}_i u, v' \comp^{n-1}_i v) \to \Hom_G(u, v)
\]
inducing an inverse
\[
\overline{L} : \pi_n(G, u' \comp^{n-1}_i u, v' \comp^{n-1}_i v) \to \pi_n(G,
u, v)
\]
of $\overline{K}$. 

Our proof will be quite technical. For this reason, we start by
giving an idea of it. Let $\beta : u' \comp^{n-1}_i u \to v'
\comp^{n-1}_i v$. One could naively think that
$L'(\beta) = \Glw[i]{n}(\gamma) \comp^{n}_i \beta$ would induce an inverse
(that would be true if $G$ were a strict \oo-groupoid).
But the source (resp.~the target) of $L'(\beta)$ is
\[  \Glw[i]{n-1}(u') \comp^{n-1}_i \big(u' \comp^{n-1}_i u)\qquad
\text{(resp.~ $\Glw[i]{n-1}(v') \comp^{n-1}_i \big(v' \comp^{n-1}_i v)$).}
\]
In particular, $L'(\beta)$ does not belong to $\Hom_G(u, v)$.  Nevertheless,
the source (resp.~the target) of $L'(\beta)$ and $u$ (resp.~of $L'(\beta)$
and $v$) coincide in dimension~$i$, that is, we have
\[
\Gls[i]{n}(L'(\beta)) = \Gls[i]{n-1}(u)
\quad\text{and}\quad
\Glt[i]{n}(L'(\beta)) = \Glt[i]{n-1}(v).
\]
We will ``correct'' the source and target of $L'(\beta)$ dimension
by dimension. We will first construct $(i+2)$-arrows $c_{i+2}$ and $d_{i+2}$
such that the $n$-arrow
\[
\alpha_{i+2} =
\Glk[n]{i+2}(d_{i+2}) \comp^{n}_{i+1}
\big(
L'(\beta)
\comp^{n}_{i+1}
\Glk[n]{i+2}(c_{i+2})
\big)
\]
has the same source (resp.~the same target) as $u$
(resp.~as $v$) in dimension~$i+1$. 
By induction,
we will define $\alpha_j$ for $i < j \le n$ such that the source (resp.~the
target) of $\alpha_j$ and $u$ (resp.~and $v$) coincide in dimension $j-1$
($\alpha_{i+1}$ being $L'(\beta)$).
We will end up by defining $L(\beta)$ as the $n$-arrow
\[
\alpha_n =
d_n \comp^{n}_{n-1} 
\Big[
\Big[
\cdots
\Big(
\Glk[n]{i+2}(d_{i+2}) \comp^{n}_{i+1}
\Big(
L'(\beta)
\comp^{n}_{i+1}
\Glk[n]{i+2}(c_{i+2})
\Big)
\Big)
\cdots
\Big]
\comp^n_{n-1} c_n
\Big]
,
\]
where the $c_j$'s (resp.~the $d_j$'s) are the $j$-arrows ``correcting the
source (resp.~the target) of $L'(\beta)$ in dimension $j-1$.''

Here is how our proof is organized. First, we define the $c_j$'s (resp.~the
$d_j$'s) as functions~$C_j$ (resp.~$D_j$) of $u$ and $u'$ (resp.~of $v$ and
$v'$). To define these functions, we will define by mutual induction (on
$j$) the $C_j$'s (resp.~the $D_j$'s), their sources $C^-_j$ (resp.~$D^-_j$),
their targets $C^+_j$ (resp.~$D^+_j$) and functions $S_j$ (resp.~$T_j$) which
(as will be proved in the fourth step) are the sources (resp.~the targets)
of the $\alpha_j$'s. An important point is that all these functions come
from morphisms of $C$. This will allow us to get liftings from~$C$. Second,
we show that the pairs of morphisms of $C$ inducing the pairs
$(C^-_j, C^+_j)$ and $(D^-_j, D^+_j)$ are admissible. This is actually
needed by the induction step of the first point. Third, we define the
$\alpha_j$'s and we prove that their sources (resp.~their targets) are given by
the~$S_j$'s (resp.~the $T_j$'s). Fourth, we define $L$ and $\overline{L}$.
Fifth, we show that $\overline{L}$ is an inverse of~$\overline{K}$.

\medskip
\noindent1. \emph{Definition of the $C_j$'s and the $D_j$'s}
\nopagebreak
\smallskip

We define by induction on $j$ such that $i + 2 \le j \le n$ maps
\[
\begin{split}
S_j, T_j & : G_{n-1} \times_{G_i} G_{n-1} \to G_{n-1}, \\
C^-_j, C^+_j, D^-_j, D^+_j & : G_{n-1} \times_{G_i} G_{n-1} \to G_{j-1}, \\
C_j, D_j : & G_{n-1} \times_{G_i} G_{n-1} \to G_j,
\end{split}
\]
induced by morphisms of $C$.

We set
\[
\begin{split}
S_{i+2}(u', u) = & \Glw[i]{n-1}(u') \comp^{n-1}_i (u' \comp^{n-1}_i u), \\
T_{i+2}(v', v) = & \Glw[i]{n-1}(v') \comp^{n-1}_i (v' \comp^{n-1}_i v); \\
\end{split}
\]
for $j$ such that $i + 2 < j \le n$, we set
\[
\begin{split}
S_j(u', u) = & \Glk[n-1]{j-1}D_{j-1}(u', u) \comp^{n-1}_{j-2} \big(S_{j-1}(u',
u) \comp^{n-1}_{j-2} \Glk[n-1]{j-1}C_{j-1}(u', u)\big),\\
T_j(v', v) = & \Glk[n-1]{j-1}D_{j-1}(v', v) \comp^{n-1}_{j-2} \big(T_{j-1}(v',
v) \comp^{n-1}_{j-2} \Glk[n-1]{j-1}C_{j-1}(v', v)\big);
\end{split}
\]
for $j$ such that $i + 2 \le j \le n$, we set
\[
\begin{split}
C^-_j(u', u) = & \Gls[j-1]{n-1}(u), \\
C^+_j(u', u) = & \Gls[j-1]{n-1}S_j(u' ,u),\\
D^-_j(v', v) = & \Gls[j-1]{n-1}T_j(v' ,v),\\
D^+_j(v', v) = & \Glt[j-1]{n-1}(v). \\
\end{split}
\]
For our definition to be complete, we need to define the $C_j$'s and the
$D_j$'s. Let $j$ be such that $i + 2 \le j \le n$. By induction, the maps
\[ C^-_j, C^+_j : G_{n-1} \times_{G_i} G_{n-1} \to G_{j-1} \]
are induced by a pair of morphisms of $C$. We will prove in the second step
of the proof that this pair is admissible. Admitting this fact, we get a
lifting of the pair and so a map
\[
C_j : G_{n-1} \times_{G_i} G_{n-1} \to G_j
\]
such that
\[ 
\Gls{j}C_j = C^-_j
\quad\text{and}\quad
\Glt{j}C_j = C^+_j.
\]
In the same way, we get a map
\[
D_j : G_{n-1} \times_{G_i} G_{n-1} \to G_j
\]
induced by $C$, such that
\[ 
\Gls{j}D_j = D^-_j
\quad\text{and}\quad
\Glt{j}D_j = D^+_j.
\]

\medskip
\noindent2. \emph{The pairs $(C^-_j, C^+_j)$ and $(D^-_j, D^+_j)$ are
admissible}
\smallskip

Let us check that the pair of morphisms of $C$ inducing the pair $(C^-_j,
C^+_j)$ is admissible. It suffices to check that the $(j-1)$-arrows
$C^-_j(u', u)$ and $C^+_j(u', u)$ are parallel for every~$(u', u)$ in
$G_{n-1} \times_{G_i} G_{n-1}$.
(The reader not convinced by this assertion can either extract a direct
proof by dualizing our calculations, or read paragraph 5.3 of
\cite{AraThtld}.) Let~$(u', u)$ be in $G_{n-1} \times_{G_i} G_{n-1}$. For $j
= i + 2$, we have
\[
\begin{split}
 \Gls{i+1}C^+_{i+2}(u', u) & = 
\Gls{i+1}\Gls[i+1]{n-1}S_{i+2}(u', u) \\
 & =
  \Gls[i]{n-1}(\Glw[i]{n-1}(u') \comp^{n-1}_i (u' \comp^{n-1}_i u)) \\
 & = \Gls[i]{n-1}(u) = \Gls{i+1}\Gls[i+1]{n-1}(u) \\
 & = \Gls{i+1}C^-_{i+2}(u', u)
\end{split}
\]
and
\[
\begin{split}
 \Glt{i+1}C^+_{i+2}(u', u) & = 
 \Glt{i+1}\Gls[i+1]{n-1}S_j(u', u) \\
 & = \Glt[i]{n-1}(\Glw[i]{n-1}(u')
 \comp^{n-1}_i (u' \comp^{n-1}_i u)) \\
 & = \Glt[i]{n-1}\Glw[i]{n-1}(u') \\
 & = \Gls[i]{n-1}(u') = \Glt[i]{n-1}(u) = \Glt{i+1}\Gls[i+1]{n-1}(u)\\
 & = \Glt{i+1}C^-_{i+2}(u', u);
\end{split}
\]
and for $i + 2 < j \le n$, we have
\[
\begin{split}
\Gls{j-1}C^+_j(u', u) & = \Gls{j-1}\Gls[j-1]{n-1}S_j(u', u) \\
& =
\Gls[j-2]{n-1}
\big(\Glk[n-1]{j-1}D_{j-1}(u', u) \comp^{n-1}_{j-2} \big(S_{j-1}(u',
u) \comp^{n-1}_{j-2} \Glk[n-1]{j-1}C_{j-1}(u', u)\big)\big) \\
& = 
\Gls[j-2]{n-1}\Glk[n-1]{j-1}C_{j-1}(u', u) = \Gls{j-1}C_{j-1}(u', u) \\
& = C^-_{j-1}(u', u) = \Gls[j-2]{n-1}(u) = \Gls{j-1}\Gls[j-1]{n-1}(u) \\
& = \Gls{j-1}C^-_j(u', u)
\end{split}
\]
and
\[
\begin{split}
\Glt{j-1}C^+_j(u', u) & = \Glt{j-1}\Gls[j-1]{n-1}S_j(u', u) \\
& =
\Glt[j-2]{n-1}
\big(\Glk[n-1]{j-1}D_{j-1}(u', u) \comp^{n-1}_{j-2} \big(S_{j-1}(u',
u) \comp^{n-1}_{j-2} \Glk[n-1]{j-1}C_{j-1}(u', u)\big)\big) \\
& = 
\Glt[j-2]{n-1}\Glk[n-1]{j-1}D_{j-1}(u', u) = \Glt{j-1}D_{j-1}(u', u) \\
& = D^+_{j-1}(u', u) = \Glt[j-2]{n-1}(u) = \Glt{j-1}\Gls[j-1]{n-1}(u) \\
& = \Glt{j-1}C^-_j(u', u).
\end{split}
\]

Very similar calculations show that the $(j-1)$-arrows $D^-_j(v', v)$ and $D^+_j(v',
v)$ are parallel for every $(v', v)$ in $G_{n-1} \times_{G_i} G_{n-1}$.

\medskip
\noindent3. \emph{Definition of the $\alpha_j$'s and calculation of their
sources and targets}
\smallskip

Let $\beta : u' \comp^{n-1}_i u \to v' \comp^{n-1}_i v$. We define by
induction on $j$ such that $i <  j \le n$, an $n$-arrow~$\alpha_j$.  For $j
= i+1$, we set
\[ \alpha_{i+1} = \Glw[i]{n}(\gamma) \comp^n_i \beta. \]
For $j > i + 1$, we set
\[ \alpha_j = \Glk[n]{j}(d_j) \comp^n_{j-1} \big(\alpha_{j-1}
\comp^n_{j-1} \Glk[n]{j}(c_j)\big), \]
where
\[
c_j = C_j(u', u)
\quad\text{and}\quad
d_j = D_j(v', v).
\]
To show that our $\alpha_j$'s are well-defined, we have to show that
\[ 
c_j : \Gls[j-1]{n-1}(u) \to \Gls[j-1]{n}(\alpha_{j-1})
\quad\text{and}\quad
d_j : \Glt[j-1]{n}(\alpha_{j-1}) \to \Glt[j-1]{n-1}(v).
\]
We first show by induction on $j$ such that $i + 1 < j \le n$ that we have
\[ \Gls{n}(\alpha_{j-1}) = S_j(u, u'). \]
For $j = i + 2$, we have
\[
\begin{split}
 \Gls{n}(\alpha_{i+1}) & = 
 \Gls{n}(\Glw[i]{n}(\gamma) \comp^n_i \beta) \\
 & = \Gls{n}\Glw[i]{n}(\gamma) \comp^{n-1}_i \Gls{n}(\beta) \\
 & = \Glw[i]{n-1}\Gls{n}(\gamma) \comp^{n-1}_i \Gls{n}(\beta) \\
 & = \Glw[i]{n-1}(u') \comp^{n-1}_i \big(u' \comp^{n-1}_i u\big) \\
 & = S_{i+2}(u', u);
\end{split}
\]
and for $j > i + 1$, we have
\[
\begin{split}
\Gls{n}(\alpha_{j-1}) & =
\Gls{n}\big(\Glk[n]{j-1}(d_{j-1}) \comp^n_{j-2} \big(\alpha_{j-2}
\comp^n_{j-2} \Glk[n]{j-1}(c_{j-1})\big)\big)\\
& =
\big(\Glk[n-1]{j-1}(d_{j-1}) \comp^{n-1}_{j-2} \big(\Gls{n}\alpha_{j-2}
\comp^n_{j-2} \Glk[n-1]{j-1}(c_{j-1})\big)\big)\\
& =
\Glk[n-1]{j-1}D_{j-1}(u', u) \comp^{n-1}_{j-2} \big(S_{j-1}(u',
u) \comp^{n-1}_{j-2} \Glk[n-1]{j-1}C_{j-1}(u', u)\big)\\
& =
S_j(u', u).
\end{split}
\]
We hence have
\[
\begin{split}
\Gls{j}(c_j) & = \Gls{j}C_j(u', u) = C^-_j(u', u) \\
& = \Gls[j-1]{n-1}(u)
\end{split}
\]
and
\[
\begin{split}
\Glt{j}(c_j) & = \Glt{j}C_j(u', u) = C^+_j(u', u) \\
& = \Gls[j-1]{n-1}S_j(u', u) = \Gls[j-1]{n-1}\Gls{n}(\alpha_{j-1})\\
& = \Gls[j-1]{n}(\alpha_{j-1}),
\end{split}
\]
i.e.,
\[
c_j : \Gls[j-1]{n-1}(u) \to \Gls[j-1]{n}(\alpha_{j-1}).
\]
Very similar calculations give
\[
d_j : \Glt[j-1]{n}(\alpha_{j-1}) \to \Glt[j-1]{n-1}(v).
\]

\medskip
\noindent4. \emph{Definition of $L$ and $\overline{L}$}
\smallskip

Let $\beta : u' \comp^{n-1}_i u \to v' \comp^{n-1}_i v$ and let $\alpha =
\alpha_n$ be the $n$-arrow defined in the previous step. Explicitly, we have
\[
\alpha
= 
d_n \comp^{n}_{n-1} 
\Big[
\Big[
\cdots
\Big(
\Glk[n]{i+2}(d_{i+2}) \comp^{n}_{i+1}
\Big(
\Big(
\Glw[i]{n}(\gamma) \comp^{n}_i \beta
\Big)
\comp^{n}_{i+1}
\Glk[n]{i+2}(c_{i+2})
\Big)
\Big)
\cdots
\Big]
\comp^n_{n-1} c_n
\Big]
.
\]
Note that 
\[
\begin{split}
\Gls{n}(\alpha) & = \Gls{n}(c_n) = u, \\
\Glt{n}(\alpha) & = \Glt{n}(d_n) = v.
\end{split}
\]
We can thus define the map
\[
L : \Hom_G(G, u' \comp^{n-1}_i u, v' \comp^{n-1}_i v) \to \Hom_G(u, v)
\]
by sending $\beta$ to $\alpha$. The formula defining $\alpha$ is
clearly functorial in $\beta$ and the map $L$ thus induces a map
\[
\overline{L} : \pi_n(G, u' \comp^{n-1}_i u, v' \comp^{n-1}_i v)
\to \pi_n(G, u, v).
\]

\medskip
\noindent5. \emph{The map $\overline{L}$ is an inverse of $\overline{K}$}
\smallskip

We first prove that $\overline{L}$ is a left inverse of $\overline{K}$. 
Consider the maps
\[
\begin{split}
 G_n \times_{G_i} G_n & \to G_n \\
 (\alpha, \gamma) & \mapsto \alpha, \\
 (\alpha, \gamma) & \mapsto LK(\alpha).
\end{split}
\]
They are both induced by a morphism of $C$. Moreover, we already know that
$\alpha$ and~$LK(\alpha)$ are parallel. We thus get from $C$ an
$(n+1)$-arrow from $\alpha$ to $LK(\alpha)$, thereby proving that we have
\[ \overline{L}\,\overline{K} = \id{\pi_n(G, u, v)}. \]

Let us now show that $\overline{L}$ is injective. We have just shown
that $\overline{K}$ is injective. This means that for
every $n$-arrow $\delta$ and all $j$ such that $0
\le j < n - 1$, the right composition by $\delta$ in codimension $n - j$ (i.e., the operation
$\alpha \mapsto \alpha \comp^n_j \delta$) reflects the
property of being homotopic. Dually, the left composition by $\delta$
in codimension $n - j$ (i.e., the operation $\alpha \mapsto \delta
\comp^n_j \alpha$) reflects the property of being homotopic. But the map $L$
is obtained as a composition of such operations. Hence the map $L$ also
reflects the property of being homotopic. This exactly means that
$\overline{L}$ is injective.

It follows that $\overline{L}$ is bijective. Its left inverse $\overline{K}$
is hence an inverse.
\end{proof}

\begin{thm}\label{thm:pi_n_u_x}
Let $n \ge 1$ and let $u$ be an $(n-1)$-arrow of $G$.  Set $x =
\Gls[0]{n-1}(u)$.
There exists an isomorphism 
\[ \pi_n(G, u) \to \pi_n(G, x), \]
natural in $X$.
\end{thm}

\begin{proof}
If $n = 1$, the result is tautological. For $n \ge 2$, by the previous lemma
and its dual, we have the following zig-zag of natural
isomorphisms:
\[
\pi_n(G, u) \overset\sim\to \pi_n(G, u \comp^{n-1}_0 \Glk[n-1]{0}(x)) 
\overset\sim\gets \pi_n(G, \Glk[n-1]{0}(x)) = \pi_n(G, x).
\]
\end{proof}

\begin{coro}
Let $n \ge 1$ and let $u : x \to y$ be a $1$-arrow. Then $u$ induces
an isomorphism
\[ \pi_n(G, x) \to \pi_n(G, y). \]
\end{coro}

\begin{proof}
This follows from the previous theorem and its dual applied to
$\Glk[n-1]{1}(u)$.
\end{proof}

\begin{coro}\label{coro:pi_n_indep}
Let $n \ge 1$ and let $x$ be an object of $G$. The group $\pi_n(G, x)$ does
not depend on the choice of the pregroupoidal globular structure on $C$.
\end{coro}

\begin{proof}
This follows from the previous theorem and Proposition \ref{prop:PI_indep}
(i.e., the analogous result for $\pi_n(G, u)$).
\end{proof}

\begin{rem}
We have only proved that $\pi_n(G, x)$ is unique up to a non-canonical
isomorphism. Indeed, the zig-zag appearing in the proof of Theorem
\ref{thm:pi_n_u_x} depends on~$\comp^{n-1}_0$, and hence on the choice of
$\Thn[0]{n-1}$. Nevertheless, one can show that the isomorphism of the
theorem does not depend on this choice (the proof of this fact is very
similar to the one of Lemma \ref{lemma:div}). We hence obtain that
$\pi_n(G, x)$ is unique up to a canonical isomorphism.
\end{rem}

\begin{tparagr}{Weak equivalences of $\infty$-groupoids}
We will say that a morphism $f : G \to H$ of \oo-groupoids of type $C$ is a
\ndef{weak equivalence} if the following conditions are satisfied:
\begin{itemize}
\item the map $\pi_0(f) : \pi_0(G) \to \pi_0(H)$ is a bijection;
\item for all $n \ge 1$ and every object $x$ of $G$, the morphism
$\pi_n(G, x) \to \pi_n(H, f(x))$, induced by $f$, is an isomorphism.
\end{itemize}
By the previous corollary, this definition does not depend on the choice of
a pregroupoidal globular structure on $C$.
\end{tparagr}

\begin{thm}\label{thm:equiv_we}
Let $f : G \to H$ be a morphism of \oo-groupoids of type $C$. The following
conditions are equivalent:
\begin{enumerate}
\item $f$ is a weak equivalence;
\item the map 
\[ \pi_0(f) : \pi_0(G) \to \pi_0(H) \]
 is a bijection, and for all
$n \ge 1$ and every $(n-1)$-arrow $u$ of $G$, the morphism $f$ induces an
isomorphism of groups
\[ \pi_n(G, u) \to \pi_n(G, f(u)); \]
\item the functor 
\[ \PI_1(f) : \PI_1(G) \to \PI_1(H) \]
is an equivalence of categories, and for every $n \ge 2$ and every pair $u, 
v$ of parallel $(n-1)$-arrows of $G$, the morphism $f$ induces a bijection
\[ \pi_n(G, u, v) \to \pi_n(G, f(u), f(v)); \]
\item the functor 
\[ \PI_1(f) : \PI_1(G) \to \PI_1(H) \]
is full and essentially surjective, and for every $n \ge 2$ and every pair $u, 
v$ of parallel $(n-1)$-arrows of $G$, the morphism $f$ induces a surjection
\[ \pi_n(G, u, v) \to \pi_n(G, f(u), f(v)). \]
\end{enumerate}
\end{thm}

\begin{proof}
The equivalence of $(1)$ and $(2)$ is an immediate consequence of the previous
theorem.

The implication $(3) \Rightarrow (2)$ is obvious. Let us show the reciprocal.
Let $n \ge 1$ and let $u, v$ be two parallel $(n-1)$-arrows of $G$. Suppose
there exists an $n$-arrow $\alpha : u \to v$ in $G$ and consider the map
\[ \pi_n(G, u) \to \pi_n(G, u, v), \]
which sends the $n$-arrow $\beta$ to the $n$-arrow
$\alpha \comp^n_{n-1} \beta : u \to v$.
Since $\PI_n(G)$ is a groupoid, this map is a bijection. The
morphism $f$ obviously commutes with this isomorphism, that is, the square
\[
\xymatrix{
\pi_n(G, u) \ar[d] \ar[r] & \pi_n(G, u, v) \ar[d] \\
\pi_n(H, f(u)))  \ar[r] & \pi_n(H, f(u), f(v)) \\
}
\]
is commutative. By hypothesis, the bottom horizontal arrow is a bijection
and it follows that the top horizontal arrow is also a bijection. Thus, it
suffices to show that if there exists an $n$-arrow $\beta : f(u) \to f(v)$
in $H$, then there exists an $n$-arrow $\alpha : u \to v$ in $G$. This is
obvious when $n = 1$ by injectivity of $\pi_0(f)$. So let $n \ge 2$ and let
$\beta : f(u) \to f(v)$ be an $n$-arrow of $H$. Set $x = \Gls{n-1}(u)$ and
$y = \Glt{n-1}(v)$. The arrow
$\Glk{n-1}(\Glw{n-1}(f(u))) \comp^n_{n-2} \beta$ is an $n$-arrow of $H$ from
$\Glw{n-1}(f(u)) \comp^{n-1}_{n-2} f(u) : f(x) \to f(x)$ to $\Glw{n-1}(f(u))
\comp^{n-1}_{n-2} f(v) : f(x) \to f(x)$. By injectivity of the map
\[ \pi_{n-1}(G, x) \to \pi_{n-1}(H, f(x)), \]
the $(n-1)$-arrows $\Glw{n-1}(u) \comp^{n-1}_{n-2} u$
and $\Glw{n-1}(u) \comp^{n-1}_{n-2} v$ are equal in $\pi_{n-1}(G,x)$. Since
$\PI_{n-1}(G)$ is a groupoid, this implies that $u = v$ in
$\pi_{n-1}(G,x,y)$ and so that there exists an $n$-arrow $\alpha : u \to v$.

The implication $(3) \Rightarrow (4)$ is obvious. Let us show the reciprocal.
Let $n \ge 1$. Let $u,v$ be two parallel $(n-1)$-arrows of $G$ and 
let $\alpha, \beta$ be two $n$-arrows from $u$ to $v$.
Suppose we have $f(\alpha) = f(\beta)$ in $\pi_n(H, f(u), f(v))$. 
By definition, there exists an $(n+1)$-arrow of $H$ from $f(\alpha)$ to
$f(\beta)$. By surjectivity of the map
\[ \pi_{n+1}(G, \alpha, \beta) \to \pi_{n+1}(H, f(\alpha), f(\beta)), \]
there exists an $(n+1)$-arrow of $G$ from $\alpha$ to $\beta$, and we thus
have $\alpha = \beta$ in $\pi_n(G, u, v)$.
\end{proof}

\begin{rem}
In \cite{Simpson}, Simpson proves an analogous result for strict
$n$-categories with weak inverses (see Theorem 2.1.C of op.~cit.). See also
Proposition 1.7 of \cite{AraMetWGpd} for the case of strict \oo-groupoids.
\end{rem}

\section{Fundamental $\infty$-groupoid functors}

\begin{tparagr}{Formal coherators}
A \ndef{formal coherator} is a weakly initial object of the category of
contractible globular extensions whose underlying category is small.
Recall that an object $X$ of a category $\C$ is said to be \ndef{weakly
initial} if for any object $Y$ of $\C$ there exists at least one arrow
from~$X$ to $Y$. Proposition \ref{prop:coh_weak_init} shows that every
coherator is a formal coherator.
\end{tparagr}

In the rest of this section, we fix a formal coherator $C$.

\begin{tparagr}{The fundamental $\infty$-groupoid functor of a
contractible globular \hbox{extension}}\label{par:constr_Pi}
Let $\M$ be a contractible globular extension whose underlying category is
cocomplete. By definition of $C$, there exists a globular functor $F : C \to
\M$. Since $\M$ is cocomplete, this functor induces an
adjunction
\[
R_{C,F} :  \pref{C}  \to  \M, \qquad 
\Pi_{C,F} : \M \to \pref{C},
\]
where $R_{C,F}$ is the unique extension of $F$ to $\pref{C}$ preserving
colimits and $\Pi_{C,F}$ is given by the following formula: for every
object $X$ of $\M$, $\Pi_{C,F}(X) = \big(S \mapsto \Hom_{\M}(F(S), X)\big)$.
Since the functor $F$ is globular, the presheaf $\Pi_{C,F}(X)$ is
globular and the previous adjunction induces an adjunction 
\[
R_{C,F} :  \wgpdC{C} \to  \M, \qquad 
\Pi_{C,F} : \M \to \wgpdC{C},
\]
which we will denote the same way. If $X$ is an object of $\M$, we will
call $\Pi_{C, F}(X)$ the \ndef{fundamental
\oo-groupoid} of $X$. Note that it depends a priori on $F$.
However, the underlying globular set of $\Pi_{C, F}$ depends neither on $F$
nor on $C$: its set of $n$-arrows is given by
\[ \Pi_{C, F}(X)_n = \Hom_{\M}(\Dn{n}, X). \]
\end{tparagr}

\begin{exs}
Let $\M = \Top$. We have seen in Example \ref{exs:contr}.\ref{item:Top_contr}
that the globular extension $\Top$ is contractible. By the previous
paragraph we get an adjunction
\[
R_{C,F} :  \wgpdC{C} \to  \Top, \qquad 
\Pi_{C,F} : \Top \to \wgpdC{C},
\]
which depends on a globular functor $F : C \to \Top$. If $X$ is a
topological space, we have
\[ \Pi_{C, F}(X)_n = \Hom_{\Top}(\Dtop{n}, X). \]
In particular, 
\begin{itemize}
\item 
$\Pi_{C, F}(X)_0$ is the set of points of $X$;
\item 
$\Pi_{C, F}(X)_1$ is the set of paths of $X$;
\item 
$\Pi_{C, F}(X)_2$ is the set of (relative) homotopies between paths of $X$;
\item 
$\Pi_{C, F}(X)_3$ is the set of (relative) homotopies between (relative)
homotopies between paths of $X$;
\item etc.
\end{itemize}
\end{exs}

\begin{paragr}
If $D$ is a globular extension, then the Yoneda functor $D \to \pref{D}$
factors through the category $\Mod{D}$ and we get a functor $D \to \Mod{D}$.
In particular, if $D$ is a contractible globular extension, we get a functor
\[ y^{}_D : D \to \wgpdC{D}. \]
\end{paragr}

\begin{prop}\label{prop:PI_f}
Let $\M$ be a cocomplete contractible globular extension and let $F : C \to
\M$ be a globular functor. Then for any object $X$ of $\M$ and any $n \ge
1$, the groupoid $\PI_n(\Pi_{C, F}(X))$ is equal to $\PI_n(y^{}_{\M}(X))$.
In particular, $\PI_n(\Pi_{C, F}(X))$ depends neither on~$F$ nor on $C$.
\end{prop}

\begin{proof}
Consider the following diagram:
\[
\xymatrix@C=1pc@R=1pc{
& \M \ar[dl]_{y^{}_{\M}} \ar[dr]^{\Pi_{C,F}} \\
\wgpdC{\M} \ar[dr]_{\PI_n} \ar[rr]^{F^*} & & \wgpdC{C} \ar[dl]^{\PI_n} \pbox{.}\\
& \Gpd
}
\]
The upper triangle is commutative by definition of $\Pi_{C,F}$ and the lower
triangle is commutative by Proposition \ref{prop:PI_f}. Hence the result.
\end{proof}

\begin{rem}\label{rem:weak_funct}
Let  $F, F' : C \to \M$ be two globular functors. Since $\Pi_{C,F}(X)$
and~$\Pi_{C,F'}(X)$ have the same underlying globular set, we can consider
the
identity morphism (of globular sets) from $\Pi_{C,F}(X)$ to $\Pi_{C,F'}(X)$.
This morphism is \emph{not} a morphism of \oo-groupoids. Our feeling is
that it should be part of the data defining a \emph{weak} morphism of
\oo-groupoids (whatever that might mean). The previous proposition would then show
that this weak morphism is a weak equivalence and so that $\Pi_{C, F}$ does
not depend on $F$ in some homotopy category.
\end{rem}

\begin{paragr}
In the rest of this section, we will explain how to construct a
fundamental \oo-group\-oid functor $\M \to \wgpdC{C}$ when $\M$ is a
model category in which every object is fibrant.

We will denote by $\varnothing$ (resp.~$\ast$) the initial (resp.~terminal)
object of a category. If $X$ is an object of a model category $\M$, we will
say that $X$ is \ndef{weakly contractible} if the unique morphism $X \to
\ast$ is a weak equivalence.
\end{paragr}

\begin{tparagr}{Cofibrant and weakly contractible functors (definition)}
Let $\M$ be a model category endowed with a functor $F : \G \to \M$.
For $n \ge 0$, we define an object $\Sn{n-1}$ endowed with a map $i_n :
\Sn{n-1} \to \Dn{n}$ in the following way. For $n = 0$, we
set
 \[ \Sn{-1} = \varnothing \]
and we define
 \[ i_0 : \Sn{-1} \to \Dn{0} \] 
as the unique morphism from the initial object to $\Dn{0}$. For $n \ge 1$,
we set
\[
\Sn{n-1} = (\Dn{n-1}, i_{n-1}) \amalg_{\Sn{n-2}} (i_{n-1}, \Dn{n-1})
\]
and
\[ i_n = (\Tht{n}, \Ths{n}) : \Sn{n-1} \to \Dn{n}. \]
Let us justify that $i_n$ is well-defined. We have to check that
$\Ths{n}i_{n-1} = \Tht{n}i_{n-1}$. But by looking at the two components of
this equality, one sees that it is equivalent to the coglobular relations.

We will say that the functor $F : \G \to \M$ is \ndef{cofibrant} if for
every $n \ge 0$, the morphism $i_n : \Sn{n-1} \to \Dn{n}$ is a cofibration
in $\M$. We will say that the functor $F : \G \to \M$ is \ndef{weakly
contractible} if for any $n \ge 0$, the object $\Dn{n}$ is weakly
contractible in $\M$.
\end{tparagr}

\begin{rem}
If $\M$ is a model category, the category $\Homi(\G, \M)$ of functors from
$\G$ to $\M$ is endowed with a model category structure coming from the
fact that $\G$ is a direct category: the so-called Reedy model structure.
It is easy to show that a functor $F : \G \to \M$ is cofibrant in this model
category if and only if it is a cofibration in the sense of the previous
paragraph. Indeed, the $n$-th latching object of a functor $F : \G \to \M$
is exactly $\Sn{n-1}$. Moreover, a functor $F : \G \to \M$ is weakly
contractible in this model category structure if and only if it is weakly
contractible in the sense of the previous paragraph. In particular, a
cofibrant and weakly contractible functor $F : \G \to \M$ is nothing but a
cofibrant replacement of the terminal object of $\Homi(\G, \M)$. Such a
functor hence exists and is in some sense unique.

We will make no use of this observation in the rest of the article.  In
particular, we will prove by hand the existence of a cofibrant weakly
contractible functor (see paragraph~\ref{par:ex_cof_wc}). We refer the
reader to Chapter 15 of \cite{HirschMC} for the theory of Reedy model
structures.
\end{rem}

\begin{ex}
Let $R : \G \to \Top$ be the functor defining the globular extension
structure of $\Top$. In this case, the morphism $i_n$ is nothing but the
inclusion of the $(n-1)$\nobreakdash-sphere~$\Stop{n-1}$ as the boundary of the $n$-disk
$\Dtop{n}$ and is hence a cofibration. On the other hand, the disks are of
course contractible. The functor $R$ is thus cofibrant and weakly
contractible.
\end{ex}

\begin{prop}\label{prop:glob_contr}
Let $\M$ be a model category endowed with a cofibrant and weakly
contractible functor $F : \G \to \M$. Then every globular sum in
$\M$ is weakly contractible.
\end{prop}

\begin{proof}
Let us first prove that the $\Ths{n}$'s and the $\Tht{n}$'s are
cofibrations. Let $n \ge 1$. We have $\Ths{n} = i_{n}\ceps{2}$ and
$\Tht{n} = i_{n}\ceps{1}$, where $\ceps{1}, \ceps{2} : \Dn{n-1} \to
\Sn{n-1}$ are the two canonical morphisms. The morphism $i_n$ is a
cofibration by hypothesis. Moreover, the $\ceps{i}$'s are both pushouts of
$i_{n-1}$ and hence are cofibrations. It follows that $\Ths{n}$ and
$\Tht{n}$ are cofibrations.

Let us now prove the assertion. We need to show that for every table of
dimensions, the globular sum associated to the table is weakly contractible.
We prove this by induction on the width $n$ of the table. If $n = 1$, the
globular sum is a $\Dn{n}$ which is weakly contractible by assumption.
Otherwise, let
\[
  S = \left( \begin{matrix} i_1 && \cdots && i_n \cr
  & i'_1 & \cdots & i'_{n-1}
  \end{matrix}
  \right)
\]
be our table of dimensions and $X$ be the globular sum associated to $S$.
Set
\[
   T = \left( \begin{matrix} i_1 && \cdots && i_{n-1} \cr
  & i'_1 & \cdots & i'_{n-2}
  \end{matrix}
  \right)
\]
and denote by $Y$ the globular sum associated to $T$. We have
\[ X = Y \amalg_{\Dn{i'_{n-1}}} \Dn{i_n}. \]
More precisely, the commutative square
\[
\xymatrix{
  \Dn{i'_{n-1}} \ar[r]^{\Tht[i'_{n-1}]{i_n}} \ar[d] & \Dn{i_n} \ar[d] \\
  Y \ar[r] & X \\
}
\]
is cocartesian. But the top horizontal arrow is a cofibration between
weakly contractible objects and is thus a trivial cofibration. Hence the
canonical morphism $Y \to X$ is a trivial cofibration. Since by induction
$Y$ is weakly contractible, $X$ is also weakly contractible. Hence the
result.
\end{proof}

\begin{prop}\label{prop:M_contr}
Let $\M$ be a model category endowed with a cofibrant and weakly
contractible functor $F : \G \to \M$. Assume that every object of $\M$ is fibrant.
Then the globular extension $(\M, F)$ is contractible.
\end{prop}

\begin{proof}
The data of an admissible pair $(f, g) : \Dn{n} \to X$ is equivalent to
the data of a morphism $\Sn{n} \to X$. Indeed, every morphism $k : \Sn{n}
\to X$ can be written $k = (f, g)$, where $f, g : \Dn{n} \to X$ satisfy
$fi_n = gi_n$. But the two components of this equality are precisely the
relations showing that $f$ and $g$ are globularly parallel.

Moreover, a morphism $h : \Dn{n} \to X$ is a lifting of the admissible pair
$(f, g) : \Dn{n} \to X$ if and only if the triangle
\[
\xymatrix{
\Dn{n+1} \ar[rd]^h \\
\Sn{n} \ar[r]_{(f,g)} \ar[u]^{i_{n+1}} & X \\
}
\]
commutes. Indeed, the commutativity of the triangle is equivalent to
the equalities $hi_n\ceps{2} = f$ and $hi_n\ceps{1} = g$. But we have
$i_n\ceps{2} = \Ths{n}$ and $i_n\ceps{1} = \Tht{n}$.

Consider now the commutative square
\[
\xymatrix{
\Sn{n} \ar[d]_{i_{n+1}} \ar[r]^{(f,g)} & X \ar[d] \\
\Dn{n+1} 
\ar[r] & \ast\pbox{.}
}
\]
By hypothesis, the morphism $i_{n+1}$ is a cofibration. Moreover, by the
previous proposition, the globular sum $X$ is weakly contractible. Since by
assumption, every object of $\M$ is fibrant, the morphism $X \to \ast$ is a
trivial fibration. Hence there exists a lifting $\Dn{n+1} \to X$. This shows
that the globular extension $(\M, F)$ is contractible.
\end{proof}

\begin{tparagr}{Cofibrant and weakly contractible functors (construction)}
\label{par:ex_cof_wc}
Let $\M$ be a model category. We will now construct a cofibrant and weakly
contractible functor from~$\G$ to~$\M$.

We define by induction on $n \ge 0$ an object $\Dn{n}$ of $\M$ and a cofibration
$i_n : \Sn{n-1} \to \Dn{n}$, where
\[
\Sn{-1} = \varnothing
\quad\text{and}\quad
\Sn{n-1} = (\Dn{n-1}, i_{n-1}) \amalg_{\Sn{n-2}} (i_{n-1}, \Dn{n-1}),
\quad n \ge 1,
\]
in the following way.
For $n = 0$, we define $\Dn{0}$ as a cofibrant replacement of the terminal
object and $i_0 : \Sn{-1} \to \Dn{0}$ as the unique morphism from the initial
object to $\Dn{0}$. For $n \ge 1$, consider the morphism
\[ (\id{\Dn{n-1}}, \id{\Dn{n-1}}) : \Sn{n-1} \to \Dn{n-1} \]
and factor it as a cofibration $i_n$, followed by a weak equivalence
$p_n$. We define $\Dn{n}$ as the middle object of this factorization
\[ \Sn{n-1} \xrightarrow{i_n} \Dn{n} \xrightarrow{p_n} \Dn{n-1}. \]
Note that since $p_n$ is a weak equivalence and $\Dn{0}$ is weakly
contractible, by induction on $n \ge 0$, the object $\Dn{n}$ is weakly
contractible.

We now set
\[ \Ths{n} = i_n\ceps{2} \quad\text{and}\quad \Tht{n} = i_n\ceps{1},
\quad n \ge 1, \]
where $\ceps{1}, \ceps{2} : \Dn{n-1} \to \Sn{n-1}$ are the canonical
morphisms.  Let us prove the coglobular relations. Let $n \ge 1$. By
definition of $\Sn{n}$, we have
$\cepsp{1}i_n = \cepsp{2}i_n$,
where $\cepsp{1}, \cepsp{2} : \Dn{n} \to \Sn{n}$ are the canonical morphisms.
We thus have
\[
\Ths{n+1}\Ths{n} = i_{n+1}\cepsp{2}i_n\ceps{2} = 
i_{n+1}\cepsp{1}i_n\ceps{2} = \Tht{n+1}\Ths{n}
\]
and
\[
\Ths{n+1}\Tht{n} = i_{n+1}\cepsp{2}i_n\ceps{1} = 
i_{n+1}\cepsp{1}i_n\ceps{1} = \Tht{n+1}\Tht{n}.
\]

We have thus defined a functor $\G \to \M$. This functor is cofibrant and
weakly contractible by construction.
\end{tparagr}

\begin{tparagr}{The fundamental $\infty$-groupoid functor of a
model category}\label{par:def_Pi_cat_mod}
Let now $\M$ be a model category in which every object is fibrant. By the
previous paragraph, there exists a cofibrant and weakly contractible
functor $F : \G \to \M$. By Proposition \ref{prop:M_contr}, the globular
extension $(\M, F)$ is contractible. We can thus apply paragraph
\ref{par:constr_Pi} and we get an
adjunction
\[
R_{C, K} :  \wgpdC{C} \to  \M, \qquad 
\Pi_{C,K} : \M \to \wgpdC{C},
\]
which depends on a globular functor $K : C \to \M$.

If $X$ is an object of $\M$, we will call $\Pi_{C, K}(X)$ the
\ndef{fundamental \oo-groupoid} of $X$. Note that it depends a priori on $F$
and $K$. In particular, its underlying globular set depends on~$F$. However,
this \oo-groupoid should not depend on $F$ and $K$ in some weak sense (see
Remark \ref{rem:weak_funct}). We will show in the next section that the
homotopy groups of this \oo-groupoid are independent of $F$ and $K$.
\end{tparagr}

\begin{rem}
If $\M$ is a combinatorial model category, we can also construct a
fundamental \oo-groupoid functor. Indeed, by a theorem of Dugger
(\cite[Theorem 1.1]{DuggerPres}), such an $\M$ is Quillen equivalent to some
left Bousfield localization $\N$ of a category of simplicial presheaves
endowed with the projective model structure. The model category $\N$ is
combinatorial and every cofibration of $\N$ is a monomorphism. By a theorem
of Nikolaus (\cite[Corollary 2.21]{NikolausAlgMod}), such a model category
is Quillen equivalent to a model category $\P$ in which every object is
fibrant.  Although these two Quillen equivalences do not compose, we obtain
a functor $Q : \M \to \P$ ``inducing'' an equivalence on homotopy
categories. We can thus define a fundamental \oo-groupoid functor as the
composition
\[ \Pi_\infty : \M \to \P \xrightarrow{\Pi_\infty} \wgpdC{C}, \]
where $\Pi_\infty : \P \to \wgpdC{C}$ is a fundamental \oo-groupoid functor
in the sense of the previous paragraph.
\end{rem}

\begin{rem}
We do not need the full strength of a model category structure to
construct the fundamental \oo-groupoid functor. Recall that if $\C$ 
is category, a \ndef{weak factorization system} consists of
two classes of morphisms $L$ and $R$ of $\C$ satisfying the following
properties: 
\begin{enumerate}
\item every morphism $f$ of $\C$ factors as $f = pi$ with $i$ in $L$ and $p$
in $R$;
\item $L$ is the class of morphisms of $\C$ having the left lifting property
with respect to~$R$;
\item $R$ is the class of morphisms of $\C$ having the right lifting property
with respect to~$L$.
\end{enumerate}

Let $\M$ be a cocomplete category admitting a terminal object, endowed
with a weak factorization system $(L, R)$. We will think of the elements of
$L$ as cofibrations and of the elements of $R$ as trivial fibrations (and
in particular as weak equivalences). We can construct as in paragraph
\ref{par:ex_cof_wc} a functor $F : \G \to \M$ by replacing every
factorization as a cofibration followed by a weak equivalence, by a
factorization as a morphism of $L$ followed by a morphism of $R$.
It is not true in general that $(\M, F)$ is a contractible globular
extension: we need to add a hypothesis saying that in some sense
every object of $\M$ is fibrant.

We will say that an object $X$ of $\M$ is \ndef{fibrant and weakly
contractible} if the unique morphism $X \to \ast$ is in $R$. Suppose now
that our weak factorization system satisfies the following additional property:
for every cocartesian square
\[
\xymatrix{
X_1 \ar[r] \ar[d] & X_2 \ar[d] \\
X_3 \ar[r] & X_4 \\
}
\]
in $\M$, if $X_1$, $X_2$, $X_3$ are fibrant and weakly contractible, then
so is $X_4$. Note that this hypothesis is satisfied when we consider the
factorization system of cofibrations and trivial fibrations of a model
category in which every object is fibrant.

Under this hypothesis, the proofs of Propositions \ref{prop:glob_contr} and
\ref{prop:M_contr} can easily be adapted and we obtain that $(\M, F)$ is a
contractible globular extension. We can thus define a fundamental
\oo-groupoid functor $\Pi_\infty : \M \to \wgpdC{C}$.
\end{rem}

\section{Quillen's theory of $\pi_1$ in a model category}

The purpose of this section is to recall to the reader some definitions and
facts about Quillen's theory of $\pi_1$ in a model category which was
introduced in Section I.2 of \cite{Quillen}.

\begin{tparagr}{Notation and terminology about model categories}
If $\M$ is a model category and $X, Y$ are two objects of $\M$, we will
denote by $[X, Y]_\M$ the set of morphisms between $X$ and $Y$ in the
homotopy category $\Ho(\M)$ of $\M$. If $\M$ is understood, we will simply
denote this set by $[X, Y]$. Recall that when $X$ is cofibrant and $Y$ is
fibrant, $[X, Y]$ is in canonical bijection with the set of morphisms $X \to
Y$ up to (left or right) homotopy.

Let $\M$ be a model category. The definition of cylinder objects varies from
author to author. We will use Quillen's original definition: a
\ndef{cylinder object} of an object $A$ of $\M$ is an object $C$ of $\M$
endowed with a factorization
\[
\xymatrix{
A \amalg A \ar[r]^-{(i_1, i_0)} & C \ar[r]^s & A
}
\]
of the codiagonal of $A$ as a cofibration followed by a weak
equivalence. Dually, a \ndef{path object} of an object $B$ of $\M$ is an
object $P$ of $\M$ endowed with a factorization
\[
\xymatrix{
B \ar[r]^r & P \ar[r]^-{(p_1, p_0)} & B \times B
}
\]
of the diagonal of $B$ as a weak equivalence followed by a fibration.  We will
always use the letters $i$, $s$, $r$ and $p$ to denote the structural maps
of cylinder and path objects.

If $f, g : A \to B$ are two morphisms of $\M$, a \ndef{left homotopy} from
$f$ to $g$ is a morphism $h : C \to B$, where $C$ is a cylinder object of $A$,
such that $hi_0 = f$ and $hi_1 = g$. Note that we have inverted Quillen's
original direction for homotopies (because we have exchanged~$i_0$ and
$i_1$) in order to be coherent with our convention for globular sums. We
apply the same treatment to right homotopies.
\end{tparagr}

We will use the following easy fact several times.

\begin{tparagr}{Existence and uniqueness of morphisms of path objects}
\label{paragr:morph_path_obj}
Let $f : A \to A'$ be a morphism in a model category $\M$. Let $P$
(resp.~$P'$) be a path object of $A$ (resp.~of~$A'$). Suppose moreover that $r
: A \to P$ is a cofibration. Then there exists a morphism $g$ such that the
square
\[
\xymatrix@C=3pc{
A \ar[d]_f \ar[r]^r & P \ar[d]_g \ar[r]^-{(p_1, p_0)} & A\times A \ar[d]^{f
\times f} \\
A'\ar[r]_{r'} & P' \ar[r]_-{(p'_1, p'_0)} & A'\times A'  \\
}
\]
is commutative. Indeed, any lifting of the commutative square
\[
\xymatrix@C=3pc{
A \ar[r]^f \ar[d]_r & A' \ar[r]^{r'} &P' \ar[d]^-{(p'_1, p'_0)} \\
P \ar[r]_-{(p_1, p_0)} & A \times A \ar[r]_{f\times f} & A'\times A'\\
}
\]
gives such a morphism $g$. Moreover, such a morphism is a weak equivalence
and is unique up to homotopy in the category of objects of $\M$ under $A$
and over $A' \times A'$ (see Proposition~7.6.14 of \cite{HirschMC}).
\end{tparagr}

From now on, we fix a model category $\M$, a cofibrant object $A$ of $\M$, a
fibrant object $B$ of $\M$ and two morphisms $f, g : A \to B$.

\begin{tparagr}{2-homotopies and correspondences}
Let $C$ and $C'$ be two cylinder objects of~$A$.
A \ndef{$2$-cylinder object} of $C$ and $C'$ is an object $D$ of $\M$ endowed
with a factorization
\[
\xymatrix{
 (C', (i'_1, i'_0)) \amalg_{A \amalg A} (C, (i_1, i_0)) \ar[r]^-{(j_1,j_0)} & D \ar[r]^t & A
}
\]
of the morphism 
\[ (s', s) : (C', (i'_1, i'_0)) \amalg_{A \amalg A} (C, (i_1, i_0)) \to A \]
as a cofibration followed by a weak equivalence.
If $C = C'$, we will simply say that $D$ is a $2$-cylinder object of $C$.

If $h : C \to B$ and $h' : C' \to B$ are two left homotopies from $f$ to
$g$, then a \ndef{left $2$-homotopy} from $h$ to $h'$ is a morphism
$H : D \to B$ such that $Hj_0 = h$ and $Hj_1 = h'$. If such an $H$ exists,
we will say that $h$ and $h'$ are \ndef{left $2$-homotopic}.

Let $C$ be a cylinder object of $A$ and let $P$ be a path object of $B$. If
$h : C \to B$ is a left homotopy from $f$ to $g$ and $k : A \to P$ is a
right homotopy from $f$ to $g$, a \ndef{correspondence} between $h$ and $k$ 
is a map $H : C \to P$ such that
\[ p_0H = h,\quad Hi_0 = k,\quad p_1H = gs\quad\text{and}\quad Hi_1 = rg. \]
If such a correspondence exists, we will say that $h$ and $k$ \ndef{correspond}.
\end{tparagr}

\goodbreak

\begin{lemma}[Quillen]\label{lemma:hom_quillen}
\ 
\begin{enumerate}
\item 
Let $h : C \to B$ be a left homotopy from $f$ to $g$. Then for every path
object $P$ of $B$, there exists a right homotopy $k : A \to P$ corresponding
to $h$.
\item Let $h : C \to B$ and $h' : C' \to B$ be two left homotopies
from $f$ to $g$. If $h$ corresponds to a right homotopy $k$, then $h$ is
left $2$-homotopic to $h'$ if and only if $h'$ corresponds to $k$. More
precisely, if $D$ is a fixed $2$-cylinder object of $C$ and $C'$, then there
exists a left $2$-homotopy $D \to B$ from $h$ to $h'$ if and only if
$h'$ and $k$ correspond.
\end{enumerate}
\end{lemma}

\begin{proof}
\ 
\begin{enumerate}
\item See Lemma 1 of \cite[Section I.2]{Quillen}.
\item See Lemma 2 of \cite[Section I.2]{Quillen}. The fact that one can fix the
$2$-cylinder $D$ is not stated but appears clearly in the proof.
\end{enumerate}
\end{proof}

\begin{tparagr}{The set $\pi_1(A, B; f, g)$}\label{paragr:desc_pi1}
We will denote by $\pi_1(A, B; f, g)$ the class of left homotopies of
$\M$ from $f$ to $g$, up to left $2$-homotopy. It is not clear a priori that
$\pi_1(A, B; f, g)$ is a set. But by the previous lemma (and its dual),
$\pi_1(A, B; f, g)$ is in canonical bijection with the class of left
homotopies from a fixed cylinder $C$ up to left $2$-homotopy.  This is
obviously a set. Note that we can even ask for the left $2$-homotopies to
use a fixed $2$-cylinder of $C$.

Dually, we can consider right homotopies up to right $2$-homotopies in an
appropriate sense. The previous lemma shows that we get this way a set in
canonical bijection with $\pi_1(A, B; f, g)$.
\end{tparagr}

\begin{tparagr}{The groupoid $\Pi_1(A, B)$}\label{paragr:pi1_comp}
Let $\Pi_1(A, B)$ be the graph defined in the following way. The objects
of~$\Pi_1(A, B)$ are the morphisms from $A$ to $B$ of $\M$. If $f, g : A \to
B$ are two such objects, the set of arrows from $f$ to $g$ in $\Pi_1(A, B)$
is $\pi_1(A, B; f, g)$.

Let $f_1, f_2, f_3 : A \to B$ be three morphisms of $\M$, and let $h : C \to
B$ and $h' : C' \to B$ be two left homotopies respectively from $f_1$ to
$f_2$ and from $f_2$ to $f_3$. Then
\[ (C', i'_0) \amalg_A (i_1, C) \]
is a left cylinder for the factorization
\[
\xymatrix{
A \amalg A \ar[r]^-{i_1\, \amalg\, i_0} & (C', i'_0) \amalg_A (i_1, C)
\ar[r]^-{(s', s)} & A
},
\]
and
\[
(C', i'_0) \amalg_A (i_1, C) \xrightarrow{(h', h)} B
\]
is a left homotopy from $f_1$ to $f_3$.
\end{tparagr}

\begin{prop}[Quillen]
The above construction defines a map
\[
\pi_1(A, B; f_2, f_3) \times \pi_1(A, B; f_1, f_2) \to \pi_1(A, B; f_1, f_3)
\]
for every $f_1, f_2, f_3 : A \to B$ and these maps induce a groupoid structure
on $\Pi_1(A, B)$.
\end{prop}

\begin{proof}
See Proposition 1 of \cite[Section I.2]{Quillen}.
\end{proof}

\begin{tparagr}{Under and over model categories}
If $X$ is an object of $\M$, we will denote by~$X\backslash \M$ the category of
objects of $\M$ under $X$, that is, the category whose objects are pairs $(Y,
u)$ where $Y$ is an object of $\M$ and $u : X \to Y$ is a morphism of $\M$,
and whose morphisms from $(Y, u)$ to $(Y', u')$ are morphisms $v : Y \to Y'$
of $\M$ such that $uv = u'$. This category inherits a structure of
model category from the one on $\M$: a morphism of $X\backslash \M$ is a
cofibration (resp.~a fibration, resp.~a weak equivalence) if
the underlying morphism of $\M$ is a cofibration (resp.~a fibration,
resp.~a weak equivalence). Note that an object $(Y, u)$ is cofibrant
(resp.~fibrant) in $X\backslash \M$ if and only if $u$ is a cofibration of
$\M$ (resp.~if and only if $Y$ is a fibrant object of $\M$).

Dually, if $X$ is an object of $\M$, we will denote by $\M/X$ the category
of objects of $\M$ over $X$. This category is nothing but $(X\backslash
M^\op)^\op$, and all the above statements can be dualized.

In what follows, we will often work in the category $\AAM$. The morphisms $f, g
: A \to B$ induce a morphism $(g, f) : \AA \to B$ and we get an object $(B,
(g, f))$ of $\AAM$. This object is fibrant since $B$ is fibrant in $\M$.
Similarly, from a cylinder object $C$ of $A$, we get an object $(C, (i_1, i_0))$
of $\AAM$. This object is cofibrant since $(i_1, i_0)$ is a cofibration of
$\M$. Each time we will consider $B$ and $C$ as objects of $\AAM$, they will
be endowed with the morphisms we have just defined.
\end{tparagr}

\begin{prop}\label{prop:desc_graph_Pi1}
Let $C$ be a cylinder object of $A$. A morphism $h : C \to B$ of $\M$
defines a morphism $(C, (i_1,i_0)) \to (B,(g,f))$ in $\AAM$ if and only if
$h$ is a homotopy from~$f$ to $g$ in $\M$. This correspondence induces a
bijection
\[
\pi_1(A, B; f, g) \cong [(C, (i_1,i_0)), (B,(g,f))]_{\AAM}.
\]
\end{prop}

\begin{proof}
\def\Ci{(C, (i_1, i_0))}
\def\Bgf{(B, (g, f))}
The first assertion is obvious. Let us prove the second one. Since the
objects $\Ci$ and $\Bgf$ are respectively cofibrant and fibrant,
\[ [\Ci, \Bgf]_{\AAM} \] 
is the set of morphisms $\Ci \to \Bgf$ up to left homotopy in $\AAM$.
Moreover, by paragraph \ref{paragr:desc_pi1}, $\pi_1(A, B; f, g)$ is the set
of left homotopies $C \to B$ between $f$ and $g$ up to left $2$-homotopy in
$\M$.

Let $h, h' : \Ci \to \Bgf$ be two morphisms. We have to show that $h$
and~$h'$ are left homotopic in $\AAM$ if and only if they are $2$-homotopic
as homotopies between $f$ and $g$ in $\M$.

Consider the morphism
$(\id{C}, \id{C}) : C \amalg_{A \amalg A} C \to C$ of $\M$. Let
\[
\xymatrix{
C \amalg_{A \amalg A} C \ar[r]^-k & D \ar[r]^-q & C
}
\]
be a factorization of this morphism as a cofibration followed by a weak
equivalence. We get an object $(D, k)$ under $\AA$ and the above
factorization makes $(D, k)$ a cylinder object of $\Ci$ in $\AAM$. The
morphisms $h$ and $h'$ are left homotopic in $\AAM$ if and only if there
exists a left homotopy between them in $\AAM$ using the cylinder $(D, k)$.

Using the weak equivalence $s : C \to A$ of $\M$, we obtain a factorization
\[
\xymatrix{
C \amalg_{A \amalg A} C \ar[r]^-k & D \ar[r]^-{sq} & A
}
\]
in $\M$, making $D$ a $2$-cylinder of $C$ in $\M$. The left homotopies $h$
and $h'$ are left $2$\nobreakdash-homotopic if and only if there exists a
left $2$-homotopy between them in $\M$ using the $2$-cylinder $D$ (see
paragraph \ref{paragr:desc_pi1}).

But it is obvious that a morphism $H : D \to B$ of $\M$ induces a left
homotopy between~$h$ and $h'$ in $\AAM$ if and only if $H$ is a left
$2$-homotopy between $h$ and $h'$ in $\M$. Hence the result.
\end{proof}

\begin{rem}
Let $C$ be a cylinder object of $A$ and let $P$ be a path object of $B$. The
previous proposition and its dual proposition imply that there is a
canonical bijection
\[
[(C, (i_1, i_0)), (B, (g, f))]_{\AAM} \cong [(A, (g, f)), (P, (p_1,
p_0))]_{\MBB}.
\]
Explicitly, a left homotopy $h : C \to B$ from $f$ to $g$ is sent to a
right homotopy $k : A \to P$ corresponding to $h$.
\end{rem}

\begin{lemma}\label{lemma:comp_lr}
Let $B'$ be a second fibrant object of $\M$ and let $v : B \to B'$ be a
morphism of $\M$. Let $C$ be a cylinder object of $A$ and let
$P$ (resp.~$P'$) be a path object of $B$ (resp.~$B'$). Assume
that $r : B \to P$ is a cofibration. Let $w : P \to P'$ be any morphism
making the diagram
\[
\xymatrix@C=3pc{
B \ar[d]_v \ar[r]^r & P \ar[d]_w \ar[r]^-{(p_1, p_0)} & B\times B \ar[d]^{v
\times v} \\
B'\ar[r]_{r'} & P' \ar[r]_-{(p'_1, p'_0)} & B'\times B'  \\
}
\]
commute (such a morphism exists by paragraph \ref{paragr:morph_path_obj}).
Then the induced map 
\[
w\circ - : [(A, (g, f)), (P, (p_1, p_0))]_{\MBB} \to [(A, (vg, vf)), (P',
(p'_1, p'_0))]_{\MBBp}
\]
does not depend on the choice of $w$, and the
square
\[
\xymatrix{
[(C, (i_1, i_0)), (B, (g, f))]_{\AAM} \ar[d]_{v \circ -} \ar[r]^-{\sim} & [(A, (g, f)),
(P, (p_1, p_0))]_{\MBB} \ar[d]^{w \circ -} \\
[(C, (i_1, i_0)), (B', (vg, vf))]_{\AAM} \ar[r]^-{\sim} & [(A, (vg, vf)),
(P', (p'_1, p'_0))]_{\MBBp},
}
\]
where the horizontal arrows are the bijections of the previous remark,
is commutative.
\end{lemma}

\begin{proof}
It suffices to show that the square of the statement is
commutative. Let $h : C \to B$ be a left homotopy from $f$ to $g$ and let
$k : A \to P$ be a right homotopy corresponding to $h$ via a correspondence
$H : C \to P$. We have to show that the left homotopy $vh : C \to B'$ from
$vf$ to $vg$ corresponds to the right homotopy $wk : A \to P'$. It is
immediate that $wH : C \to P'$ is the desired correspondence.
\end{proof}

\begin{tparagr}{Functoriality of $\Pi_1$}
\label{paragr:def_Pi_1}
Let $B'$ be a second fibrant object of $\M$ and let $v : B \to B'$ be a
morphism of $\M$. The morphism $v$ induces a morphism of graphs
 \[ \Pi_1(A, v) : \Pi_1(A, B) \to \Pi_1(A, B') \]
by sending a morphism $f : A \to B$ of $\M$ to $vf : A \to B'$, and a left
homotopy $h : C \to B$ representing an element of $\pi_1(A, B; f, g)$ to the
left homotopy $vh : C \to B'$. It is easy to see that this morphism of
graphs is a functor. Note that in the bijection of
Proposition~\ref{prop:desc_graph_Pi1}, the association $h \mapsto vh$
corresponds to the post-composition by $v$ seen as a morphism of
$\Ho(\AAM)$.

Dually, if $A'$ is a second cofibrant object and $u : A' \to A$ is a
morphism of $\M$, then $u$ induces a functor
\[ \Pi_1(u, B) : \Pi_1(A, B) \to \Pi_1(A', B). \]
Note that we need to consider \emph{right} homotopies to define this functor.

By Lemma 3 of \cite[Section I.2]{Quillen}, the square
\[
\xymatrix{
\Pi_1(A, B) \ar[r] \ar[d] & \Pi_1(A', B) \ar[d] \\
\Pi_1(A, B') \ar[r] & \Pi_1(A', B')
}
\]
is commutative. This also follows easily from Lemma \ref{lemma:comp_lr}.
We can thus define a functor
\[ \Pi_1(u, v) : \Pi_1(A, B) \to \Pi_1(A', B'). \]
\end{tparagr}

\begin{prop}\label{prop:Pi_1_we}
Let $u : A' \to A$ be a weak equivalence between cofibrant objects and let $v
: B \to B'$ be a weak equivalence between fibrant objects. Then the functor
\[ \Pi_1(u, v) : \Pi_1(A, B) \to \Pi_1(A', B') \]
is an equivalence of categories.
\end{prop}

\begin{proof}
  By definition of $\Pi_1(u, v)$ and by duality, we can assume that $u$ is an
  identity. But in the bijection of Proposition \ref{prop:desc_graph_Pi1},
  $\Pi_1(A, v)$ is the post-composition by $v$ in the homotopy category of
  $\AAM$. Since $v$ is a weak equivalence, this post-composition is a
  bijection, hence the result.
\end{proof}

\section{Quillen homotopy groups in a model category}

In this section, we fix a model category $\M$ in which every object is
fibrant.

\begin{tparagr}{Connected components}
Let $X$ be an object of $\M$. The set $\pi_0(X)$ of
\ndef{connected components} of $X$ is defined by
\[
  \pi_0(X) = [\ast, X].
\]
This set can be described in term of homotopy classes
thanks to the formula
\[
  \pi_0(X) \cong [\Dn{0}, X],
\]
where $\Dn{0}$ is any cofibrant contractible object of $\M$. It is obvious
that $\pi_0$ defines a functor from $\M$ to the category of sets and that
this functor sends weak equivalences to bijections.
\end{tparagr}

\begin{tparagr}{Based objects}
Let $X$ be an object of $\M$. A \ndef{base point} $x$ of $X$ is a morphism
$\Dn{0} \to X$ of $\M$, where~$\Dn{0}$ is cofibrant and weakly contractible.
We will say that an object of $\M$ endowed with a base point is a
\ndef{based object} of $\M$. 

Let $(X, x : \Dn{0} \to X)$ and $(X', x' : \Dnp{0} \to X)$ be two
based objects of $\M$. A \ndef{morphism of based objects} from $(X, x)$ to $(X',
x')$ is given by morphisms $f : X \to X'$ and $f_0 : \Dn{0} \to
\Dnp{0}$ of $\M$ making the square
\[
\xymatrix{
\Dn{0} \ar[d]_{f_0} \ar[r]^x & X \ar[d]^f \\
\Dnp{0} \ar[r]_{x'} & X'
}
\]
commute. Note that such an $f_0$ is necessarily a homotopy equivalence since
$\Dn{0}$ and $\Dnp{0}$ are both cofibrant, fibrant and weakly contractible.
We will denote such a morphism of based objects by $(f, f_0)$.

We adopt the following convention on notation: if $(X, x)$ (resp.~ $(X',
x')$) is a based object of $\M$, then the source of $x$ (resp.~ of $x'$) will
be denoted by $\Dn{0}$ (resp.~by $\Dnp{0}$), unless otherwise specified.
\end{tparagr}

\begin{tparagr}{The fundamental group}
\label{paragr:pi_1}
Let $(X, x : \Dn{0} \to X)$ be a based object of $\M$. The \ndef{fundamental
group} of $(X, x)$ is the group $\pi_1(\Dn{0}, X; x, x)$. We will denote it
by $\pi_1(X, x)_\M$, or briefly, by $\pi_1(X, x)$.

Let $(f, f_0) : (X, x) \to (X', x')$ be a morphism of based objects. The
morphism $(f, f_0)$ can be decomposed as
\[ (X, x) \xrightarrow{(f, \id{\Dn{0}})} (X', x'f_0) \xrightarrow{(\id{X'},
f_0)} (X', x'). \]
By applying the functor $\Pi_1$, we get a diagram
\[ \pi_1(X, x) \to \pi_1(X', x'f_0) \leftarrow \pi_1(X', x'). \]
Since $f_0$ is a weak equivalence, the right arrow is an isomorphism and, 
using its inverse, we get a morphism
\[ \pi_1(f, f_0) : \pi_1(X, x) \to \pi_1(X', x'). \]
One can easily check that this definition makes $\pi_1$ a functor from the
category of based objects of $\M$ to the category of groups.

If follows from Proposition \ref{prop:Pi_1_we} that if $f$ is a weak
equivalence, then $\pi_1(f, f_0)$ is an isomorphism. Moreover, by paragraph
\ref{paragr:def_Pi_1}, if $\Dn{0} = \Dnp{0}$ and $f, g : (X, x) \to (X',
x')$ are two morphisms of $\Dn{0}\backslash\M$ inducing the same morphism in
$\Ho(\Dn{0}\backslash\M)$, then we have 
\[ \pi_1(f, \id{\Dn{0}}) = \pi_1(g, \id{\Dn{0}}). \]
\end{tparagr}

\begin{tparagr}{Loop objects}
Let $(X, x)$ be a based object of $\M$ and let $P$ be a path object of~$X$
such that $r : X \to P$ is a cofibration. The \ndef{loop object} of $X$
based at $x$ (using $P$) is defined as the pullback of the diagram
\[ \Dn{0} \xrightarrow{(x,x)} X\times X \xleftarrow{(p_1, p_0)} P.  \]
We will denote it by $\Omega^P_x X$. Since every object of $\M$ is fibrant and
$(p_1, p_0)$ is a fibration, the cartesian square defining $\Omega^P_x X$ is actually
a homotopy cartesian square and $\Omega^P_x X$ is a model for the homotopy
pullback
\[ \Dn{0} \xrightarrow{x} X \xleftarrow{x} \Dn{0}. \]
In particular, the image of $\Omega^P_x X$ in $\Ho(\M)$ does not depend on
$P$ up to a canonical isomorphism. 

One can explicitly construct this canonical isomorphism. Let $P'$ be a
second path object of $X$ and let $g : P \to P'$ be a morphism making the 
diagram
\[
\xymatrix@C=3pc{
X \ar[d]_{\id{X}} \ar[r]^r & P \ar[d]_g \ar[r]^-{(p_1, p_0)} & A\times A
\ar[d]^{\id{X \times X}}\\
X\ar[r]_{r'} & P' \ar[r]_-{(p'_1, p'_0)} & A\times A  \\
}
\]
commute (such a morphism exists by paragraph \ref{paragr:morph_path_obj}). Then
the commutative diagram
\[
\xymatrix@C=3.5pc{
\Dn{0} \ar[d]_{\id{\Dn{0}}} \ar[r]^-{(x, x)} & X \times X \ar[d]_{\id{X\times
X}} & \ar[l]_-{(p_1, p_0)} P \ar[d]^g\\
\Dn{0} \ar[r]_-{(x, x)} & X \times X & \ar[l]^-{(p'_1, p'_0)} P'
}
\]
induces a morphism $f$ from $\Omega^P_x X$ to $\Omega^{P'}_x X$.
Since the pullbacks defining these objects are homotopy pullbacks and $g$ is
a weak equivalence, $f$ is also a weak equivalence. The morphism $f$ hence
induces our canonical isomorphism in $\Ho(\M)$.

The commutative diagram
\[
\xymatrix{
\Dn{0} \ar[d]_{\id{\Dn{0}}} \ar[r]^x & X \ar[r]^r & P \ar[d]^{(p_1, p_0)} \\
\Dn{0} \ar[rr]_{(x, x)} & & X \times X
}
\]
induces a morphism $c_x : \Dn{0} \to \Omega^P_x X$. The loop object
$\Omega^P_x X$ is thus endowed with the structure of a based object $(\Omega^P_x
X, c_x)$. As above, the image of this object in $\Ho(\Dn{0}\backslash\M)$
does not depend on $P$ up to a canonical isomorphism. For this reason, we
will often denote this object by $(\Omega_x X, c_x)$, without reference to
$P$.

Let $(f, f_0) : (X, x) \to (X', x')$ be a morphism of based objects of $\M$. 
Consider the decomposition of $(f, f_0)$ as
\[ (X, x) \xrightarrow{(f, \id{\Dn{0}})} (X', x'f_0) \xrightarrow{(\id{X'},
f_0)} (X', x'). \]
The first morphism can be seen as a morphism of $\Dn{0}\backslash\M$. Since
the objects $\Omega_x X$ and~$\Omega_{x'f_0} X'$ are defined as homotopy
pullbacks, by the theory of derived functors, the morphism~$f$ induces a
morphism from $(\Omega_x X, c_x)$ to $(\Omega_{x'f_0} X, c_{x'f_0})$
in~$\Ho(\Dn{0}\backslash \M)$.

This morphism can be described explicitly in the following way.
Let $P$ (resp.~$P'$) be a path object of $X$ (resp.~of $X'$) and
let $g : P \to P'$ be a morphism making the diagram
\[
\xymatrix@C=3pc{
X \ar[d]_f \ar[r]^r & P \ar[d]_g \ar[r]^-{(p_0, p_1)} & X\times X \ar[d]^{f
\times f} \\
Y\ar[r]_{r'} & P' \ar[r]_-{(p'_1, p'_0)} & Y\times Y  \\
}
\]
commute (such a morphism exists by paragraph \ref{paragr:morph_path_obj}).
Then the commutative diagram
\[
\xymatrix@C=3.5pc{
\Dn{0} \ar[d]_{\id{\Dn{0}}} \ar[r]^-{(x, x)} & X \times X \ar[d]_{f\times f}
& \ar[l]_-{(p_1, p_0)} P \ar[d]^g\\
\Dn{0} \ar[r]_-{(x', x')f_0} & X' \times X' & \ar[l]^-{(p'_1, p'_0)} P'
}
\]
induces a morphism $\Omega_{\Dn{0}} f : \Omega_x X \to \Omega_{x'f_0}
X'$ and $(\Omega_{\Dn{0}} f, \id{\Dn{0}})$ is a morphism of based objects
from $(\Omega_x X, c_x)$ to $(\Omega_{x'f_0} X', c_{x'f_0})$. This morphism induces
our canonical morphism in $\Ho(\Dn{0}\backslash\M)$.

The second morphism $(\id{X'}, f_0)$ gives rise to a commutative diagram
\[
\xymatrix@C=3.5pc{
\Dn{0} \ar[d]_{f_0} \ar[r]^-{(x', x')f_0} & X' \times X' \ar[d]_{\id{X'\times
X'}}
& \ar[l]_-{(p'_1, p'_0)} P' \ar[d]^{\id{P'}}\\
\Dnp{0} \ar[r]_-{(x', x')} & X' \times X' & \ar[l]^-{(p'_1, p'_0)} P'
\pbox{.}
}
\]
This diagram induces a canonical morphism $\Omega_{f_0} X' : \Omega_{x'f_0} X' \to
\Omega_{x'} X'$, and $(\Omega_{f_0} X', f_0)$ is a morphism of based objects
from $(\Omega_{x'f_0} X', c_{x'f_0})$ to $(\Omega_{x'} X', c_{x'})$.

We set
\[ \Omega_{f_0} f = \Omega_x X \xrightarrow{\Omega_{\Dn{0}} f}
\Omega_{x'f_0} X' \xrightarrow{\Omega_{f_0} X'} \Omega_{x'} X'. \]
This morphism induces a canonical morphism in $\Ho(\M)$. Moreover,
$(\Omega_{f_0} f, f_0)$ is a morphism of based objects from $(\Omega_x X,
c_x)$ to $(\Omega_{x'} X', c_{x'})$
\end{tparagr}

\begin{tparagr}{Higher homotopy groups}
Let $(X, x)$ be a based object of $\M$. Since $(\Omega_x X, c_x)$ is
well-defined, up to a canonical isomorphism coming from
$\Dn{0}\backslash\M$, as an object $\Ho(\Dn{0}\backslash\M)$,
the group~$\pi_1(\Omega_x X, c_x)$ is well-defined by the last point of
paragraph \ref{paragr:pi_1}.

Let $(f, f_0) : (X, x) \to (X', x')$ be a morphism of based objects of $\M$. We
claim that the morphism $\pi_1(\Omega_x f, f_0)$ is well-defined. Recall
that by definition, $\Omega_{f_0} f$ is the composition of
\[ \Omega_x X \xrightarrow{\Omega_{\Dn{0}} f}
\Omega_{x'f_0} X' \xrightarrow{\Omega_{f_0} X'} \Omega_{x'} X', \]
where the first morphism is well-defined as a morphism of
$\Ho(\Dn{0}\backslash\M)$ and the second morphism is well-defined as a
morphism of $\M$. It follows that
$\pi_1(\Omega_{f_0} f, f_0)$ is the composition of
\[ \pi_1(\Omega_x X, c_x)
\xrightarrow{\pi_1(\Omega_{\Dn{0}} f, \id{\Dn{0}})} 
\pi_1(\Omega_{x'f_0} X', x'f_0) \xrightarrow{\pi_1(\Omega_{f_0} X', f_0)}
\pi_1(\Omega_{x'} X', c_{x'}),
\]
where both morphisms are well-defined (the left one is well-defined by the
last point of paragraph~\ref{paragr:pi_1}). One easily checks that
$\pi_1(\Omega_{f_0} f, f_0)$ is functorial in $(f, f_0)$.

For $n \ge 2$, we define, by induction on $n$, the \ndef{$n$-th homotopy
group} of $(X, x)$ as
\[ \pi_n(X, x) = \pi_{n-1}(\Omega_x X, c_x). \]
It follows from the above discussion that $\pi_n$ is a well-defined functor
from the category of based objects of $\M$ to the category of (abelian)
groups.

Let $(f, f_0) : (X, x) \to (X', x')$ be a morphism of based objects of $\M$.
It is immediate, by induction on $n \ge 1$, that if $f$ is a weak
equivalence, then
\[ \pi_n(f, f_0) : \pi_n(X, x) \to \pi_n(X', x') \]
is an isomorphism.
\end{tparagr}

\begin{tparagr}{Functoriality of $[\Dn{0}, \Omega_x X]$}
Let $(f, f_0) : (X, x) \to (X', x')$ be a morphism of based objects of $\M$.
The morphisms $\Omega_{f_0} f : \Omega_x X \to \Omega_{x'} X'$ and $f_0 :
\Dn{0} \to \Dnp{0}$ induce a diagram
\[ [\Dn{0}, \Omega_x X] \to [\Dn{0}, \Omega_{x'} X'] \leftarrow [\Dnp{0},
\Omega_{x'} X']. \]
Since $f_0$ is a weak equivalence, the right arrow is a bijection. 
Using the inverse of this bijection, we get 
a map
\[ [f, f_0] : [\Dn{0}, \Omega_x X] \to [\Dnp{0}, \Omega_{x'} X']. \]
Note that the inverse of the bijection is induced by any inverse of $f_0$ up
to homotopy (remember that $\Dn{0}$ and $\Dnp{0}$ are fibrant and
cofibrant). The morphism $[f, f_0]$ is easily seen to be functorial in $(f,
f_0)$.
\end{tparagr}

\begin{prop}\label{prop:pi_1_Omega}
Let $(X, x : \Dn{0} \to X)$ be a based object of $\M$.
Then there exists a canonical bijection
\[ \pi_1(X, x) \cong [\Dn{0}, \Omega_x X] = \pi_0(\Omega_x X), \]
natural in $(X, x)$.
\end{prop}

\begin{proof}
Let $P$ be a path object of $X$ such that $r : P \to X$ is a cofibration.
Recall that 
\[ 
\Omega_x X = \Omega^P_x X = (P, (p_1, p_0)) \times_{X\times X}
((x, x), \Dn{0}).
\]
If $Y$ is an object of $\M$, a morphism $f : Y \to \Omega_x X$ of $\M$ is
hence given by a pair of morphisms $u : Y \to P$ and $v : Y \to \Dn{0}$ of
$\M$ such that $(p_1, p_0)u = (x, x)v$. We will denote by $(u, v)$ this
morphism.

By the dual of Proposition \ref{prop:desc_graph_Pi1}, we have a canonical
bijection
\[
\pi_1(X, x) \cong [(\Dn{0}, (x, x)), (P, (p_1, p_0))]_{\MXX}.
\]
Let
\[
m : [(\Dn{0}, (x, x)), (P, (p_1, p_0))]_{\MXX} \to [\Dn{0}, \Omega_xX]
\]
be the map sending a morphism
\[ u : (\Dn{0}, (x, x)) \to (P, (p_1, p_0)) \]
to the morphism 
\[ (u, \id{\Dn{0}}) : \Dn{0} \to \Omega_x X. \]
Let us show that this map is well-defined. Let $C$ be a cylinder
object of $\Dn{0}$. Then $(C, (x, x)s)$ is a cylinder object of $(\Dn{0},
(x, x))$ in $\MXX$. Thus, a left homotopy $h : (C, (x, x)s) \to (P, (p_1,
p_0))$ of $\MXX$ from a morphism $u$ to a morphism $u'$ induces a left
homotopy $(h, s) : C \to \Omega_xX$ of $\M$ from the morphism $(u,
\id{\Dn{0}})$ to the morphism $(u', \id{\Dn{0}})$.

Let us show that the map $m$ is surjective.
Let $(u, v) : \Dn{0} \to \Omega_xX$ be a morphism of~$\M$.
We have to show that $(u, v)$ is left homotopic to $(u', \id{\Dn{0}})$ in
$\M$ for some morphism $u' : (\Dn{0}, (x, x)) \to (P, (p_1, p_0))$ of
$\MXX$. Consider the commutative square
\[
\xymatrix@C=3pc@R=2pc{
  \Dn{0} \amalg \Dn{0} \ar[d]_{(i_1, i_0)} \ar[r]^-{(\id{\Dn{0}}, v)} & \Dn{0} \ar[d] \\
  C \ar[r] & \ast \pbox{.}
}
\]
Since $\Dn{0}$ is fibrant and weakly contractible, this square admits a
lifting $h : C \to \Dn{0}$. Consider now the commutative square
\[
\xymatrix{
\Dn{0} \ar[d]_{i_0} \ar[rr]^u & & P  \ar[d]^{(p_1, p_0)} \\
C \ar[r]_{h} &  \Dn{0} \ar[r]_-{(x, x)} & X \times X \pbox{.}
}
\]
Since $\Dn{0}$ is cofibrant, $i_0$ is a trivial cofibration and the square
admits a lifting $k : C \to P$.
Set $u' = ki_1$. The morphism $u'$ induces a morphism $u' : (\Dn{0}, (x, x))
\to (P, (p_1, p_0))$ of~$\MXX$, and the morphisms $h$ and $k$ induce a
morphism $(k, h) : C \to \Omega_x X$ which is a left homotopy of $\M$ from
$(u, v)$ to $(u', \id{\Dn{0}})$.

Let us now show that $m$ is injective. Let $u, u' : (\Dn{0}, (x, x)) \to (P,
(p_1, p_0))$ be two morphisms of $\MXX$. Suppose $(k, h) : C \to \Omega_x X$
is a left homotopy of $\M$ from~$(u, \id{\Dn{0}})$ to $(u', \id{\Dn{0}})$.
We have to show that $u$ and $u'$ are left homotopic in $\MXX$. Let $C'$
be a cylinder object of $C$ in $\M$. Consider the
commutative square
\[
\xymatrix@C=3pc@R=2pc{
  C \amalg C \ar[d]_{(i'_1, i'_0)} \ar[r]^-{(s, h)} & \Dn{0} \ar[d] \\
  C' \ar[r] & \ast \pbox{.}
}
\]
This square admits a lifting $H : C' \to \Dn{0}$ for the same reasons as
above. Consider now the commutative square
\[
\xymatrix{
C \ar[d]_{i'_0} \ar[rr]^k & & P  \ar[d]^{(p_1, p_0)} \\
C' \ar[r]_{H} &  \Dn{0} \ar[r]_-{(x, x)} & X \times X \pbox{.}
}
\]
Since $\Dn{0}$ is cofibrant, $C$ is cofibrant and $i'_0 : C \to C'$
is a trivial cofibration. The square hence admits a lifting
$K : C' \to P$.  The morphism $k' = Ki'_1 : C \to P$ induces a morphism $k'
: (C, (x, x)s) \to (P, (p_1, p_0))$ which is a left homotopy of $\MXX$
from~$k'i_0$ to $k'i_1$. But we have
\[
k'i_0 = Ki'_1i_0 = Ki'_0i_0 = ki_0 = u,
\]
and a similar calculation shows that $k'i_1 = u'$.

Finally, let us show the naturality of this bijection. Let $(f, f_0) : (X,
x) \to (X, x')$ be a morphism of based objects. By the decomposition of such
a morphism given in paragraph~\ref{paragr:pi_1}, it suffices to show the
result when $f$ or $f_0$ is an identity. Suppose first that $f_0$ is the
identity.  Let $P$ (resp.~$P'$) be a path object of $X$
(resp.~of $X'$) and assume that $r : X \to P$ is a cofibration. Let $g :
P \to P'$ be a morphism making the diagram
\[
\xymatrix@C=3pc{
X \ar[d]_f \ar[r]^r & P \ar[d]_g \ar[r]^-{(p_1, p_0)} & X\times X \ar[d]^{f
\times f} \\
X'\ar[r]_{r'} & P' \ar[r]_-{(p'_1, p'_0)} & X'\times X'  \\
}
\]
commute.
Consider the naturality square
\[
\xymatrix{
[(\Dn{0}, (x, x)), (P, (p_1, p_0))]_{\MXX} \ar[r] \ar[d] & [\Dn{0},
\Omega_x X] \ar[d] \\
[(\Dn{0}, (x, x)), (P', (p'_1, p'_0))]_{\MXX} \ar[r] & [\Dn{0},
\Omega_{x'} X'] \pbox{.}\\
}
\]
The right vertical map is induced by $g$ by definition and the left
vertical map is induced by $g$ by Lemma \ref{lemma:comp_lr}. The
square is hence commutative.

Suppose now $f$ is the identity. The vertical maps in the naturality square
\[
\xymatrix{
[(\Dn{0}, (x, x)), (P, (p_1, p_0))]_{\MXX} \ar[r] \ar[d] & [\Dn{0},
\Omega_xX] \ar[d] \\
[(\Dnp{0}, (x, x)), (P, (p_1, p_0))]_{\MXX} \ar[r] & [\Dnp{0},
\Omega_xX] \\
}
\]
are both induced by an inverse of $f_0$ up to homotopy and the square is hence
commutative.
\end{proof}

\begin{tparagr}{A description of the composition of homotopies}
Let $X$ be an object of $\M$. Choose $\Dn{0}$ a cofibrant replacement of the
terminal object of $\M$ and $\Dn{1}$ a cylinder object of $\Dn{0}$.  Denote
by
\[
\xymatrix{
\Dn{0} \amalg \Dn{0} \ar[r]^-{(\Tht{1}, \Ths{1})} & \Dn{1} \ar[r]^-{\Thk{0}} & \Dn{0} 
}
\]
the associated factorization.

Let $x, x', x'' : \Dn{0} \to X$ be three morphisms of $\M$. Every element of
$\pi_1(\Dn{0}, X; x, x')$ can be represented by a homotopy $l : \Dn{1} \to
X$ from $x$ to $x'$. Similarly, an element of~$\pi_1(\Dn{0}, X; x', x'')$
can be represented by a homotopy $l' : \Dn{1} \to X$ from $x'$ to $x''$. Let
$l$ and~$l'$ be such homotopies. We will describe a homotopy $\Dn{1} \to X$
from $x$ to $x''$ representing the composition of $l$ and $l'$.

Let 
\[ 
\Dn{1} \amalg_{\Dn{0}} \Dn{1}
=
(\Dn{1}, \Ths{1}) \amalg_{\Dn{0}} (\Tht{1}, \Dn{1})
\]
be the cylinder considered in paragraph \ref{paragr:pi1_comp}.
Denote by 
\[ \ceps{1}, \ceps{2} : \Dn{1} \to \Dn{1} \amalg_{\Dn{0}} \Dn{1} \]
the two canonical morphisms and
consider the commutative square
\[
\xymatrix@C=4pc{
\Dn{0} \amalg \Dn{0} \ar[r]^-{(\ceps{1}\Tht{1}, \ceps{2}\Ths{1})} \ar[d]_{(\Tht{1},
\Ths{1})} & \Dn{1} \amalg_{\Dn{0}} \Dn{1} \ar[d] \\
\Dn{1} \ar[r] & \ast \pbox{.}
}
\]
The left vertical morphism is a cofibration by definition. Since
$\Dn{1} \amalg_{\Dn{0}} \Dn{1}$ is a cylinder object of a weakly contractible
object, it is also weakly contractible. Moreover, by assumption on the model
category $\M$, every object is fibrant. The square hence admits a lifting,
i.e., there exists a morphism
\[ \Thn{1} : \Dn{1} \to \Dn{1} \amalg_{\Dn{0}} \Dn{1} \]
such that
\[
\Thn{1}\Ths{1} = \ceps{2}\Ths{1}
\quad\text{and}\quad
\Thn{1}\Tht{1} = \ceps{1}\Tht{1}.
\]
The morphism $(l', l)\Thn{1}$ is the announced homotopy from $x$ to $x''$.
\end{tparagr}

\begin{prop}\label{prop:comp_Pi_fib}
With the notation of the above paragraph, the homotopy $(l, l')\Thn{1}$
corresponds under the bijection of Proposition \ref{prop:desc_graph_Pi1},
to the composition of $l$ and $l'$. In other words, the homotopies
\[
(l', l) : \Dn{1} \amalgd{0} \Dn{1} \to X
\quad\text{and}\quad
(l', l)\Thn{1} : \Dn{1} \to X.
\]
are left $2$-homotopic.
\end{prop}

\begin{proof}
 Let $D$ be a $2$-cylinder object of the cylinders $\Dn{1} \amalgd{0}
 \Dn{1}$ and $\Dn{1}$. Consider the commutative square
\[
\xymatrix@C=5.5pc{
\Dn{1} \amalgd{0} \big(\Dn{1} \amalgd{0} \Dn{1}\big) \ar[d]_{(j_1,j_0)}
\ar[r]^-{(\Thn{1}, \id{\Dn{1} \amalgd{0} \Dn{1}})} & \Dn{1}
\amalgd{0} \Dn{1} \ar[d] \\
D \ar[r] & \ast \pbox{.}
}
\]
The left vertical morphism is a cofibration by definition and
the object $\Dn{1} \amalg_{\Dn{0}} \Dn{1}$ is fibrant and weakly
contractible.  The square hence admits a lifting $H : D \to \Dn{1}
\amalgd{0} \Dn{1}$ and $(l', l)H$ is the desired left $2$-homotopy
from $(l', l)$ to $(l', l)\Thn{1}$,
\end{proof}

\section{Comparison of Quillen and Grothendieck homotopy groups}

\begin{paragr}
In this section, we fix a formal coherator $C$, a model category $\M$ in
which every object is fibrant, and a fibrant and weakly contractible functor
$F : \G \to \M$. By paragraph~\ref{par:def_Pi_cat_mod}, such a functor
induces a functor
\[
\Pi_{C,K} : \M \to \wgpdC{C},
\]
which depends on a globular functor $K : C \to \M$. From now on,
we will denote this functor by $\Pi_\infty$ and, if $X$ is an object
of~$\M$, we will call $\Pi_\infty(X)$ the fundamental \oo-groupoid of $X$.
Recall that the set of $n$-arrows of this fundamental \oo-groupoid
is the set of morphisms from $\Dn{n}$ to $X$ in $\M$.

We recall some notation from paragraph \ref{par:ex_cof_wc}.  For $n \ge 0$,
we have an object $\Sn{n-1}$ of~$\M$ and a morphism $i_n : \Sn{n-1} \to
\Dn{n}$. For $n = 0$, the morphism $i_0$ is the unique morphism $\varnothing
\to \Dn{0}$ and, for $n \ge 1$, we have
\[ i_n = (\Tht{n}, \Ths{n}) : \Sn{n-1} = \Dn{n-1} \amalg_{\Sn{n-2}} \Dn{n-1} \to \Dn{n}. \]
The hypothesis on $F$ exactly means that $i_n$ is a cofibration and that the
$\Dn{n}$ is weakly contractible.

For $n \ge 0$, we will denote by $j_{n+1} : \Sn{n-1} \to \Dn{n+1}$ the
composition of the canonical morphism $\Sn{n-1} \to \Sn{n} = \Dn{n}
\amalg_{\Sn{n-1}} \Dn{n}$ followed by $i_{n+1}$. Note that we have
\[
 j_{n+1} = (\Tht{n+1}\Tht{n}, \Ths{n+1}\Ths{n}) : \Sn{n-1} = \Dn{n-1}
\amalg_{\Sn{n-2}} \Dn{n-1} \to \Dn{n+1}.
\]

We will also fix a choice of morphisms
\[ \Thk{n} : \Dn{n+1} \to \Dn{n}, \quad n \ge 0, \]
of $C$ such that
\[
\Thk{n}\Ths{n+1} = \id{\Dn{n}}
\quad\text{and}\quad
\Thk{n}\Tht{n+1} = \id{\Dn{n}}.
\]
These morphisms are weak equivalences since the $\Dn{n}$'s are weakly
contractible. Moreover, they induce a choice of units $\Glk{n} : G_n \to
G_{n-1}$ for every \oo-groupoid $G$ of type $C$.
\end{paragr}

\begin{lemma}\label{lemma:cyl}
\ 
\begin{enumerate}
\item The object $\Dn{0}$ is a cofibrant replacement of the terminal object
of $\M$.
\item For every $n \ge 0$, the object $(\Dn{n+1}, j_{n+1})$ is a cylinder
object of $(\Dn{n}, i_n)$ in the category $\Sn{n-1}\bs\M$
for the factorization
\[
\xymatrix{
(\Sn{n}, l_n) \ar[r]^-{i_{n+1}} & (\Dn{n+1}, j_{n+1}) \ar[r]^-{\Thk{n}} &
(\Dn{n}, i_n) \pbox{,}
}
\]
where $l_n : \Sn{n-1} \to \Sn{n} = \Dn{n} \amalg_{\Sn{n-1}} \Dn{n}$ is the
canonical morphism.
\end{enumerate}
\end{lemma}

\begin{proof}
The first assertion is true by definition. Let us prove the second one.
The equalities \[
i_{n+1}l_n = j_{n+1}, \quad
\Thk{n}j_{n+1} =
\Thk{n}(\Tht{n+1}\Tht{n}, \Ths{n+1}\Ths{n}) = (\Tht{n}, \Ths{n}) = i_n,
\]
and
\[
\Thk{n}i_{n+1} = \Thk{n}(\Tht{n+1}, \Ths{n+1}) = (\id{\Dn{n}}, \id{\Dn{n}}),
\]
show that we indeed have a factorization of the codiagonal of $(\Dn{n}, i_n)$ in
$\Sn{n-1}\bs\M$. Moreover, $i_{n+1}$ is a cofibration and we have already
noticed that $\Thk{n}$ is a weak equivalence.
\end{proof}

\begin{prop}\label{prop:hom_hom}
Let $X$ be an object of $\M$ and let $G$ be its fundamental \oo-groupoid.
\begin{myenumeratep}
\item Two objects $x, y : \Dn{0} \to X$ of $G$ are homotopic as objects of
$G$ if and only if $x, y : \Dn{0} \to X$ are equal as morphisms of
$\Ho(\M)$. In particular, we have a canonical bijection
\[
\pi_0(\Pi_\infty(X)) \to \pi_0(X),
\]
natural in $X$.
\item Let $n \ge 1$ and let $x, y : \Dn{n-1} \to X$ be two parallel
$(n-1)$-arrows of $G$. Two $n$-arrows $u, v : \Dn{n} \to X$ from $x$ to
$y$ are homotopic as $n$-arrows of $G$ if and only if $u, v : (\Dn{n}, i_n) \to
(X, (y, x))$ are equal as morphisms of $\Ho(\Sn{n-1}\bs\M)$. In particular,
we have a canonical bijection
\[
\pi_n(\Pi_\infty(X), x, y) \to [(\Dn{n}, i_n), (X, (y, x))]_{\Sn{n-1}\bs\M},
\]
natural in $X$.
\end{myenumeratep}
\end{prop}

\begin{proof}
\ 
\begin{myenumeratep}
\item 
By definition, two objects $x, y : \Dn{0} \to X$ of $G$ are homotopic
if there exists a $1$\nbd-arrow $h : x \to y$ in $G$, i.e.,
a morphism $h : \Dn{1} \to X$ of $\M$ such that $h\Ths{1} = x$ and~$h\Tht{1}
= y$. By the previous lemma, $\Dn{1}$ is a cylinder object of
$\Dn{0}$ and so $h$ is a left homotopy from $x$ to $y$ in $\M$. The result then
follows from the fact that $\Dn{0}$ is cofibrant and $X$ is fibrant.
\item 
By definition, the $n$-arrows $u, v : \Dn{n} \to X$ are homotopic if there
exists an \hbox{\text{$(n+1)$-arrow}} $h : u \to v$ in $G$, i.e., a
morphism $h : \Dn{n+1} \to X$ of $\M$ such that $h\Ths{n+1} = u$ and
$h\Tht{n+1} = v$. Such an $h$ induces a morphism $h : (\Dn{n+1}, j_{n+1})
\to (X, (y, x))$. But by the previous lemma, $(\Dn{n+1}, j_{n+1})$ is a
cylinder object of $(\Dn{n}, i_n)$ in $\Sn{n-1}\bs\M$ and so $h$ is a left
homotopy between $u, v : (\Dn{n}, i_n) \to (X, (y, x))$ in $\Sn{n-1}\bs\M$.
The result then follows from the fact that $i_n$ is a cofibration and $X$ is
fibrant.
\end{myenumeratep}
\end{proof}

\begin{prop}\label{prop:desc_var_pi1}
Let $X$ be an object of $\M$.
There is a canonical isomorphism of groupoids
\[ \varpi_1(\Pi_\infty(X)) \cong \Pi_1(\Dn{0}, X), \]
natural in $X$.
\end{prop}

\begin{proof}
Let $x, y : \Dn{0} \to X$ be two objects of $\varpi_1(\Pi_\infty(X))$. By 
the previous proposition, we have
\[
\Hom_{\varpi_1(\Pi_\infty(X))}(x, y) \cong [(\Dn{1}, i_1), (X, (y, x))]_{\Sn{0}\bs\M}.
\]
Hence by Proposition \ref{prop:desc_graph_Pi1}, the underlying graphs
of $\varpi_1(\Pi_\infty(X))$ and of $\Pi_1(\Dn{0}, X)$ are canonically
isomorphic. This isomorphism is obviously natural in $X$.  Moreover, by
definition, the composition of $\varpi_1(\Pi_\infty(X))$ is induced by any
morphism
\[ \Thn{1} : \Dn{1} \to \Dn{1} \amalg_{\Dn{0}} \Dn{1} \]
such that
\[
\Thn{1}\Ths{1} = \ceps{2}\Ths{1}
\quad\text{and}\quad
\Thn{1}\Tht{1} = \ceps{1}\Tht{1}.
\]
The result thus follows from Proposition \ref{prop:comp_Pi_fib}.
\end{proof}

\begin{prop}\label{prop:pi_n_pi_1}
Let $X$ be an object of $\M$. Let $n \ge 2$ and
let $u : x \to y$ be an $(n-1)$-arrow of the fundamental \oo-groupoid
of $X$. We have a canonical isomorphism of groups
\[ \pi_n(\Pi_\infty(X), u) \cong \pi_1((X, (y, x)), u)_{\Sn{n-2}\backslash
\M}, \]
natural in $X$.
\end{prop}

\begin{proof}
We have the following series of natural bijections of sets:
\begin{align*}
\pi_n(\Pi_\infty(X), u) 
& \cong [(\Dn{n}, i_n), (X, (u, u))]_{\Sn{n-1}\bs\M}
&& \qquad\text{(by Proposition \ref{prop:hom_hom})} \\
& \cong \pi_1((\Dn{n-1}, i_{n-1}), (X, (y, x)); u, u)_{\Sn{n-2}\bs\M}
&& \qquad\text{(see below)}\\
& = \pi_1((X, (y, x)), u)_{\Sn{n-2}\bs\M}. \\
\end{align*}
Let us justify the second bijection. By Lemma~\ref{lemma:cyl}, the object $(\Dn{n},
j_n)$ is a cylinder object of~$(\Dn{n-1}, i_{n-1})$ in $\Sn{n-2}\bs\M$. This
bijection thus comes from Proposition \ref{prop:desc_graph_Pi1} applied to
the model category $\Sn{n-2}\bs\M$. 

Furthermore, the composition of $\pi_n(\Pi_\infty(X), u)$ is induced by any morphism
\[ \Thn{n} : \Dn{n} \to \Dn{n} \amalgd{n-1} \Dn{n} \]
such that
\[
\Thn{n}\Ths{n} = \ceps{2}\Ths{n}
\quad\text{and}\quad
\Thn{n}\Tht{n} = \ceps{1}\Tht{n}.
\]
Denote by $k_n : \Sn{n-2} \to \Dn{n} \amalgd{n-1} \Dn{n}$ the composition of
$i_{n-1} : \Sn{n-2} \to \Dn{n-1}$ followed by the canonical morphism
$\Dn{n-1} \to \Dn{n} \amalgd{n-1} \Dn{n}$.
We have
\[
 (\Dn{n} \amalgd{n-1} \Dn{n}, k_n) = (\Dn{n}, j_n) \amalg_{(\Dn{n-1},
 i_{n-1})} (\Dn{n}, j_n),
\]
where the amalgamated sum is taken in $\Sn{n-2}\bs\M$. Moreover, a morphism
$\Thn{n}$ as above induces a morphism
$\Thn{n} : (\Dn{n}, j_n) \to (\Dn{n} \amalgd{n-1} \Dn{n}, k_n)$.
It thus follows from Proposition~\ref{prop:comp_Pi_fib} that the bijection
we have defined is a morphism of groups.
\end{proof}

\begin{prop}
Let $X$ be an object of $\M$ and let $x$ be an object of the
fundamental \oo-groupoid of $X$. For every $n \ge 2$,
there is a canonical isomorphism of groups
\[
 \pi_n(\Pi_\infty(X), x) \cong \pi_{n-1}(\Pi_\infty(\Omega_x X), c_x),
\]
natural in $X$.
\end{prop}

\begin{proof}
We have the following series of natural isomorphisms: 
\begin{align*}
  \pi_n(\Pi_\infty(X), x) 
  & = \pi_n(\Pi_\infty(X), \Glk[n-1]{0}(x)) \\
  & \cong \pi_1((X, (\Glt{n-1}\Glk[n-1]{0}(x),\Gls{n-1}\Glk[n-1]{0}(x))),
  \Glk[n-1]{0}(x))_{\Sn{n-2}\backslash \M} \\
  & \qquad\text{(by the previous proposition)} \\
  & = \pi_1((X, (\Glk[n-2]{0}(x),\Glk[n-2]{0}(x))),
  \Glk[n-1]{0}(x))_{\Sn{n-2}\backslash \M} \\
  & \cong [(\Dn{n-1}, i_{n-1}), \Omega_{\Glk[n-1]{0}(x)} (X, (\Glk[n-2]{0}(x),
  \Glk[n-2]{0}(x)))]_{\Sn{n-2}\backslash \M} \\
  & \qquad\text{(by Proposition \ref{prop:pi_1_Omega})} \\
  & \cong [(\Dn{n-1}, i_{n-1}), (\Omega_{\Glk[n-1]{0}(x)} X,
  (c_{\Glk[n-2]{0}(x)}, c_{\Glk[n-2]{0}(x)}))]_{\Sn{n-2}\backslash \M} \\
  & \cong [(\Dn{n-1}, i_{n-1}), (\Omega_x X,
  (c_{\Glk[n-2]{0}(x)}, c_{\Glk[n-2]{0}(x)}))]_{\Sn{n-2}\backslash \M} \\
  & \qquad\text{(see below)} \\
  & \cong \pi_{n-1}(\Pi_\infty(\Omega_x X), c_{\Glk[n-2]{0}(x)}) \\
  & \qquad\text{(by Proposition \ref{prop:hom_hom})} \\
  & = \pi_{n-1}(\Pi_\infty(\Omega_x X), \Glk[n-2]{0}(c_x)) \\
  & = \pi_{n-1}(\Pi_\infty(\Omega_x X), c_x).
\end{align*}
To end the proof, it suffices to show that we have a canonical weak
equivalence
\[
(\Omega_{\Glk[n-1]{0}(x)} X, (c_{\Glk[n-2]{0}(x)}, c_{\Glk[n-2]{0}(x)}))
\to
(\Omega_x X, (c_{\Glk[n-2]{0}(x)}, c_{\Glk[n-2]{0}(x)}))
\]
in $\Sn{n-2}\backslash \M$. Consider the commutative diagram
\[
\xymatrix@C=7pc{
\Dn{n-1} \ar[d]_-{\Thk[n-1]{0}} \ar[r]^-{(\Glk[n-1]{0}(x), \Glk[n-1]{0}(x))} &
X \times X \ar[d]_{\id{X\times X}} & \ar[l]_-{(p_1, p_0)} P
\ar[d]^{\id{P}}\\
\Dn{0} \ar[r]_-{(x, x)} & X \times X & \ar[l]^-{(p_1, p_0)} P \pbox{,}
}
\]
where $P$ is a path object of $X$ in $\M$. This diagram induces a morphism
\[ \Omega_{\Glk[n-1]{0}(x)} X \to \Omega_x X \] 
of $\M$. Since $\Thk[n-1]{0}$ is a weak equivalence and the pullbacks
defining loop objects are homotopy pullbacks, this morphism is a weak
equivalence. Moreover, it induces a morphism
\[
(\Omega_{\Glk[n-1]{0}(x)} X, (c_{\Glk[n-2]{0}(x)}, c_{\Glk[n-2]{0}(x)}))
\to
(\Omega_x X, (c_{\Glk[n-2]{0}(x)}, c_{\Glk[n-2]{0}(x)}))
\]
in $\Sn{n-2}\backslash \M$. This is our desired weak equivalence.
\end{proof}

\begin{thm}
Let $X$ be an object of $\M$.
\begin{itemize}
\item There is a canonical bijection
  \[ \pi_0(\Pi_\infty(X)) \cong \pi_0(X), \]
natural in $X$.
\item Let $n \ge 1$ and let $x$ be an object of the fundamental \oo-groupoid
of $X$. There is a canonical isomorphism of groups
\[ \pi_n(\Pi_\infty(X), x) \cong \pi_n(X, x),\]
natural in $X$.
\end{itemize}
\end{thm}

\begin{proof}
Let us prove the result by induction on $n \ge 0$. For $n = 0$ and $n = 1$,
the result is a direct consequence of Proposition \ref{prop:desc_var_pi1}.
For $n \ge 2$, we have
\begin{align*}
  \pi_n(\Pi_\infty(X), x)
  & \cong \pi_{n-1}(\Pi_\infty(\Omega_x X), c_x) 
  && \text{(by the previous proposition)} \\
  & \cong \pi_{n-1}(\Omega_x X, c_x)
  && \text{(by induction hypothesis)} \\
  & = \pi_n(X, x),
\end{align*}
thereby proving the theorem.
\end{proof}

\begin{coro}
If $X$ is an object of $\M$, the homotopy groups of $\Pi_\infty(X)$ depend
only on~$X$. In particular, they do not depend on the choice of the formal
coherator $C$.
\end{coro}

\begin{coro}\label{coro:charac_we}
Let $f : X \to Y$ be a morphism of $\M$. Then $\Pi_\infty(f)$ is a weak
equivalence of \oo-groupoids of type $C$ if and only if the following
conditions are satisfied:
\begin{itemize}
\item the map $\pi_0(f) : \pi_0(X) \to \pi_0(Y)$ is a bijection;
\item for every $n \ge 1$ and every base point $x : \Dn{0} \to X$, the
morphism $\pi_n(f, x) : \pi_n(X, x) \to \pi_n(Y, f(x))$ is an isomorphism.
\end{itemize}
\end{coro}

\begin{proof}
This follows immediately from the naturality of the isomorphisms of the
above theorem.
\end{proof}

\begin{coro}
Let $f$ be a weak equivalence of $\M$. Then $\Pi_\infty(f)$ is a
weak equivalence of \oo-groupoids of type $C$.
\end{coro}

\begin{proof}
This follows immediately from the previous corollary and the fact that weak
equivalences of $\M$ induce isomorphisms on homotopy groups.
\end{proof}

\begin{coro}
Let $\M$ be $\Top$ endowed with its usual model category structure. A
map~$f$ is a weak equivalence of topological spaces if and only if
$\Pi_\infty(f)$ is a weak equivalence of \oo-groupoids of type $C$.
\end{coro}

\begin{proof}
This follows immediately from Corollary \ref{coro:charac_we}.
\end{proof}

\begin{rem}
This corollary is proved directly in Section 4.4 of \cite{AraThesis} (see in
particular Corollary 4.4.11).
\end{rem}

\begin{tparagr}{The functor $\overline{\Pi_\infty}$}
Let $\W_\Top$ be the class of weak equivalences of topological spaces and
let $\W_{\wgpd{}_C}$ be the class of weak equivalences of \oo-groupoids of
type $C$. By the previous corollary, the functor $\Pi_\infty$ sends
$\W_\Top$ into $\W_{\wgpd{}_C}$. This functor thus induces a functor
\[ \overline{\Pi_\infty} : \Hot = \Top[\W_\Top^{-1}] \to \Ho(\wgpdC{C}) =
\wgpdC{C}[{\W^{-1}_{\wgpd{}_C}}]. \]
\end{tparagr}

We can now state a precise version of Grothendieck's conjecture:
\begin{conj}[Grothendieck]\label{conj:Groth}
If $C$ is a coherator, the functor $\overline{\Pi_\infty}$ is an equivalence
of categories.
\end{conj}

\bibliographystyle{amsplain}
\bibliography{biblio}

\end{document}